\documentclass[11pt]{article}
\usepackage[english]{babel}
\usepackage[a4paper,left=2cm,right=2cm,top=2.5cm,bottom=2.5cm]{geometry}
\usepackage{graphicx}
\usepackage[pdfborder={0 0 0}]{hyperref}
\usepackage[utf8]{inputenc}

\usepackage{authblk}
\usepackage{mathtools}
\usepackage{amssymb}
\usepackage{dsfont}
\usepackage{amsmath}
\usepackage{url}
\usepackage{float}
\usepackage{comment}
\usepackage[T1]{fontenc}
\usepackage{bm}
\usepackage{graphicx}
\usepackage[normalem]{ulem}
\usepackage{commath}
\usepackage{amsthm}
\usepackage{mathrsfs}  
\usepackage{amsmath,amscd}
\usepackage{hyperref}
\usepackage{bm}
\usepackage{array}
\usepackage[nottoc,notlot,notlof]{tocbibind}
\usepackage[title]{appendix}
\usepackage{tikz}
\usepackage[title]{appendix}
\usepackage{xcolor}
\usepackage{leftidx}

\usetikzlibrary{fadings}

\usepackage{pgfplots}

\usepackage{cleveref}

\usepackage{tikz-cd}
\usetikzlibrary{matrix,arrows}
\usepackage{subfiles}

\renewcommand{\subset}{\subseteq}

\newcommand{\End}{\mathrm{End}}
\newcommand{\Der}{\mathrm{Der}}
\newcommand{\rk}{\mathrm{rk}}

\newcommand{\inv}{{}^{-1}}
\newcommand{\dd}{\textup{d}}
\newcommand{\lhk}{\left(}
\newcommand{\rhk}{\right)}

\newcommand{\mc}{\mathcal}
\newcommand{\mf}{\mathfrak}
\newcommand{\mb}{\mathbb}

\newcommand{\pairing}[2]{\langle#1,#2\rangle}
\newcommand{\cb}[2]{[\![#1,#2]\!]}

\newcommand{\nrt}[1]{\textcolor{black}{#1}}

\makeatletter
\newcommand{\opnorm}{\@ifstar\@opnorms\@opnorm}

\makeatother

\pgfplotsset{compat=1.10}
\usepgfplotslibrary{fillbetween}
\usetikzlibrary{patterns}

\newcommand{\Hom}{\mathrm{Hom}}

\newcommand{\tweedrie}[6]{\lhk \begin{matrix} #1 & #2\\#3 & #4 \\#5&#6\end{matrix} \rhk}

\newcommand{\kolomtwee}[2]{\left ( \begin{matrix} #1\cr #2 \end{matrix} \right)}
\newcommand{\kolomdrie}[3]{\left ( \begin{matrix} #1\cr #2\cr #3 \end{matrix} \right)}

\newcommand{\Addresses}{{
\bigskip
\footnotesize

\textsc{KU Leuven, Department of Mathematics, Celestijnenlaan 200B box 2400, 3001 Leuven, Belgium} \par \nopagebreak

\bigskip
\footnotesize

Revised at:\\ \textsc{Max Planck Institute for Mathematics, Vivatsgasse 7, 53111 Bonn, Germany} \par \nopagebreak

\bigskip
Current address:\\ \textsc{University of W\"urzburg, Emil-Fischer-Strasse 31, 97074 W\"urzburg, Germany}
\textit{E-mail address}:\href{mailto:ksingh@mpim-bonn.mpg.de}{ksingh@mpim-bonn.mpg.de}
}}

\theoremstyle{plain}
\newtheorem{thm}{Theorem}[section]
\newtheorem*{thm*}{Theorem}
\newtheorem{lem}[thm]{Lemma}
\newtheorem{prop}[thm]{Proposition}
\newtheorem{cor}[thm]{Corollary}

\theoremstyle{definition}
\newtheorem{defn}[thm]{Definition} 
\newtheorem{exmp}[thm]{Example}
\newtheorem{rmk}[thm]{Remark}

\newtheorem{conv}[thm]{Convention}
\newtheorem*{conv*}{Convention}
\newtheorem*{prob*}{Problem}
\newtheorem*{claim}{Claim}

\newtheorem{assumptions}[thm]{Assumptions}

\allowdisplaybreaks
\title{Stability of fixed points in Poisson geometry and higher Lie theory}
\date{ \null}
\author{Karandeep J. Singh}

\begin{document} 
\maketitle

\abstract{We provide a uniform approach to obtain sufficient criteria for a (higher order) fixed point of a given bracket structure on a manifold to be stable under perturbations. Examples of bracket structures include Lie algebroids, Lie $n$-algebroids, singular foliations, Lie bialgebroids, Courant algebroids and Dirac structures in split Courant algebroids admitting a Dirac complement. We in particular recover stability results of Crainic-Fernandes for zero-dimensional leaves, as well as the stability results of higher order singularities of Dufour-Wade.

These stability problems can all be shown to be specific instances of the following problem: given a differential graded Lie algebra $\mf g$, a differential graded Lie subalgebra $\mf h$ of finite codimension in $\mf g$ and a Maurer-Cartan element $Q\in \mf h^1$, when are Maurer-Cartan elements near $Q$ in $\mf g$ gauge equivalent to elements of $\mf h^1$? 

We show that the vanishing of a finite-dimensional cohomology group associated to $\mf g,\mf h$ and $Q$ implies a positive answer to the question above, and therefore implies stability of fixed points of the geometric structures described above.
}

\tableofcontents

\section{Introduction}
In differential geometry, there are various structures of infinitesimal nature on a manifold $M$ which induce a partition of $M$ into immersed submanifolds called leaves. Examples include Lie algebra actions, involutive distributions and Poisson structures. These are all examples of Lie algebroids.
A Lie algebroid over a manifold $M$ is a vector bundle $A\to M$, equipped with a bundle map $$\rho : A\to TM$$ covering the identity on $M$, together with a Lie bracket on the space of sections of $A$ such that for every section $x,y\in\Gamma(A)$ and $f\in C^\infty(M)$, we have the equality
$$
[x,fy]_A = \rho(x)(f) y + f[x,y]_A.
$$
This property together with the Jacobi identity for $[-,-]_A$ implies that $\rho:\Gamma(A) \to \mf X(M)$ is a Lie algebra map, therefore the image of $\rho$ defines a singular foliation on $M$. As the space of Lie algebroid structures on a  vector bundle $A$ carries a topology, one can ask when a leaf $L$ of a given Lie algebroid structure $(\rho,[-,-]_A)$ is stable when perturbing the Lie algebroid structure. More precisely:

\emph{Given a leaf $L \subset M$, when is it the case that every Lie algebroid structure near $(\rho,[-,-]_A)$ has a leaf near $L$ which is diffeomorphic to $L$}?

For compact leaves $L$, this question was first answered in \cite{CrFe}. Here it was shown that if the first cohomology of the Lie algebroid restricted to $L$ with values in a certain representation vanishes, then for every Lie algebroid structure near the original one, there exists a \emph{family of leaves} diffeomorphic to $L$. Moreover, when $A = T^\ast M$, the cotangent bundle of $M$, the authors prove a stronger result for when the same conclusion holds when restricting the class of Lie algebroid structures to only those coming from Poisson structures, as well as a separate result which guarantees the existence of a family of leaves \emph{symplectomorphic} to a given one.\\
The results of \cite{CrFe} only depend on the first order approximation of $(\rho,[-,-]_A)$ around the leaf $L$. Let $L = \{p\}\subset M$ be a fixed point, that is, a zero-dimensional leaf, and assume that $\rho$ has a higher order of vanishing in $p$. Then the cohomological assumptions of \cite{CrFe} will not be satisfied, hence nothing can be concluded about the stability of $p$. In this case the result was extended to give a criterion for stability of higher order fixed points of Lie algebroids and Poisson structures in \cite{dufour2005stability}, which guarantees the existence of a family of \emph{higher order} fixed points near $p$.
\subsection*{Main results}
In this article we focus on the case of fixed points. In Section \ref{sec:directproof} we first recall the stability result of \cite{CrFe} and \cite{dufour2005stability} for first order fixed points of Lie algebroids, and prove it directly in terms of Lie algebroid data, instead of passing through the identification with fiberwise linear multivector fields on the dual vector bundle. In Section \ref{sec:dictionary} we show that the stability question for zero-dimensional leaves is equivalent to a question about the comparison of Maurer-Cartan elements of a differential graded Lie algebra and Maurer-Cartan elements of a chosen differential graded Lie subalgebra as explained below\nrt{, highlighting the deformation theoretic character of the problem following Deligne's principle \cite{delignedgla}}. Using this reformulation as motivation, we state and prove the main theorem, which gives a sufficient condition for the inclusion of the differential graded Lie subalgebra to be locally surjective on equivalence classes of Maurer-Cartan elements (Theorem \ref{thm:mainthm}). The rest of the article is dedicated to applying \nrt{the theorem} in various situations by making appropriate choices for the differential graded Lie algebras involved. In Section \ref{sec:applications} we apply the main theorem to obtain:
\begin{itemize}
    \item[-] the general result of \cite{dufour2005stability} for Lie algebroids (Theorem \ref{thm:liealgdk}),
    \item[-] a stability result for (higher order) fixed points of Lie $n$-algebroids (Theorems \ref{thm:lienalg1} and \ref{thm:lienalgk}), with an application to singular foliations (Section \ref{sec:singfolapplic}).
\end{itemize} In Section \ref{sec:additionalstr} we apply the main Theorem to obtain stability results for fixed points within a subclass of structures, such as:
\begin{itemize}
    \item[-] (higher order) fixed points of Lie algebroid structures on a vector bundle $A^\ast$, which form a Lie bialgebroid pair with a \emph{given} Lie algebroid structure on $A$ (Theorem \ref{thm:liebialg}), 
    \item[-] (higher order) fixed points of Poisson structures compatible with a fixed Nijenhuis tensor (Theorem \ref{thm:PN}),
    \item[-] (higher order) fixed points of Courant algebroids (Theorem \ref{thm:courant}),
    \item[-] fixed points of Dirac structures in split Courant algebroids admitting a Dirac complement (Theorem \ref{thm:Dirac}).
\end{itemize}
The results are all of the following form: given a structure $Q$ as above, and a (higher order) fixed point $p\in M$ of this structure, there will be a finite-dimensional cohomology group associated to this structure and the (higher order) fixed point. If this cohomology group vanishes, then for every neighborhood $V$ of $p\in M$, there exists a $C^s$-neighborhood $\mc U$ of the structure $Q$ such that every $Q'\in \mc U$ has a fixed point of the same kind in $V$. The precise value of $s\in \mb Z_{\geq 0}$ will depend on the type of structure and the order of the fixed point.

We approach these questions using the main theorem as follows. Let $p \in M$. 
\begin{itemize}
    \item[-] The relevant structures are identified with Maurer-Cartan elements of a certain differential graded Lie algebra $(\mf g,\partial,[-,-])$ given by the sections of some vector bundle.
    \item[-]  In here, we identify a differential graded Lie subalgebra $\mf h \subset \mf g$, for which the Maurer-Cartan elements are those Maurer-Cartan elements of $\mf g$ for which some \emph{given} $p\in M$ is a fixed point of desired type. Let $Q\in \mf h^1$ denote such a structure.
    \item[-]  $\mf g^0$ acts on $\mf g^1$ by the gauge action, as can be found in \cite{manettidefcomp} for instance. In our examples, this action is by vector bundle automorphisms $\Phi$ covering a diffeomorphism $\phi$ of $M$. In particular, $p\in M$ is a fixed point of a Maurer-Cartan element $Q'$ gauge transformed by $X\in \mf g^0$ if and only if $\phi(p)\in M$ is a fixed point of the untransformed $Q'$.
\end{itemize} 
The question of stability of $p$ can now roughly be formulated as follows:

\emph{Given a Maurer-Cartan element $Q$ of $\mf h$, when is it true that any Maurer-Cartan element of $\mf g$ near $Q$ is gauge equivalent to a Maurer-Cartan element of $\mf h$?}

 Under some conditions on $\mf g$, $\mf h$ and the gauge action, the most restrictive one being that $\mf g^i/\mf h^i$ is finite-dimensional for $i = 0,1,2$, we show that a sufficient condition for a positive answer is that $$H^1(\mf g/\mf h,\overline{\partial + [Q,-]}) = 0.$$ Here $\overline{\partial+[Q,-]}$ is the induced map on the quotients. In several examples, this recovers known cohomology groups, and in all examples the chain complexes will consist of finite-dimensional vector spaces.\\
Where possible, we also point out relations between various structures and the cohomology groups that arise from them.\\
In a future work, we plan to explore the extent to which these results can be generalised so that they may be applied to higher-dimensional leaves.
\paragraph{\bf Acknowledgements} 
We thank Marius Crainic and Marco Zambon for fruitful discussions and helpful comments. We thank Camille Laurent-Gengoux and Sylvain Lavau for sharing their preprint and providing an appropriate framework to apply the results to singular foliations. We acknowledge the FWO and FNRS under EOS projects G0H4518N and G0I2222N. We also thank Utrecht University and the Max Planck Institute for Mathematics for their hospitality and financial support. 

\paragraph{\bf Conflicts of interest} 
We state that there is no conflict of interest.

\section{A direct proof of stability of fixed points of Lie algebroids}\label{sec:directproof}
In this section we reprove the stability result of (first order) fixed points of Lie algebroids using the approach of \cite{dufour2005stability}, writing it directly in Lie algebroid data instead of passing through the identification with linear Poisson structures on the dual vector bundle. This proof already contains the key arguments to prove the main theorem of this article. We first give the definition of a Lie algebroid.
\begin{defn}\label{def: Classical Lie algebroid}
Let $M$ be a smooth manifold. A \textit{Lie algebroid over $M$} is a triple $(A,\rho,[-,-])$, where
\begin{itemize}
    \item[i)] $A\to M$ is a vector bundle, 
    \item[ii)] $\rho:A \to TM$ is a vector bundle map \nrt{covering $\text{id}_M:M\to M$}, 
    \item[iii)] $[-,-]:\Gamma(A) \times \Gamma(A) \to \Gamma(A)$ is an $\mb R$-bilinear skew-symmetric map,
\end{itemize}
such that
\begin{itemize}
    \item[a)] for all $x,y \in \Gamma(A)$, $f\in C^\infty(M)$, we have
    $$
    [x,fy] = \rho(x)(f)y + f[x,y].
    $$
    \item[b)] for all $x,y,z \in \Gamma(A)$, 
    $$
    [x,[y,z]] + [y,[z,x]] + [z,[x,y]] = 0.
    $$
\end{itemize}
\end{defn}
\begin{lem}\label{lem:anchorbrackpres}
Property a) and b) imply that
$$
\rho([x,y]) = [\rho(x),\rho(y)]
$$
for all $x,y\in \Gamma(A)$.
\end{lem}
Recall that a Lie algebroid induces a partition of $M$ in connected immersed submanifolds called \emph{leaves}, such that the tangent space to the leaf through a point $p$ coincides with the image of $\rho$ at $p$.
\begin{defn}\label{def:singpoint}
Let $M$ be a smooth manifold and let $(A,\rho,[-,-])$ be a Lie algebroid \nrt{over $M$}. A point $p\in M$ is a \textit{fixed point of $(\rho,[-,-])$} if $\rho_p = 0$. 
\end{defn}
Note that fixed points are exactly zero-dimensional leaves. \\
Above any point $p\in M$, the Lie algebra structure of $\Gamma(A)$ induces the structure of a Lie algebra on $\ker(\rho_p)$ called the \emph{isotropy Lie algebra}.
\begin{lem} \label{lem:isolie}
Let $(A,\rho,[-,-])$ be a Lie algebroid over $M$, and let $p \in M$. Let\\
$x,y \in \ker(\rho_p:A_p \to T_p M)$, and let $\tilde{x},\tilde{y} \in \Gamma(A)$ be extensions of $x,y$ respectively. Then the element
$$
\mu_p(x,y):=[\tilde{x},\tilde{y}](p)
$$
is well-defined and lies in $\ker(\rho_p)$. Moreover, the map
$$
\mu_p:\ker(\rho_p) \times \ker(\rho_p) \to \ker(\rho_p)
$$
satisfies the Jacobi identity, equipping $\ker(\rho_p)$ with a Lie algebra structure.
\end{lem}
For the rest of this section, assume that $p \in M$ is a fixed point of the Lie algebroid $(A,\rho,[-,-])$. Denote by $\mf g_p$ the isotropy Lie algebra at $p$. As a vector space it is just $A_p$, while the Lie bracket is given as in Lemma \ref{lem:isolie}.\\
The isotropy Lie algebra has a natural representation $\tau:\mf g_p \to \End(T_pM)$ on $T_pM$, called the Bott representation or the linear holonomy representation. For $x\in \mf g_p, v\in T_pM$, it is defined by
\begin{equation}\label{eq:linholrep}
\tau(x)(v) = [\rho(\tilde{x}),\tilde{v}](p) \in T_pM,
\end{equation}
where $\tilde{x}\in \Gamma(A), \tilde{v} \in \mf X(M)$ are extensions of $x$ and $v$ respectively.\\
The cohomology of the isotropy Lie algebra $\mf g_p$ with values in the linear holonomy representation $\tau$ on $T_pM$ plays a vital role in Theorem \ref{thm:liealgd1} and will return several times throughout the article, so we recall the definition of Lie algebra cohomology.
\begin{defn}\label{def:cecomplex}
Let $\mf g$ be a Lie algebra, and let $\sigma:\mf g\to \End(V)$ be a representation on a vector space $V$. 
\begin{itemize}
    \item [i)] The \textit{Chevalley-Eilenberg complex} is the cochain complex $$(S(\mf g^\ast[-1])\otimes V,d^\sigma_{CE}),$$ where for $\alpha \in S^k(\mf g^\ast[-1])\otimes V$, $v_0,\dots, v_k \in \mf g$, we have
\begin{align*}
d^\sigma_{CE}(\alpha)(v_0,\dots,v_k) =& \sum_{i=0}^k (-1)^{i+k} \sigma(v_i)(\alpha(v_0,\dots, \widehat{v_i},\dots, v_k)) \\&+ \sum_{0\leq i<j\leq k} (-1)^{i+j+k}\alpha([v_i,v_j],v_0,\dots, \widehat{v_i},\dots,\widehat{v_j},\dots, v_k).
\end{align*}
\item[ii)] When $V = \mb R$ and $\sigma = 0$, we denote by $$(S(\mf g^\ast[-1]), d_{CE})$$ the corresponding complex.
\end{itemize}
\end{defn}
The following lemma is easy to show:
\begin{lem}
$$
(S(\mf g^\ast[-1]),d_{CE})
$$
is a differential graded commutative algebra, and for any representation $\sigma:\mf g\to \End(V)$, the Chevalley-Eilenberg complex $$(S(\mf g^\ast[-1])\otimes V, d_{CE}^\sigma)$$ defines a differential graded module over this algebra.
\end{lem}
For our purposes, the relevant part of the complex is in degrees 0, 1 and 2. In these degrees, we denote the complex by
\[
\begin{tikzcd}
V \arrow{r}{d_0} & \mf g^\ast[-1] \otimes V  \arrow{r}{d_1} &S^2(\mf g^\ast[-1]) \otimes V,
\end{tikzcd}
\]
with 
\begin{equation}\label{eq:diff1}
    d_0(v)(x) = \sigma(x)(v),
\end{equation}
and
\begin{equation}\label{eq:diff2}
d_1(\alpha)(x,y) = -\sigma(x)(\alpha(y)) + \sigma(y)(\alpha(x)) + \alpha([x,y])
\end{equation}
for $v \in V, \alpha \in \mf g^\ast[-1] \otimes V, x,y \in \mf g[1]$. 

Denote the corresponding cohomology groups by $H^i_{CE}(\mf g,V)$ for $i = 0,1$. 
\begin{rmk}\label{rmk:cecohom}
\begin{itemize}
\item[]
\item[i)]This definition differs from the one which is mostly used by an overall factor of $(-1)^k$ when $\alpha \in S^k(\mf g^\ast[-1])\otimes V$. This is because we choose to write the complex using the shifted symmetric product rather than the unshifted wedge product. 
\item[ii)] Note that for $H^i_{CE}(\mf g,V)$ to be defined for $i = 0,1$, it is not necessary that the bracket of $\mf g$ satisfies the Jacobi identity. It is only needed that there is an $\mb R$-bilinear map $[-,-]:\mf g\times \mf g \to \mf g$, and that 
$$
\sigma([x,y]) = \sigma(x)\sigma(y) - \sigma(y)\sigma(x)
$$
for $x,y \in \mf g$.
\end{itemize}
\end{rmk}
The last thing we need to formulate the stability result for fixed points of Lie algebroids is a topology on the space of Lie algebroid structures. It will be shown in Section \ref{sec:exampleLiealg} that Lie algebroid structures on $A$ are contained in the sections of some vector bundle $E$, hence can be equipped with the weak $C^k$-topology induced by the one on $\Gamma(E)$. An element $Q \in \Gamma(E)$ can be seen as a pair $(\rho,[-,-])$, where $\rho$ and $[-,-]$ are as in Definition \ref{def: Classical Lie algebroid}, without requirement b). This is the description in terms of multiderivations as in \cite{CrMo08}.
\begin{thm}[\cite{CrFe},\cite{dufour2005stability}]\label{thm:liealgd1}
Let $M$ be a manifold and $(A,\rho,[-,-])$ be a Lie algebroid. Let $p \in M$ be a fixed point of $(A,\rho,[-,-])$. Let $\mf g_p$ be the isotropy Lie algebra at $p$. Assume that
$$H^1_{CE}(\mf g_p,T_pM) =0.$$ 
Then for any open neighborhood $U\subset M$ of $p$, there is a $C^1$-neighborhood $\mc U$ of $(\rho,[-,-])$ in the space of Lie algebroid structures such that for any $(\rho',[-,-]') \in \mc U$ there exists a family $I\subset U$ of fixed points of $(\rho',[-,-]')$ parametrized by an open neighborhood of the origin of $H^0_{CE}(\mf g_p,T_pM)$.
\end{thm}
\begin{proof}
The setup of the proof will be similar to the proof of Theorem 1.2 of \cite{dufour2005stability}. The only difference is that we work with the description of Lie algebroids in terms of an anchor and a bracket, while the authors of \cite{dufour2005stability} work with a special class of Poisson structures on $A^\ast$. \\
As the statement is local in $M$, we may assume that $M = V$ is a finite-dimensional real vector space, $p = 0 \in V$, and that $A = \mf g_p \times V$ is a trivial bundle.\\
The outline of the proof is as follows. 
\begin{itemize}
    \item[i)] We construct a smooth map $\text{ev}_{Q}:V\to W = \mf g_p^\ast[-1]\otimes V$, parametrized continuously by elements \nrt{$Q$ of $\Gamma(E)$} with the $C^1$-topology, which contains the Lie algebroid structures on $A$.
    \item[ii)] We construct a smooth map $R_{q,Q}: \mf g_p^\ast[-1] \otimes V = W \to T = S^2(\mf g_p^\ast[-1])\otimes V$, parametrized continuously by $(q,Q)\in V\times \Gamma(E)$, where the second factor has the $C^1$-topology.
\end{itemize}
These maps will have the following properties: denoting by $Q_0 := (\rho,[-,-])$,
\begin{itemize}
    \item [a)] $\text{ev}_{Q_0}(p) = 0$, and $(D(\text{ev}_{Q_0}))_p= -d_0:V \to W$, where $d_0$ is defined in equation \eqref{eq:diff1} for the Bott representation $\tau$. Moreover, if $Q$ is a Lie algebroid structure, $\text{ev}_Q(q) = 0$ if and only if $q$ is a fixed point of $Q$.
    \item[b)] $R_{q,Q}(0) = 0$ for every $(q,Q)\in V \times \Gamma(E)$, and $(D(R_{p,Q_0}))_0 = d_1$, with $d_1$ as in equation \eqref{eq:diff2} for the Bott representation $\tau$.
    \item[c)] Whenever $Q \in \Gamma(E)$ is a Lie algebroid structure and $q\in V$, we have
    $$
    R_{q,Q}(\text{ev}_{Q}(q)) = 0\in T.
    $$
\end{itemize}

The existence of the maps with these properties is sufficient to prove the theorem. Before we construct the maps, we show how the \nrt{result} follows from these properties. Let $C$ be a complement to $\ker(d_1) = \text{im}(d_0)$ in $W = \mf g_p^\ast[-1]\otimes V$.\\
First\nrt{,} property b) implies that $R_{p,Q_0}$ restricted to $C$ is an immersion at $0\in C$ as $C$ has trivial intersection with $\ker(d_1 = (D(R_{p,Q_0}))_0)$. By Lemma \ref{lem:immersion}, there exists an open neighborhood $O$ of $0 \in C$, an open neighborhood $S$ of $p \in V$ and a $C^1$-neighborhood $\mc U_2$ of $Q_0\in \Gamma(E)$ such that $\left. R_{q,Q}\right|_{O}$ is an injective immersion for $Q\in \mc U_2$ and $q \in S$.\\
Property a) implies that ev$_{Q_0}$ intersects $O\subset C$ transversely in $p$, as $d_0 = -(D(\text{ev}_{Q_0}))_p$, and $\text{im}(d_0) = \ker(d_1)$ by the cohomological assumption. Therefore by Lemma \ref{lem:transverse}, there exists a $C^1$-neighborhood of $Q_0 \in \mc U_1 \subset \Gamma(E)$ such that for any $Q \in \mc U_1$, there exists a $q\in S$ such that $\text{ev}_Q(q) \in O$.\\
By property c), for any Lie algebroid structure $Q \in \mc U = \mc U_1 \cap \mc U_2$, and any $q\in V$ we have
$$
R_{q,Q}(\text{ev}_Q(q)) = 0.
$$
By injectivity of $R_{q,Q}$ restricted to $O$, combined with the fact that ev$_{Q}(q) \in O$, it follows that
$$
\text{ev}_Q(q) = 0,
$$
or equivalently, $q$ is a fixed point of $Q$. For the existence of the family of fixed points, apply the final statement of Lemma \ref{lem:transverse}, again using that $R_{q,Q}$ is injective restricted to $O$.

Now we define the maps. 
\begin{itemize}
\item[i)]For $Q = (\sigma,[-,-]')$, let 
$$
\text{ev}_Q:V\to W
$$
be defined by $$\text{ev}_Q(v) = T^\ast_{v}(\sigma)_p\in W,$$ where $T_v:V \to V$ is the translation by $v \in V$. The continuous dependence on $Q \in \Gamma(E)$ holds by definition of the topology on $\Gamma(E)$.
\item[ii)]
Now let $q\in V$, and $Q = (\sigma,[-,-]') \in \Gamma(E)$. Define
$$
R_{q,Q}:W\to T
$$
by
$$
R_{q,Q}(\alpha)(x,y) = -T^\ast_q([\sigma(x),\widetilde{\alpha}(\widetilde{y})] - [\sigma(y),\widetilde{\alpha}(\widetilde{x})] - \alpha([\widetilde{x},\widetilde{y}]'))_p,
$$
where $\alpha \in W$, $x,y \in \mf g_p$, and the $\tilde{\cdot}$ indicates the extension of $\cdot$ to a constant section.  We postpone the proof of the continuous dependence on $(q,Q)\in V\times \Gamma(E)$ to Section \ref{sec:exampleLiealg}, as it will be shown more generally. 
\item[a)]
It is then clear that $\text{ev}_{Q_0}(p) = 0$, and that $q$ is a fixed point of $Q$ if and only if $\text{ev}_Q(q) = 0$. For the statement about the derivative of $\text{ev}_{Q_0}$ note that the translation map $T_v:V\to V$ is the time-1 flow of the \emph{constant} vector field induced by $v\in V$, so for $v\in V \cong T_p V, x \in A_p$ viewing both as constant sections, we have
$$
(D(\text{ev}_{Q_0}))_p(v)(x) = \left. \frac{d}{dt}\right|_{t=0} T^\ast_{tv}(\rho)_p(x) = [v,\rho(x)](p) = -d_0(v)(x).
$$
Here in the second to last equality we note that $T^\ast_{-tv}(\rho)_p(x)$ is the pushforward of the vector field $\rho(x)$ by the diffeomorphism $T_{tv}:V\to V$, which is the time $t$ flow of the constant vector field $v$.
\item[b)] The properties regarding its value and derivative are immediate.
\item[c)] We note that 
\begin{align*}
    R_{q,Q}(\text{ev}_{Q}(q)) = - T^\ast_q([\sigma(x),\sigma(y)] -\sigma([x,y]'))_p = 0.
\end{align*}
\end{itemize}
\end{proof}
\begin{rmk}\label{rmk:liealg1}
\begin{itemize}
    \item []
    \item [-]Observe that for the proof it was only necessary that $\sigma:\Gamma(A) \to \mf X(M)$ preserves the brackets, and the full Jacobi identity for $[-,-]'$ was not needed. So the theorem actually yields a stability criterion for \emph{almost} Lie algebroid structures on $A$ (see \cite{univlinfty}, Definition 3.7). For almost Lie algebroids, the requirement that the Jacobiator vanishes identically is replaced by the requirement that it is $C^\infty(M)$-multilinear. In this case, the fiber of $A$ over a singular point $p\in M$ carries a bracket (which does not necessarily satisfy the Jacobi identity), and the action on $T_pM$ still makes sense, so the cohomology $H^1_{CE}(\mf g,T_pM)$ is well-defined according to Remark \ref{rmk:cecohom}ii).
    \item[-] If $0\in W$ is a regular value for $\text{ev}_{Q_0}$, the map $R_{\cdot,\cdot}$ \nrt{is not} needed. However, for dimensional reasons this can only happen if the Lie algebroid $A$ has rank 1. When the rank of $A$ is 1, $0$ being a regular value of $\text{ev}_{Q_0}$ is equivalent to the cohomological assumption.
    \item[-] The map $R_{q,Q}$ is linear. It is therefore unnecessary to use Lemma \ref{lem:immersion}. The linearity is a consequence of the fact that constant vector fields commute, and in Theorem \ref{thm:mainthm} there will be a quadratic term. As the arguments are very similar we choose to give the general argument here, and refer back to this in the proof of the main theorem.
\end{itemize}
\end{rmk}
\section{The main theorem}\label{sec:mainthm}
In this section, we state and prove a generalization of Theorem \ref{thm:liealgd1}, which is an algebraic statement about differential graded Lie algebras (Theorem \ref{thm:mainthm}). We give some basic background on differential graded Lie algebras in Section \ref{sec:backgrounddgla}. To motivate the generalization, we characterize Lie algebroid structures in terms of certain elements in a graded Lie algebra, and show how the problem of stability of singular points can be reformulated in terms of this graded Lie algebra (section \ref{sec:dictionary}). In Section \ref{ssec:mainthm}, we then state the assumptions and prove Theorem \ref{thm:mainthm}.
\subsection{Differential graded Lie algebras}\label{sec:backgrounddgla}
\begin{defn}
A \textit{differential graded Lie algebra} is a triple $(\mf g,\partial,[-,-])$, where 
\begin{itemize}
    \item [i)] $\mf g$ is a non-negatively graded real vector space,
    \item [ii)] $\partial:\mf g\to \mf g$ is a linear map of degree 1,
    \item[iii)] $[-,-]:\mf g\times \mf g \to \mf g$ is a degree 0 graded skew-symmetric bilinear map,
\end{itemize}
satisfying
\begin{itemize}
    \item[a)] $\partial\circ\partial = 0$,
    \item[b)] $\partial([x,y]) = [\partial(x),y] + (-1)^{|x|} [x,\partial(y)]$ for all $x,y \in \mf g$ where $|x|$ denotes the degree of $x$,
    \item[c)] $[[x,y],z] = [x,[y,z]] - (-1)^{|x||y|}[y,[x,z]]$ for $x,y,z \in \mf g$, where $|x|$ and $|y|$ denote the degree of $x$ and $y$ respectively.
\end{itemize}
\end{defn}
Any element $Q\in \mf g^1$ can be used to twist the differential $\partial$, by considering the map $$\partial + [Q,-].$$ The resulting triple $(\mf g,\partial +[Q,-], [-,-])$ still satisfies property b) and c). It will in general not satisfy a), as
$$
(\partial+[Q,-])\circ (\partial + [Q,-]) = \left[\partial(Q) + \frac{1}{2}[Q,Q],-\right].
$$
This leads us to the following definition.
\begin{defn}
Let $(\mf g, \partial,[-,-])$ be a differential graded Lie algebra, and $Q \in\mf g^1$ an element of degree 1. $Q$ is a \emph{Maurer-Cartan} element if 
$$
\partial(Q) + \frac{1}{2}[Q,Q] = 0 \in \mf g^2.
$$
Denote the set of all Maurer-Cartan elements of $\mf g$ by $MC(\mf g)$.
\end{defn}
\begin{conv}
Although the examples we consider may have nonzero negative degrees, we will implicitly set the negative degrees equal to zero. Note that this is harmless: the bracket of two elements of non-negative degrees has non-negative degree, and the differential respects this. We may use the bracket on negative degrees to define certain subspaces however.
\end{conv}
We start with some examples.
\begin{exmp}
\begin{itemize}
    \item []
    \item [1)] We can consider two extreme cases: if $\partial = 0$, a differential graded Lie algebra $(\mf g,0,[-,-])$ is just a graded Lie algebra. If $[-,-] =0$, a differential graded Lie algebra $(\mf g,\partial,0)$ is just a cochain complex. In the former case, a Maurer-Cartan element is simply a degree 1 element $Q\in \mf g^1$ such that $[Q,Q] = 0$. In the latter, Maurer-Cartan elements are precisely 1-cocycles.
    \item[2)] Let $M$ be a smooth manifold. Then $(\Gamma(S( TM[-1])[1],0,[-,-]_{SN})$, where $[-,-]_{SN}$ is the Schouten-Nijenhuis bracket, is the graded Lie algebra of multivector fields on $M$. Its Maurer-Cartan elements are $\pi \in \Gamma(S^2(TM[-1]))$ such that $[\pi,\pi]_{SN} = 0$. These are precisely the Poisson bivectors.
    \item[3)] Let $M$ be a smooth manifold, and $A$ a vector bundle over $M$. The space of multiderivations \nrt{of $A$} has a graded Lie algebra structure as described in Proposition 1 in \cite{CrMo08}, and its Maurer-Cartan elements are precisely the Lie algebroid structures on $A$. \\
    An equivalent description is also given in \cite{CrMo08}, in which the bracket is more intuitive, which can be generalized to graded vector bundles. Borrowing notation from graded geometry, of which the basics can be found in \cite{CATTANEO_2011}, any vector bundle $A\to M$ gives rise to a graded manifold $A[1]$. Its functions are given by the non-negatively graded commutative algebra
    $$
    C^\infty(A[1]) := \Gamma(S(A^\ast[-1])).
    $$
    Now there is a natural graded Lie algebra associated to $A[1]$: it is the graded Lie algebra of graded derivations of $C^\infty(A[1])$, which we denote by  
    $$
    \mf X(A[1]) := \Der_{\mb R}(C^\infty(A[1])).
    $$
    The graded commutator bracket
    $$
    [X,Y] = X\circ Y - (-1)^{|X||Y|} Y \circ X
    $$
    equips $\mf X(A[1])$ with a graded Lie algebra structure, where $X,Y\in \mf X(A[1])$ of degrees $|X|,|Y|$ respectively. It was first observed by A. Vaintrob \cite{Vaintrob} that a degree 1 derivation $Q$ on $C^\infty(A[1])$ satisfying $[Q,Q] = 2Q^2 = 0$ is equivalent to the data of a Lie algebroid.
\end{itemize}
\end{exmp}
We will encounter more examples \nrt{of} differential graded Lie algebras in this text. \\
Another construction on differential graded Lie algebras we will need is the gauge action of $\mf g^0$ on $\mf g^1$. For nilpotent graded Lie algebras, the \nrt{action} can be written down as an infinite sum which terminates, see for instance the formula in \cite{manettidefcomp} above Remark V.33. As the differential graded Lie algebras we will encounter will come with a notion of differentiable paths, we take Remark V.33 in \cite{manettidefcomp} as a definition.
\begin{defn}\label{def:gaugeeq}
Let $Q\in \mf g^1$, and $v\in \mf g^0$. Consider the initial value problem in $\mf g^1$ given by
\begin{equation}\label{eq:gaugeeq}
\frac{d}{dt}\gamma_t = [v,\gamma_t]-\partial(v), \gamma_0 = Q.
\end{equation}
Assume that there exists a unique solution for $t\in [0,1]$. Then the \emph{gauge action} of $v$ on $Q$ is defined to be $\gamma_1$, and will be denoted by $Q^v$.
\end{defn}
This action satisfies a property similar to the exponential map for Lie groups.
\begin{lem}
For $t \in [0,1]$, we have $\gamma_t = Q^{tv}$. That is, the time $t$ solution of the initial value problem associated to $v$ is equal to the time 1 solution of the initial value problem associated to $tv$.
\end{lem}
One of the main properties of the gauge action in \cite{manettidefcomp} is that it should preserve the Maurer-Cartan elements of $\mf g$. While for the nilpotent case this can be proven, we will assume this. 
\subsection{Lie algebroid stability in terms of a graded Lie algebra}\label{sec:dictionary}
In this section we will reformulate the problem of stability of fixed points of Lie algebroid structures on $A\to M$, as well as the solution provided by Theorem \ref{thm:liealgd1} completely in terms of operations on the graded Lie algebra $\mf X(A[1])$, as motivation for Theorem \ref{thm:mainthm}. We do this in four steps:
\begin{itemize}
    \item[i)] Identify Lie algebroid structures on $A$ with Maurer-Cartan elements of $\mf X(A[1])$.
    \item[ii)] Identify a graded Lie subalgebra $\mf X_{p,1}(A[1])\subset \mf X(A[1])$ whose Maurer-Cartan elements are Lie algebroid structures for which a \emph{given} $p\in M$ is a fixed point.
    \item[iii)] Use the gauge action of $\mf X^0(A[1])$ to identify a neighborhood of $\mf X^0(A[1])/\mf X^0_{p,1}(A[1])$ with a neighborhood of $p\in M$, so that $q\in M$ near $p$ will be a fixed point of a Lie algebroid structure $Q$ if and only if $Q$ is gauge equivalent to an element in $\mf X^1_{p,1}(A[1])$.
    \item[iv)] Interpret the cohomology appearing in Theorem \ref{thm:liealgd1} in terms of $\mf X(A[1])$, $\mf X_{p,1}(A[1])$ and $Q\in \mf X_{p,1}^1(A[1])$, which is the Lie algebroid structure we start with.
\end{itemize}
We carry out the steps described above.
\begin{itemize} 
\item[i)]
We start by recalling the bijection between Lie algebroid structures on a vector bundle $A$ in terms of an anchor and a bracket as in Definition \ref{def: Classical Lie algebroid} and Lie algebroid structures defined using degree 1 vector fields on the graded manifold $A[1]$, due to T. Voronov \cite{qmanvoronov}.
\begin{lem}
Let $M$ be a manifold and $(A,\rho,[-,-]_A)$ a Lie algebroid. Then the Lie algebroid differential
\begin{align}\label{eq:liealgdr}
Q(\alpha)(X_0,\dots,X_k) =& \sum_{i=0}^k (-1)^{i+k} \rho(X_i)(\alpha(X_0,\dots, \widehat{X_i},\dots, X_k)) \\\nonumber&+ \sum_{0\leq i<j\leq k} (-1)^{i+j+k}\alpha([X_i,X_j]_A,X_0,\dots, \widehat{X_i},\dots,\widehat{X_j},\dots, X_k)
\end{align}
for $\alpha \in \Gamma(S^k(A^\ast[-1])), X_0,\dots ,X_k \in \Gamma(A)$ defines a degree 1 derivation of the graded algebra $\Gamma(S(A^\ast[-1]))$ satisfying $[Q,Q] = 2Q^2 = 0.$\\
Conversely, let $Q $ be a degree 1 derivation of $\Gamma(S(A^\ast[-1]))$ satisfying $[Q,Q]=0$. Identifying $\iota:\Gamma(A) \to \mf X^{-1}(A[1])$ using the contraction map,
the structure maps
$$
\rho(X)(f) = -[Q,\iota_X](f),
$$
$$
[X,Y]_A = \iota^{-1}([[\iota_Y,Q],\iota_X])
$$
for $f\in C^\infty(M), X,Y \in \Gamma(A)$
equip $A$ with a Lie algebroid structure.
\end{lem}
From now on, a Lie algebroid $(A,\rho,[-,-]_A)$ will be denoted by $(A,Q)$ where $Q$ is defined by equation \eqref{eq:liealgdr}. 
\item[ii)] We now identify a graded Lie subalgebra of $\mf X(A[1])$, whose Maurer-Cartan elements are precisely those Lie algebroid structures for which a \emph{given} $p\in M$ is a fixed point. For that we start with an observation:
\begin{lem}
Let $(A,Q)$ be a Lie algebroid over $M$. Then $p\in M$ is a fixed point of $Q$ if and only if \begin{equation}\label{eq:singpointcond}
    Q(C^\infty(M)) \subset I_p\Gamma(A^\ast[-1]),
\end{equation}where $I_p$ denotes the vanishing ideal of $p \in M$.
\end{lem}
\begin{proof}
For any $X\in \Gamma(A)$ and $f\in C^\infty(M)$ we have the equality
$$
\rho(X)(f) = -\iota_X(Q(f)).
$$
If $Q(f) \in I_p\Gamma(A^\ast[-1])$, we get that $\rho(X)(f) \in I_p$. As $X$ and $f$ are arbitrary, this implies that $\rho_p = 0 \in A_p^\ast\otimes T_pM$, so $p\in M$ is a fixed point of $Q$. \\
Conversely, if $\rho_p = 0 \in A_p^\ast \otimes T_pM$, then for every $a\in A_p$, $Q(f)_p(a) = 0$. As $a \in A_p$ and $f$ are arbitrary, $Q(f)_p = 0$ hence $Q(f) \in I_p \Gamma(A^\ast[-1])$.
\end{proof}
Note that condition (\ref{eq:singpointcond})
defines a linear subspace of $\mf X^1(A[1])$. This condition can be extended to other degrees, which then defines a graded Lie subalgebra of $(\mf X(A[1]),0,[-,-])$:
\begin{lem}\label{lem:liesubalg} For $p \in M$, $k= 0,\dots, \rk(A)$, define
$$
\mf X^k_{p,1}(A[1]) := \{Q \in \mf X^k(A[1]) \mid Q(C^\infty(M))\subset I_p \Gamma(S^k(A^\ast[-1]))\}.
$$
Then $(\mf X_{p,1}(A[1]), 0, [-,-])$ is a graded Lie subalgebra of $(\mf X(A[1]),0,[-,-])$.
\end{lem}
\begin{proof}
We need to show that
$$
[\mf X^k_{p,1}(A[1]),\mf X^l_{p,1}(A[1])] \subset \mf X^{k+l}_{p,1}(A[1]),
$$
 which is a straightforward computation.  
\end{proof}
\begin{rmk}
The subscript 1 indicates that it is the subspace of vector fields which have prescribed vanishing behavior up to first order in $p$. We will also encounter higher order vanishing conditions later.
\end{rmk}
\item[iii)]
The next question is now how to keep track of a Lie algebroid structure having a fixed point at $q\neq p \in M$. This is where the gauge action of $\mf X(A[1])$ comes into play. We unpack the definition and construct the solution of the gauge action.\\
Recall that in the proof of Theorem \ref{thm:liealgd1}, we used the translation map $T_v$, which was the time-1 flow of the constant vector field $v \in \mf X(V) \cong C^\infty(V,V)$. We now describe this in a coordinate free way. First\nrt{,} we describe $\mf X^0(A[1])$. The following lemma is well-known, for a proof we refer to the appendix of \cite{ZAMBON2013155}. For details on covariant differential operators on $A$, also known as derivations of $A$, we refer to \cite[Section 3.4]{genthylgpdlalgd}.
\begin{lem} \label{lem:cdodeg0}
$$\mf X^0(A[1]) \cong CDO(A^\ast) \cong CDO(A).$$
\end{lem}

As covariant differential operators are infinitesimal vector bundle automorphisms, they can be integrated to vector bundle automorphisms covering a diffeomorphism of the base $M$. \\
Let $D\in CDO(A)$, and extend it to $S(A^\ast[-1])$. Assume that its symbol $\sigma(D) \in \mf X(M)$ is complete and let $$(\tilde{\Phi}^D_{-t})^\ast:\Gamma(S(A^\ast[-1]))\to \Gamma(S(A^\ast[-1]))$$ be the automorphism associated to it. That is, $$\frac{d}{dt}(\tilde{\Phi}^D_{-t})^\ast = D^\ast \circ (\tilde{\Phi}^D_{-t})^\ast, $$ where $D^\ast:\Gamma(S(A^\ast[-1])) \to \Gamma(S(A^\ast[-1]))$ is dual to $D$. As $\mf X(A[1])$ is the set of derivations of the algebra $C^\infty(A[1])=\Gamma(S(A^\ast[-1]))$, any automorphism of $C^\infty(A[1])$ induces an automorphism of $\mf X(A[1])$ by conjugation: for $X \in \mf X(A[1])$, it is defined by
$$
(\tilde{\Phi}^D_{-t})_\ast(X) = (\tilde{\Phi}^D_{t})^{\ast} \circ X \circ (\tilde{\Phi}^D_{-t})^\ast.
$$
The automorphism $(\tilde{\Phi}_{-t}^{D})_\ast$ interacts nicely with the subalgebra $\mf X_{p,1}(A[1])$\nrt{, moving $p$ by the flow of the symbol of $D$}:
\begin{lem}\label{lem:singpointmove}
Let $Q$ be a Lie algebroid structure, and $p\in M$. \nrt{Let $D\in CDO(A)$, with symbol $\sigma(D) =: X$.}  The automorphism $(\tilde{\Phi}_{-1}^D)_\ast$ satisfies 
$$
(\tilde{\Phi}_{-1}^D)_\ast(Q) \in \mf X_{p,1}^1(A[1]) \iff Q\in \mf X ^1_{\phi^X_{1}(p),1}(A[1]).
$$
\end{lem}
\begin{proof}
Let $f \in C^\infty(M)$. Then 
\begin{align}\label{eq:pushforwardgradedgeom}
(\tilde{\Phi}^D_{-1})_\ast(Q)(f) &= (\tilde{\Phi}^D_{1})^\ast(Q(f\circ \phi_{-1}^X)).
\end{align}
Now for the implication $\Rightarrow$, we assume that the left hand side lies in $I_p\Gamma(A^\ast)$. It follows that $Q(f\circ \phi_{-1}^X)$ lies in $$(\tilde{\Phi}_{-1}^D)^{\ast}(I_p\Gamma(A^\ast)) = I_{\phi_1^X(p)} \Gamma(A^\ast)$$
As $(\phi^X_{-1})^\ast$ is surjective, the implication follows.\\
Conversely, for the implication $\Leftarrow$, it follows immediately that the right hand side of \eqref{eq:pushforwardgradedgeom} lies in $I_p\Gamma(A^\ast)$, proving the lemma.
\end{proof}
It is not a coincidence that the notation for the automorphism of $\mf X(A[1])$ associated to an element $D \in CDO(A) \cong \mf X^0(A[1])$ resembles the notation for the pushforward of a vector field along a diffeomorphism. The following lemma shows that this automorphism is precisely the gauge transformation by $D^\ast$: 
\begin{lem}\label{lem:solgaugeeq} Let $Q \in \mf X(A[1])$, and $D \in CDO(A)$ with symbol $X$. Then whenever $\phi_t^X$ is defined, we have the equality 
$$
\frac{d}{dt} (\tilde{\Phi}_{-t}^D)_{\ast}(Q) = [D^\ast,(\tilde{\Phi}_{-t}^D)_{\ast}(Q)]
$$
with $(\tilde{\Phi}_{0}^D)_{\ast}(Q) = Q$.
\end{lem}
We can now rephrase the question of stability in a way which only involves operations intrinsic to the graded Lie algebra $\mf X(A[1])$. 
For simplicity, we write $ \mf X(A[1])=:\mf g$, $ \mf X_{p,1}(A[1])=:\mf h$, and for $Q \in \mf g$, $D \in \mf g^0$, we write $$Q^D:=(\tilde{\Phi}_{-1}^D)_{\ast}(Q).$$ The stability problem can now roughly be formulated as:

\emph{Let $(A,Q)$ be a Lie algebroid over $M$, and $p\in M$ a fixed point. When is it the case that for any Lie algebroid structure $Q'$ near $Q$, there exists a $D^\ast \in \mf g^0$ such that the  solution $\gamma:[0,1] \to \mf X^1(A[1])$ of the initial value problem
$$
\frac{d}{dt}\gamma_t = [D^\ast,\gamma_t], \gamma_0 = Q'
$$
satisfies $\gamma_1 \in \mf h^1$?}

Note that if $D \in \mf h^0\subset \mf g^0$, then $Q^D \in \mf h$ if and only if $Q\in \mf h$. We may therefore restrict our search for such a $D^\ast \in \mf g^0$ to a \emph{complement} of $\mf h^0$ in $\mf g^0$. Such a complement is naturally isomorphic to $T_pM$. For our purposes however, it will be more convenient to work with the quotient $\mf g^0/\mf h^0$, with a chosen $\mb R$-linear splitting $\Sigma:\mf g^0/\mf h^0 \cong T_pM \to \mf g^0$. As we are interested in small neighborhoods of the point $p$ and will only look at the action by elements of the image of this splitting, the requirement that the symbol of a differential operator is a complete vector field is not restrictive. Indeed, on a coordinate chart where $A$ trivializes, it is easy to see that the constant extension of an element of $T_pM$ defines a splitting.

    \item[iv)]
Now that we have phrased the problem in this context, there is also a way to phrase the answer provided by Theorem \ref{thm:liealgd1} in this context. In fact, the cochain complex associated to the Bott representation arises naturally:
\begin{lem}\label{lem:isocohom}
Let $(A,Q)$ be a Lie algebroid over $M$, and $p\in M$ a fixed point, with isotropy Lie algebra $\mf g_p$. Let $(\mf X(A[1]),[Q,-],[-,-])$ and $(\mf X_{p,1}(A[1]),[Q,-],[-,-])$ the associated differential graded Lie algebras. For $k = 0,\dots, \rk(A) = r$, there is a natural isomorphism
$$
\mf X^k(A[1])/\mf X^k_{p,1}(A[1]) \to S^k(\mf g_p^\ast[-1])\otimes T_pM
$$
which intertwines the differential $\overline{[Q,-]}$ induced on the quotient complex \\$(\mf X(A[1])/\mf X_{p,1}(A[1]),\overline{[Q,-]})$ with the Chevalley-Eilenberg differential $d_{CE}^\tau$ on the right hand side.
\end{lem}
\begin{proof}
There is a short exact sequence of graded $C^\infty(M)$-modules in which $\mf X(A[1])$ sits:
\[
\begin{tikzcd}
0\arrow{r} & C^\infty(A[1]) \otimes \Gamma(A[1]) \arrow{r}{\iota} & \mf X(A[1]) \arrow{r}{\sigma}& C^\infty(A[1]) \otimes \mf X(M) \arrow{r}& 0,
\end{tikzcd}
\]
where $\iota$ denotes the contraction and $\sigma$ the restriction to $C^\infty(M)\subset C^\infty(A[1])$.\\
Note that this shows that $\mf X(A[1]) = \Gamma(E)$ for some graded vector bundle $E$, as any connection on $A$ splits the sequence. Now the graded Lie algebra $\mf X_{p,1}(A[1])$ also sits inside a short exact sequence of graded $C^\infty(M)$-modules:
\[
\begin{tikzcd}
0\arrow{r} & C^\infty(A[1]) \otimes \Gamma(A[1]) \arrow{r}{\iota} & \mf X_{p,1}(A[1]) \arrow{r}{\sigma}& I_pC^\infty(A[1]) \otimes \mf X(M) \arrow{r}& 0
\end{tikzcd}
\]
Consequently, on the quotients we get a short exact sequence of vector spaces
\[
\begin{tikzcd}
0 \arrow{r}& 0 \arrow{r}& \mf X(A[1])/\mf X_{p,1}(A[1]) \arrow{r}{\overline{\sigma}} & S(\mf g_p^\ast[-1])\otimes T_pM\arrow{r}& 0,
\end{tikzcd}
\]
proving the isomorphism. 

As both $\mf X(A[1])/\mf X_{p,1}$ and $S(\mf g_p^\ast[-1])\otimes T_pM$ are a left dg-module over $(S(\mf g_p^\ast [-1]),d_{CE})$, it suffices to check compatibility of the differentials in degree 0, which can be done by inspection.

\end{proof}
For future reference, we state the short exact sequence as a lemma.
\begin{lem}\label{lem:sesder}
There is a short exact sequence of graded $C^\infty(M)$-modules
\[
\begin{tikzcd}\label{diag:sesder}
0\arrow{r} & C^\infty(A[1]) \otimes \Gamma(A[1]) \arrow{r}{\iota} & \mf X(A[1]) \arrow{r}{\sigma}& C^\infty(A[1]) \otimes \mf X(M) \arrow{r}& 0.
\end{tikzcd}
\]
Moreover, any connection on $A$ induces a splitting of the sequence.
\end{lem}
\end{itemize}
This section can now be summarized to give an alternative formulation of Theorem \ref{thm:liealgd1}, which involves only operations on the graded Lie algebra. Let $\mf g = \mf X(A[1]), \mf h = \mf X_{p,1}(A[1])$  as above, and let $\Sigma:\mf g^0/\mf h^0 \to \mf g^0$ denote a splitting of the quotient map. 
\begin{thm}[Reformulation of Theorem \ref{thm:liealgd1}]\label{thm:liealgd1dgla} Let $Q\in \mf h$ be a Maurer-Cartan element of $\mf h$ (hence of $\mf g$). If $$H^1(\mf g/\mf h,\overline{[Q,-]}) = 0,$$ then for every open neighborhood $V$ of $0 \in \mf g^0/\mf h^0\cong T_p M$ there exists a $C^1$-open neighborhood $\mc U$ of $Q$ in the space of Maurer-Cartan elements of $(\mf g,0,[-,-])$ such that for any $Q'\in \mc U$, there exists a family $I\subset V$, parametrized by an open neighborhood of the origin of $H^0(\mf g/\mf h,\overline{[Q,-]})$, with $(Q')^{\Sigma(v)}\in \mf h^1$ for all $v\in I$. 
\end{thm}
Under some assumptions on $\mf g$ and $\mf h$, this theorem will hold in more generality. Making appropriate choices for $\mf g$ and $\mf h$ will then yield similar results in other contexts.

\subsection{The main theorem: assumptions and proof} \label{ssec:mainthm}
In this section we state some assumptions on a pair consisting of a differential graded Lie algebra $\mf g$ and a differential graded Lie subalgebra $\mf h$ so that Theorem \ref{thm:liealgd1dgla} holds in more generality, and prove this general result.
\begin{assumptions}\label{ass:mainthmass} Assume we have the following:
\begin{itemize}
    \item [i)] A differential graded Lie algebra $(\mf g,\partial,[-,-])$,
    \item[ii)] a differential graded Lie subalgebra $(\mf h,\partial,[-,-])$ such that $\mf g^i/\mf h^i$ is finite-dimensional for $i = 0,1,2$,
    \item[iii)] splittings $\sigma_i:\mf g^i/\mf h^i \to \mf g^i$ for $i = 0,1$,
    \item[iv)] a Maurer-Cartan element $Q \in \mf h^1\subset \mf g^1$,
\end{itemize}
such that
\begin{itemize}
    \item [a)] $\mf g^i$ for $i = 0,1,2$ are locally convex topological vector spaces such that the projections $p_i:\mf g^i\to \mf g^i/\mf h^i$ are continuous,
    \item[b)] $\partial:\mf g^1 \to \mf g^2$ is continuous,
    \item[c)] $[-,-]:\mf g^1\times \mf g^1\to \mf g^2$ is continuous,
    \item[d)] There is an open neighborhood $U$ of $0 \in \mf g^0/\mf h^0$ such that for any $Q'\in \mf g^1$, the gauge action as in Definition \ref{def:gaugeeq} of $\sigma_0(v)$ for $v\in U$ on $Q'$ is defined, the assignment
    $$
    U\times \mf g^1 \ni (v,Q') \mapsto (Q')^{\sigma_0(v)}\in \mf g^1
    $$
    is jointly continuous, and its class mod $\mf h^1$ depends smoothly on $v\in U$ for each fixed $Q'$.
    \item[e)] For $v\in U$, $Q'\in \mf g^1$ is Maurer-Cartan if and only if $(Q')^{\sigma_0(v)}$ is Maurer-Cartan.
\end{itemize}
\end{assumptions}
\begin{rmk}
The choice of $Q$ implies that $(\mf g, \partial + [Q,-])$ is a cochain complex, with $(\mf h,\partial+[Q,-])$ as a subcomplex. We can therefore take the quotient complex which we denote by $(\mf g/\mf h,\overline{\partial+[Q,-]})$.
\end{rmk}
The following lemma gives a sufficient condition for condition e) of Assumptions \ref{ass:mainthmass} to be satisfied. In particular, the lemma applies when $\mf g$ is degreewise given by the sections of some vector bundle, and the bracket $[-,-]$ is a first order bidifferential operator.
\begin{lem}\label{lem:uniqueness}
Let $U\subset \mf g^0/\mf h^0$ as in d) of Assumptions \ref{ass:mainthmass}. If $v\in U$ and the initial value problem 
\begin{equation}\label{eq:expaddeg2}
\frac{d}{dt}\gamma_t = [\sigma_0(v),\gamma_t], \gamma_0 = 0 \in \mf g^2,
\end{equation}
has only the constant solution $\gamma_t \equiv 0 \in \mf g^2$, then $(Q')^{\sigma_0(v)} \in \mf g^1$ is a Maurer-Cartan element if and only if $Q' \in \mf g^1$ is.
\end{lem}
\begin{proof}
Let $\alpha:[0,1]\to \mf g^1$ be a solution to equation \eqref{eq:gaugeeq}, where $\alpha_0 = Q'$ is a Maurer-Cartan element. The expression
$$
\gamma_t := \partial\alpha_t + \frac{1}{2}[\alpha_t,\alpha_t]
$$
satisfies the initial value problem \eqref{eq:expaddeg2}, hence must be identically 0. 
\end{proof}
We can now state the main theorem, which roughly states that given a Maurer-Cartan element $Q\in \mf h^1$, the vanishing of a certain cohomology group implies that every Maurer-Cartan element $Q'$ of $\mf g$ near $Q$ is gauge equivalent to a Maurer-Cartan element of $\mf h$. Moreover, it also describes the amount of different gauge equivalences that take $Q'$ into $\mf h$.
\begin{thm}[\nrt{Main theorem}]\label{thm:mainthm}
Assume that we are in the setting as described in Assumptions \ref{ass:mainthmass}. Assume that $$H^1(\mf g/\mf h,\overline{\partial+[Q,-]}) = 0.$$ Then for every open neighborhood $V$ of $0\in U$ there exists an open neighborhood $\mc U\subset MC(\mf g)$ of $Q$ such that for any $Q'\in \mc U$ there exists a family $I$ in $V$ parametrized by an open neighborhood of $0\in H^0(\mf g/\mf h,\overline{\partial+[Q,-]})$ with $(Q')^{\sigma_0(v)}\in \mf h^1$ for $v\in I$. 
\end{thm}
\begin{proof}
The proof setup is similar to the proof of Theorem \ref{thm:liealgd1}, the difference being in the maps which are used. We repeat the key steps of the proof.
\begin{itemize}
    \item[i)] Construct a smooth map $\text{ev}_{Q'}:V \to \mf g^1/\mf h^1$ depending continuously on $Q'\in \mf g^1$,
    \item[ii)] construct a smooth map $R_{v,Q'}:\mf g^1/\mf h^1 \to \mf g^2/\mf h^2$ depending continuously on $(v,Q') \in V\times \mf g^1$,
\end{itemize}
satisfying
\begin{itemize}
    \item [a)] $\text{ev}_Q(0) = 0 \in \mf g^1/\mf h^1$, and $(D(\text{ev}_{Q}))_0 = -(\overline{\partial +[Q,-]}): \mf g^0/\mf h^0 \cong T_0 V \to T_0\mf g^1/\mf h^1 \cong \mf g^1/\mf h^1$. Moreover, $(Q')^{\sigma_0(v)} \in \mf h^1$ if and only if $\text{ev}_{Q'}(v) = 0 \in \mf g^1/\mf h^1$.
    \item [b)] $R_{v,Q'}(0) = 0 \in \mf g^2/\mf h^2$ for every $(v,Q') \in V\times \mf g^1$, and $(D(R_{0,Q}))_0 = \overline{\partial+[Q,-]}: \mf g^1/\mf h^1 \to \mf g^2/\mf h^2$.
    \item[c)] Whenever $Q'\in \mf g^1$ is Maurer-Cartan, for every $v\in \mf g^0/\mf h^0$ we have:
    $$
    R_{v,Q'}(\text{ev}_{Q'}(v)) = 0\in \mf g^2/\mf h^2.
    $$
\end{itemize}
We recall the way the result follows from these properties. Let $C$ be a complement to $$\ker(\overline{\partial+[Q,-]}:\mf g^1/\mf h^1 \to \mf g^2/\mf h^2) = \text{im}(\overline{\partial+[Q,-]}:\mf g^0/\mf h^0 \to \mf g^1/\mf h^1)$$ in $\mf g^1/\mf h^1$.\\
First property b) implies that $R_{0,Q}$ restricted to $C$ is an immersion at $0\in C$ as $C$ has trivial intersection with $\ker(\overline{\partial+[Q,-]}= (D(R_{0,Q}))_0:\mf g^1/\mf h^1 \to \mf g^2/\mf h^2 )$. By Lemma \ref{lem:immersion}, there exists an open neighborhood $O$ of $0 \in C$, an open neighborhood $S$ of $0 \in V$ and a neighborhood $\mc U_2$ of $Q\in \mf g^1$ such that $\left. R_{v,Q'}\right|_{O}$ is an injective immersion for $Q'\in \mc U_2$ and $v \in S$, where we use the continuous dependence of $R$ on the parameters $(v,Q') \in V\times \mf g^1$.\\
Property a) implies that ev$_{Q}$ intersects $O\subset C$ transversely in $0$, as $$\overline{\partial+[Q,-]} = -(D(\text{ev}_{Q}))_0:\mf g^0/\mf h^0 \to \mf g^1/\mf h^1,$$ and $$\ker(\overline{\partial+[Q,-]}:\mf g^1/\mf h^1 \to \mf g^2/\mf h^2) = \text{im}(\overline{\partial+[Q,-]}:\mf g^0/\mf h^0 \to \mf g^1/\mf h^1)$$ by the cohomological assumption. Therefore by Lemma \ref{lem:transverse}, there exists a neighborhood $\mc U_1$ of $Q\in \mf g^1$ such that for any $Q' \in \mc U_1$, there exists a $v\in S$ such that $\text{ev}_{Q'}(v) \in O$.\\
By property c), for any Maurer-Cartan element $Q' \in \mc U = \mc U_1 \cap \mc U_2$, and any $v\in V$ we have
$$
R_{v,Q'}(\text{ev}_{Q'}(v)) = 0.
$$
By injectivity of $R_{v,Q'}$ restricted to $O$, combined with the fact that ev$_{Q'}(v) \in O$, it follows that
$$
\text{ev}_{Q'}(v) = 0,
$$
or equivalently, $(Q')^{\sigma_0(v)}\in \mf h^1$. For the existence of the family of fixed points, apply the final statement of Lemma \ref{lem:transverse}, again using that $R_{v,Q'}$ is injective restricted to $O$.\\
We define the maps.
\begin{itemize}
    \item [i)]Let $Q'\in \mf g^1$. Then for $v \in V$, we set
$$
\text{ev}_{Q'}(v) = (Q')^{\sigma_0(v)} + \mf h^1.
$$
Then by assumption on the gauge action, the map depends continuously on $Q'$ and smoothly on $v\in V$.
\item[ii)] Next, for $(v,Q') \in V\times \mf g^1$, $\hat{Q} + \mf h^1 \in \mf g^1/\mf h^1$ we set
\begin{align*}
R_{v,Q'}(\hat{Q} + \mf h^1) = \partial(\hat{Q}) &+ [(Q')^{\sigma_0(v)}- \sigma_1((Q')^{\sigma_0(v)}+ \mf h^1),\sigma_1(\hat{Q} +\mf h^1)] \\&+ \frac{1}{2}[\sigma_1(\hat{Q} + \mf h^1),\sigma_1(\hat{Q}+\mf h^1)] + \mf h^2.
\end{align*}
\item[a)] The properties regarding the value of $\text{ev}_{Q'}$ hold by definition. For the differential, we compute for $v\in \mf g^0/\mf h^0 \cong T_0V$
\begin{align*}
    \left.\frac{d}{dt}\right|_{t=0} \text{ev}_{Q}(tv) &= \left.\frac{d}{dt}\right|_{t=0} (Q)^{t\sigma_0(v)} + \mf h^1\\
    &= \left.\left[\sigma_0(v),Q^{t\sigma_0(v)}\right]-\partial(v) \right|_{t=0} + \mf h^1\\
    &= [\sigma_0(v),Q] -\partial(v) + \mf h^1 \\
    &= -\overline{\partial+[Q,-]}(v).
\end{align*}
\item[b)] The properties regarding the values and derivatives of $R_{v,Q'}$ are immediate. 
\item[c)] Finally, for $Q'\in \mf g^1$, $v \in V$, setting $X = (Q')^{\sigma_0(v)}$, we compute
\begin{align*} 
     R_{v,Q'}(\text{ev}_{Q'}(v)) =& \partial\left(X\right) + \left[X-\sigma_1(X + \mf h^1),\sigma_1(X+ \mf h^1)\right]+ \frac{1}{2}\left[\sigma_1(X+\mf h^1),\sigma_1(X + \mf h^1)\right] + \mf h^2\\
     =& \partial\left(X\right) + \left[X-\sigma_1(X + \mf h^1),\sigma_1(X+ \mf h^1)\right] + \frac{1}{2}\left[\sigma_1(X+\mf h^1),\sigma_1(X + \mf h^1)\right] \\&
     +\frac{1}{2} \left[X-\sigma_1(X + \mf h^1),X-\sigma_1(X + \mf h^1)\right]+\mf h^2\\
     =&  \partial\left(X\right) + \frac{1}{2}\left[X- \sigma_1(X + \mf h^1) + \sigma_1(X + \mf h^1),\right.\left.X- \sigma_1(X + \mf h^1) + \sigma_1(X + \mf h^1)\right]\\
     =& \partial\left(X\right) + \frac{1}{2}\left[X,X\right] + \mf h^2,
\end{align*}
which vanishes if $Q'$ is a Maurer-Cartan element. Here the second equality holds because $X - \sigma_1(X +\mf h^1) \in \mf h^1$, and $[\mf h^1,\mf h^1] \subset \mf h^2$.
\end{itemize}
\end{proof}
\begin{rmk}\label{rmk:mainthmrmk}
We start with a couple of remarks, highlighting the differences with the proof of Theorem \ref{thm:liealgd1}.
\begin{itemize}
    \item[i)] The map $\partial$ does not appear in Theorem \ref{thm:liealgd1}. As the input of Theorem \ref{thm:mainthm} consists of a differential graded Lie algebra, and a Maurer-Cartan element $Q$ to start with, when the differential is an inner derivation of the graded Lie algebra one can take $\partial = 0$, and look at Maurer-Cartan elements near $Q$ rather than Maurer-Cartan elements of the \emph{differential} graded Lie algebra $(\mf g,[Q,-],[-,-])$ near $0\in \mf g^1$. One reason to do this is that the initial value problem for the gauge equation is homogeneous. As we will see in Section \ref{sec:dirac}, this cannot be done if $\partial$ is not inner.
    \item[ii)] The role of $\sigma_0$ and $\sigma_1$ in the case of Theorem \ref{thm:liealgd1} was to extend elements $v\in T_pM$ and $\rho \in A_p^\ast \otimes T_pM$ to constant sections of their respective vector bundles. This explains why the map $R_{v,Q'}$ was linear in that example: the pair $(\rho,[-,-])$ defined by a constant anchor and a zero bracket satisfies the axioms of a Lie algebroid, hence the second term is identically zero.
    \item[iii)] Note that although $\sigma_1$ appears in the proof, it does not appear in the statement. It is however necessary: as $\mf h$ is not an ideal, but only a graded Lie subalgebra, the quotients do not inherit a graded Lie algebra structure.
    \item[iv)] The requirement that $[-,-]:\mf g^1\times \mf g^1\to \mf g^2$ is continuous is needed to show that 
    \begin{align*}
    R:U\times \mf  g^1&\to C^\infty\left(\mf g^1/\mf h^1,\mf g^2/\mf h^2\right)\\
    (v,Q') &\mapsto R_{v,Q'}
    \end{align*}
    is continuous. This condition can be replaced by a different condition which is easier to check, and is satisfied in our applications. We give it in Lemma \ref{lem:Rcontcond} below.
    \item[v)] Marco Zambon pointed out that Theorem \ref{thm:mainthm} has \nrt{the following} deformation theoretic interpretation. The inclusion $i:(\mf h, \partial,[-,-]) \hookrightarrow (\mf g,\partial,[-,-])$ is a map of differential graded Lie algebras, hence induces a map on the level of Maurer-Cartan sets
    $$
    i_{MC}:MC(\mf h) \to MC(\mf g),
    $$
    and on the level of Maurer-Cartan varieties
    $$
    \overline{i_{MC}}:MC(\mf h)/\mf h^0 \to MC(\mf g)/\mf g^0.
    $$
    The cohomological assumption $H^1(\mf g/\mf h,\overline{\partial+[Q,-]}) =0$ implies that the induced map on the degree 1 cohomology 
    $$
    H^1(i):H^1(\mf h,\partial+[Q,-]) \to H^1(\mf g,\partial+[Q,-])
    $$
    is surjective by the long exact sequence associated to the short exact sequence of cochain complexes
    \[
    \begin{tikzcd}
    0\arrow{r}&(\mf h,\partial+[Q,-]) \arrow{r}{i}&(\mf g,\partial+[Q,-]) \arrow{r}& (\mf g/\mf h,\overline{\partial+[Q,-]}) \arrow{r} &0.
    \end{tikzcd}
    \]
    As $H^1(\mf h,\partial+[Q,-])$ and $H^1(\mf g,\partial+[Q,-])$ morally play the role of the tangent spaces to the respective Maurer-Cartan varieties at $Q$ with $H^1(i)$ playing the role of the differential of $i$, Theorem \ref{thm:mainthm} turns a stronger version of surjectivity of $H^1(i)$ into (local) surjectivity of $\overline{i_{MC}}$. 
    \item[vi)] It would be interesting to know if the requirement $H^1(\mf g/\mf h,\overline{\partial+[Q,-]}) = 0$ can be weakened to the surjectivity of $H^1(i)$. 
    \item[vii)] Theorem \ref{thm:mainthm} only yields a sufficient condition, and not a necessary one. In some places of the text we comment on this.
    \item[viii)] By modifying the maps $\text{ev}$ and $R$, we can generalize this result to the case where $\mf g$ is an $L_\infty$-algebra, and $\mf h$ a strict $L_\infty$-subalgebra, see \cite{stablinfty}.
    \item[ix)] If there exists a subspace $K\subset\mf g^1/\mf h^1$ such that for any Maurer-Cartan element $Q'\in MC(\mf g)$, $Q' + \mf h^1\in K$, we can replace $H^1(\mf g/\mf h,\overline{\partial + [Q,-]})$ in the hypothesis by $$H^1_{red}:= \frac{K\cap \ker(\overline{\partial + [Q,-]}:\mf g^1/\mf h^1 \to \mf g^2/\mf h^2)}{\text{im}(\overline{\partial+[Q,-]}:\mf g^0/\mf h^0 \to K)},$$ and the conclusion remains true. Observe that the space we quotient by is well-defined: As by assumption $Q$ is a Maurer-Cartan element, so is $Q^{t\sigma_0(v)}$ for all $t\in \mb R, v \in \mf g^0/\mf h^0$, and the differential $\partial+ [Q,-]$ is the $t$-derivative at $t=0$.
    \item[x)] In applications, we will often only specify $\mf h$ in degrees 0, 1 and 2. To complete this into a differential graded Lie subalgebra, we can take $\mf h^i = \mf g^i$ for $i>2$. 
\end{itemize}
\end{rmk}
As mentioned in Remark \ref{rmk:mainthmrmk}iv), there is a simpler condition than the continuity of $[-,-]:\mf g^1\times \mf g^1 \to \mf g^2$ which guarantees the continuity of $$R:U\times \mf  g^1\to C^\infty\left(\mf g^1/\mf h^1,\mf g^2/\mf h^2\right)$$ where $U$ is as in Assumptions \ref{ass:mainthmass}iv).
\begin{lem}\label{lem:Rcontcond}
Assume that there exists a closed subspace $F\subset \mf g^1$ such that
\begin{itemize}
    \item [-] $F$ has finite codimension in $\mf g^1$,
    \item [-] The bracket $[-,-]: \mf g^1 \times \mf g^1 \to \mf g^2/\mf h^2$ factors through $\overline{[-,-]}:\mf g^1/F \times \mf g^1/F \to \mf g^2/\mf h^2$.
\end{itemize}
Then the map $R:U\times \mf g^1 \to C^\infty\left(\mf g^1/\mf h^1,\mf g^2/\mf h^2\right)$ is continuous, when the right hand side carries the $C^1$-topology.
\end{lem}
\begin{proof}
Note first that because the assignment $U\times \mf g^1\ni (v,Q')\mapsto (Q')^{\sigma_0(v)}$ is continuous, it is sufficient to show that the restriction of $R$ to $\{0\} \times \mf g^1$ is continuous. Further, while the map $R$ is not linear in $\mf g^1$, it is not very far from it, as it is affine. So continuity of $R$ is equivalent to the continuity of the map
$$
R-R_{0,0}: \{0\}\times \mf g^1 \to C^\infty(\mf g^1/\mf h^1,\mf g^2/\mf h^2).
$$
This map takes values in $\Hom(\mf g^1/\mf h^1,\mf g^2/\mf h^2)$, the \emph{linear} maps, and is given by
$$
(R_{0,Q'} - R_{0,0}) = [Q' - \sigma_1(Q'+\mf h^1),\sigma_1(-)] + \mf h^2 = \overline{[p_F(Q'-\sigma_1(Q'+ \mf h^1),p_F(\sigma_1(-))]} \in \mf g^2/\mf h^2,
$$
which is the composition of a continuous linear map $p_F$ with a linear map between finite-dimensional vector spaces, hence continuous.
\end{proof}
\begin{rmk}
A way to think about Lemma \ref{lem:Rcontcond} is as follows: given a manifold $M$, the graded Lie algebra $\mf g = \mf X^\bullet(M)[1]$ of multivector fields and a point $p\in M$ we can take $\mf h := I_p\mf X^\bullet(M)[1]$. Then, $\mf g^i/\mf h^i \cong \wedge^{i+1} T_pM$, and the quotient map is simply the evaluation at $p\in M$. To compute the value of the Schouten-Nijenhuis bracket $[\pi_1,\pi_2]_{SN}$ of bivector fields $\pi_1,\pi_2 \in \mf X^2(M)$ in $p$ however, it is not sufficient to know the values of $\pi_1,\pi_2$ at $p\in M$: we also need to know the value of their first derivatives. In other words, we need to know the equivalence class of $\pi_1,\pi_2$ mod $I_p^2 \mf X^2(M)$, which could be taken as $F$.
\end{rmk}
\section{Higher order fixed points of Lie ($n$-)algebroids}\label{sec:applications} 
In this section we give some applications of Theorem \ref{thm:mainthm}. We first apply it to obtain a stability result for higher order fixed points of Lie algebroids (Theorem \ref{thm:liealgdk}), which we then show to be equivalent to Theorem 1.3 of  \cite{dufour2005stability}.

We then apply Theorem \ref{thm:mainthm} to obtain a stability result for (higher order) fixed points of Lie $n$-algebroids (Theorems \ref{thm:lienalg1} and \ref{thm:lienalgk}). 

The results for Lie $n$-algebroids can then be applied to singular foliations, and we obtain a stability result (Theorem \ref{thm:singfolstab}) and a formal rigidity result (Corollary \ref{cor:singfolrig}) for linear singular foliations. 

\subsection{Higher order fixed points of Lie algebroids}\label{sec:exampleLiealg}
Let $(A,Q)$ be a Lie algebroid over the manifold $M$. The first application will be Theorem 1.3 of \cite{dufour2005stability}.\\
This theorem yields a cohomological criterion for stability of a \emph{higher order} fixed point of a Lie algebroid\nrt{:
\begin{defn}\label{def:liealgk}
Let $(A,Q)$ be a Lie algebroid over $M$ with anchor $\rho:A \to TM$ and bracket $[-,-]_A$. Let $k\geq1$. We say $p\in M$ is a \emph{fixed point of order $k$ of $Q$} if for all sections $x,y \in \Gamma(A)$, we have 
$$
\rho(x) \in I_p^k \mf X(M), \,[x,y]_A \in I_p^{k-1} \Gamma(A),
$$
where for $l \in \mb Z_{\geq0}$, $I_p^l\subset C^\infty(M)$ is the ideal of functions vanishing at $p$ up to order $l$.
\end{defn}
\begin{rmk}\label{rmk:liealgkwhybracket}
     While the requirement on the anchor is what we are interested in, the restriction on the bracket comes out naturally: for instance, when $A = T^\ast M$, and the Lie algebroid structure comes from a Poisson structure, both requirements are satisfied, as the structure constants with respect to a coordinate frame are precisely the partial derivatives of coefficients of the Poisson bivector.\\
     Another motivation comes from the structure equations of a Lie algebroid: as
     $$
     \rho([x,y]_A) = [\rho(x),\rho(y)], 
     $$
     and the right hand side lies in $I_p^{2k-1} \mf X(M)$ if $\rho(x),\rho(y)\in I_p^k \mf X(M)$, so does the left hand side. Hence, it is natural to require that $[x,y]_A\in I_p^{k-1}\Gamma(A)$.\\
     Finally, we will see below that these requirements ensure that Lie algebroid structures for which $p \in M$ is a fixed point of order $k$ are the Maurer-Cartan elements of a differential graded Lie subalgebra of $\mf X(A[1])$.
\end{rmk}}
\subsubsection{The ingredients}\label{sssec:liealgk}
In this section we show that we are in the setting of Assumptions \ref{ass:mainthmass}.
\begin{itemize}
\item[i)] As we still deal with Lie algebroid structures, we take the graded Lie algebra $$(\mf g_{LA} = \mf X(A[1]),0,[-,-]).$$
\item[ii)]
We now define the graded Lie subalgebra of $\mf g_{LA}$ which contains the Lie algebroid structures with a higher order fixed point $p\in M$.
\begin{defn}\label{def:liealgksubalg}
Let $p \in M$, $j\geq -1,k\geq 0$ and let $A$ be a vector bundle. Define
\begin{align*}
\mf X^j_{p,k}(A[1]) := \{\delta\in \mf X^j(A[1]) \mid \delta(C^\infty(M)) \subset I_p^k &\Gamma(S^j(A^\ast[-1])),\\ &\delta(\Gamma(A^\ast[-1])) \subset I_p^{k-1}\Gamma(S^{j+1}(A^\ast[-1]))\},
\end{align*}
where $I_p\subset C^\infty(M)$ is the ideal of functions vanishing at $p$.
\end{defn}
These subspaces behave well with respect to the graded commutator of vector fields, extending Lemma \ref{lem:liesubalg}. We leave the proof to the reader.
\begin{lem}\label{lem:commcompat} For $p \in M$, $j_1,j_2\geq 0, k_1,k_2 \geq 0$, we have
$$[\mf X^{j_1}_{p,k_1}(A[1]), \mf X^{j_2}_{p,k_2}(A[1])] \subset \mf X^{j_1+j_2}_{p,k_1+k_2-1}\nrt{(A[1])}.$$
\end{lem}
\begin{rmk}
As the numbers $k_1,k_2, k_1+k_2-1$ represent the order to which elements of the respective spaces vanish at $p$ in a certain sense, Lemma \ref{lem:commcompat} is not too surprising. It is an extension of the fact that given vector fields $X,Y\in \mf X(M)$ vanishing to order $K_1,K_2 \geq 0$ in a point respectively, their Lie bracket $[X,Y]$ vanishes to order $K_1 + K_2 -1$.
\end{rmk}

\begin{cor} For $p\in M$,
$$
\left(\mf g_{LA}(p,k) := \bigoplus_{j = 0}^{\rk(A)}\mf X^j_{p,j(k-1)+1}(A[1]),0,[-,-]\right)
$$
is a graded Lie subalgebra of $(\mf X(A[1]),0 ,[-,-])$.
\end{cor}
\nrt{The graded Lie subalgebra $\mf g_{LA}(p,k)$ captures the Lie algebroid structures for which $p$ is a fixed point of order $k$:
\begin{lem}
Let $(A,Q)$ be a Lie algebroid over $M$, $k \geq 1$. A point $p\in M$ is a fixed point of order $k$ of $Q$ if and only if $Q \in \mf g_{LA}(p,k)$.
\end{lem}}

We therefore obtain a graded Lie subalgebra $\mf g_{LA}(p,k) \subset \mf g_{LA}$ corresponding to Lie algebroid structures for which $p\in M$ is a fixed point of order $k$. The cochain spaces of the relevant complex will consist of $\mf g_{LA}/\mf g_{LA}(p,k)$. We will give a description of these vector spaces, proving that the quotients $\mf g^i_{LA}/\mf g^i_{LA}(p,k)$ are finite-dimensional for $i = 0,1,2$. Note that Lemma \ref{lem:sesder} extends to $\mf g_{LA}(p,k)$:

\begin{lem}\label{lem:vanishingses}
For $j,k\geq 0$, let $ r:= j(k-1)$. Let $V := A[1]$, and $V^\ast = A^\ast[-1]$. \nrt{The sequence of $C^\infty(M)$-modules}
\[
\begin{tikzcd}\label{diag:vanishingses}
0\arrow{r}& I_p^{r}\Gamma(S^{j+1}(V^\ast)\otimes V) \arrow{r}{\iota}& \mf X^j_{p,r+1}(V) \arrow{r}{\sigma} & I_p^{r+1}\Gamma(S^j(V^\ast)\otimes TM) \arrow{r}& 0
\end{tikzcd}
\]
\nrt{is exact.}
\end{lem}
Consequently, the cochain spaces $\mf g_{LA}/\mf g_{LA}(p,k)$ sit inside a short exact sequence:
\begin{cor} Using notation from Lemma \ref{lem:vanishingses}, there is a short exact sequence
\[
\begin{tikzcd}
0 \arrow{r}& J^{r-1}_p(S^{j+1}(V^\ast)\otimes V) \arrow{r}{\bar{\iota}} & \mf X^j(V)/\mf X^j_{p,r+1}(V) \arrow{r}{\bar{\sigma}} & J_p^{r}(S^j(V^\ast)\otimes TM)\arrow{r} & 0
\end{tikzcd},
\]
where for a vector bundle $E$ and a positive integer $l$, $$J^l_p(E) := \Gamma(E)/I_p^{l+1}\Gamma(E)$$ are the $l$-jets of $E$ at $p$. In particular, the cochain spaces are finite dimensional vector spaces.
\end{cor}
The differential however, does not restrict to the outer parts of the sequence.
\begin{rmk}
For $k = 1$, we have $r= 0$, and we are back in the setup of Section \ref{sec:directproof}.
\end{rmk}
\item[iii)]By restricting \nrt{to} a sufficiently small coordinate neighborhood of the point $p\in M$ over which $A$ trivializes, we can choose splittings $\sigma_i:\mf g_{LA}^i/\mf g_{LA}^i(p,k)\to \mf g_{LA}^i$ for $i = 0,1$ by lifting jets to polynomial sections.
\item[iv)] Now pick a Lie algebroid structure $Q$ such that $p\in M$ is a fixed point of order $k$, i.e. a Maurer-Cartan element $Q\in \mf g^1_{LA}(p,k)$.
\end{itemize}
We now check that the data above satisfies Assumptions \ref{ass:mainthmass}.
\begin{itemize}
\item[a)]
As each of the degrees of $\mf g_{LA}$ are the sections of some vector bundle we take various $C^s$-topologies. 
\begin{itemize}
    \item [-] On $\mf g_{LA}^0$, take the $C^\infty$-topology,
    \item[-] On $\mf g_{LA}^1$, take the $C^{2k-1}$-topology,
    \item[-] On $\mf g_{LA}^2$, take the $C^{2k-2}$-topology.
\end{itemize}
These choices make the projections continuous.
\item[b)] The differential is identically zero, hence is continuous.
\item[c)] This follows from a local computation, where it is crucial that for $X,Y \in \mf g_{LA}^1$, the $(2k-2)$-jet of $[X,Y]$ depends bilinearly on the $(2k-1)$-jets of $X$ and $Y$.
\item[d)] As the gauge action for $\mf g_{LA}$ is given by the flow of some vector field, the choice of the splittings above implies that it is defined for any $v \in \mf g_{LA}^0/\mf g_{LA}^0(p,k)$.\\
Finally,
\begin{lem}
The gauge action $\mf g_{LA}^0/\mf g_{LA}^0(p,k)\times \mf g^1_{LA} \to \mf g^1_{LA}$ is continuous.
\end{lem}
\begin{proof}
We need to show that for $(v,Q_1) \in \mf g_{LA}^0/\mf g_{LA}^0 (p,k)\times \mf g_{LA}^1$, any compact set $K\subset M$, and $\epsilon >0$, there is a compact set $K'\subset M$, $\epsilon'>0$ and $\delta>0$ with the property that if
$$
\lVert Q_1-Q_2 \rVert_{K',2k-1}<\epsilon', \lVert v-w\rVert<\delta,
$$
we have
$$
\lVert Q_1^{\sigma_0(v)} - Q_2^{\sigma_0(w)} \rVert_{K,2k-1} < \epsilon,
$$
where $\lVert \cdot\rVert_{K,2k-1}$ denotes the $C^{2k-1}$-seminorm associated to the compact set $K$.\\
Note that
\begin{align*}
    \lVert Q_1^{\sigma_0(v)} - Q_2^{\sigma_0(w)} \rVert_{K,2k-1} \leq \lVert Q_1^{\sigma_0(v)}-Q_1^{\sigma_0(w)} \rVert_{K,2k-1} + \lVert (Q_1-Q_2)^{\sigma_0(w)} \rVert_{K,2k-1}.
\end{align*}
By uniform continuity of $Q_1$ restricted to $K$, there exists $\delta>0$ such that the first term is at most $\frac{\epsilon}{2}$ if $\lVert v-w \rVert < \delta$. Recall that the gauge action by $\sigma_0(u)$ is given by some vector bundle automorphism of $\mf g_{LA}^1$, and vector bundle automorphisms induce continuous maps on the space of sections. Hence there exists some constant $C>0$, such that for $K' = \phi(\overline{B_{v}(\delta)} \times K)$
$$
\lVert (Q_1-Q_2)^{\sigma_0(w)}\rVert_{K,2k-1} \leq C\lVert Q_1-Q_2\rVert_{K',2k-1}.
$$
Here $\phi:\mf g_{LA}^0/\mf g_{LA}^0(p,k) \times K \to M$ for $(w,x) \in \mf g_{LA}^0/\mf g_{LA}^0(p,k)\times K$ is defined by
$$
\phi(w,x) = \phi_w(x),
$$
and $\phi_w$ is the time-1 flow on $M$ of the symbol of the element $\sigma_0(w)$, using that $\phi$ is continuous.\\
Setting $\epsilon'= \frac{\epsilon}{2C}$ then yields the result. 
\end{proof}
\item[e)] Lemma \ref{lem:uniqueness} implies that the gauge action preserves Maurer-Cartan elements.
\end{itemize}

\subsubsection{Applying the main theorem}
Now that all assumptions are satisfied, taking $\mf g_{LA} = \mf X(A[1])$, $\mf g_{LA}(p,k) = \bigoplus_{i = 0}^{\rk(A)} \mf X_{p,i(k-1)+1}^i(A[1])$ \nrt{as in Definition \ref{def:liealgksubalg}}, the main theorem implies:
\begin{thm}\label{thm:liealgdk}
Let $(A,Q)$ be a Lie algebroid over $M$. Let $p \in M$ be a fixed point of order $k\geq 1$, \nrt{as in Definition \ref{def:liealgk}}. Assume that $$H^1(\mf g_{LA}/\mf g_{LA}(p,k),\overline{[Q,-]}) = 0.$$ Then for every open neighborhood $U$ of $p\in M$, there exists a $C^{2k-1}$-neighborhood $\mc U$ of $Q$ such that for any Lie algebroid structure $Q' \in \mc U$ there is a family $I$ in $U$ of fixed points of order $k$ of $Q'$ parametrized by an open neighborhood of $$0 \in H^0(\mf g_{LA}/\mf g_{LA}(p,k),\overline{0^Q}).$$
\end{thm}
\subsubsection{Equivalence with the Dufour-Wade stability theorem for Lie algebroids}\label{sec:comparison}
In this section we compare Theorem \ref{thm:liealgdk} with Theorem 1.3 of \cite{dufour2005stability}. We will show that Theorem \ref{thm:liealgdk} is equivalent to Theorem 1.3 of \cite{dufour2005stability}. The conclusions of the theorems are equivalent so we will show that the cohomological assumptions are also equivalent. The theorem was originally proven in \cite{dufour2005stability} using a special kind of multivector fields on $A^\ast$ rather than vector fields on $A[1]$, and most notably, the differentials are not the same.\\
The first difference can be quite easily explained: the graded Lie algebra of multivector fields on $A^\ast$ which are fiber-wise linear \cite[Section 4.9]{CrMo08} \nrt{is} canonically isomorphic to the graded Lie algebra of multiderivations on $A$, and the same holds for vector fields on $A[1]$ \cite[Section 2.5]{CrMo08}.\\
The second difference is that we work with a quotient of the vector fields on $A[1]$, whereas in \cite{dufour2005stability} the authors work with representatives of the classes in the quotients, namely the multivector fields on $A^\ast$, which are fiber-wise linear, and polynomial in coordinates on $M = V$. However, even when making these identifications, the differential $$\overline{[Q,-]}:\mf g_{LA}/\mf g_{LA}(p,k)\to \mf g_{LA}/\mf g_{LA}(p,k)$$ and the differential of \cite{dufour2005stability} on the complexes do not agree. 

We want to show that the vanishing of the respective cohomologies are equivalent conditions. To see this, we first do an intermediate step, which will double as a more convenient way to check the cohomological hypothesis of Theorem \ref{thm:liealgdk}.\\
The complex
\[
\begin{tikzcd}
    \mf X^0(A[1])/\mf X^0_{p,1}(A[1]) \arrow{r}{\overline{[Q,-]}}& \mf X^1(A[1])/\mf X^1_{p,k}(A[1]) \arrow{r}{\overline{[Q,-]}} &\mf X^2(A[1])/\mf X^2_{p,2k-1}(A[1]) 
\end{tikzcd}
\]
has a finite descending filtration. Indeed: \nrt{define} for $t = 0,\dots, k$ (omitting the $A[1]$ for convenience):
\begin{align*}
F^t(\mf X^0/\mf X^0_{p,1})&:= \begin{cases} \mf X^0/\mf X^0_{p,1} & t = 0,\dots, k-1 \\ 0 & t = k\end{cases},\\
F^t(\mf X^1/\mf X^1_{p,k})&:= \mf X^1_{p,t}/\mf X^1_{p,k},\\
F^t(\mf X^2/\mf X^2_{p,2k-1}) &:= \mf X^2_{p,k-1+t}/\mf X^2_{p,2k-1}.
\end{align*}
The differentials preserve this filtration, so we obtain $k$ complexes on the graded quotient: for $t = 0,\dots, k-2$ we get
\begin{equation}\label{diag:assgrad0}
\begin{tikzcd}
0\arrow{r} & \mf X^1_{p,t}/\mf X^1_{p,t+1} \arrow{r}{gr(\overline{[Q,-]})}& \mf X^2_{p,t+k-1}/\mf X^2_{p,t+k}
\end{tikzcd},
\end{equation}
and for $t = k-1$ we have
\begin{equation}\label{diag:assgradk1}
\begin{tikzcd}
\mf X^0/\mf X_{p,1}^0\arrow{r}{gr(\overline{[Q,-]})} &\mf X^1_{p,k-1}/\mf X^1_{p,k} \arrow{r}{gr(\overline{[Q,-]})} & \mf X^2_{p,2k-2}/\mf X^2_{p,2k-1}.
\end{tikzcd}
\end{equation}
Denote the corresponding cohomology groups by
\begin{equation}\label{eq:gradedcohomologygroups}
gr_t H^1(\mf g_{LA}/\mf g_{LA}(p,k),\overline{[Q,-]})
\end{equation}
for $t = 0,\dots, k-1$.
The vanishing of these cohomologies is equivalent to the vanishing of $H^1(\mf g_{LA}/\mf g_{LA}(p,k),\overline{0}^Q)$.
\begin{prop}
$$
H^1(\mf g_{LA}/\mf g_{LA}(p,k),\overline{[Q,-]}) = 0 \iff gr_t H^1(\mf g_{LA}/\mf g_{LA}(p,k),\overline{[Q,-]}) = 0 
$$
for $t=0,\dots, k-1$.
\end{prop}
\begin{proof}
$``\implies"$
The key observation is that we can pick linear splittings of the sequences 
\[
\begin{tikzcd}
0 \arrow{r} & \mf X^i_{p,t+1} \arrow{r}& \mf X^i_{p,t} \arrow{r} & \mf X^i_{p,t}/\mf X^i_{p,t+1} \arrow{r}& 0
\end{tikzcd}
\]
for $i =1, t = 0,\dots, k-1$, and $i =2, t= k-1,\dots, 2k-2$. This gives rise to (filtered) linear isomorphisms 
$$
\mf X^i_{p,(i-1)(k-1)}/\mf X^i_{p,k + (i-1)(k-1)} \cong \bigoplus_{t = 0}^{k-1} \mf X^i_{p,t+(i-1)(k-1)}/\mf X^i_{p,t+1+(i-1)(k-1)} 
$$
for $i = 1,2$. Decomposing the differential with respect to this isomorphism it is block upper triangular, with the block diagonal being precisely $gr(\overline{[Q,-]})$. \\
$'' \Longleftarrow"$ Let $\delta + \mf X_{p,k}^1 \in \mf X^1/\mf X_{p,k}^1$ such that $[Q,\delta] \in \mf X_{p,2k-1}^2$. Then in particular $[Q,\delta] \in \mf X^2_{p,k} \supset \mf X_{p,2k-1}^2$. As $$gr_0 H^1(\mf g_{LA}/\mf g_{LA}(p,k),\overline{[Q,-]}) = 0,$$ we find that $\delta \in \mf X_{p,1}$, using \eqref{diag:assgrad0} for $t=0$. Applying this reasoning inductively, we eventually find that $\delta \in \mf X_{p,k-1}$.\\ Finally, the assumption
$$
gr_{k-1}H^1(\mf g_{LA}/\mf g_{LA}(p,k),\overline{[Q,-]}) = 0,
$$
implies that there exists $X\in \mf X^0$ such that $[Q,X] - \delta \in \mf X_{p,k}^1$ using \eqref{diag:assgradk1}, concluding the proof.
\end{proof}
Now we can make a direct connection with \cite{dufour2005stability}: the cohomology groups appearing in their Theorem 1.3 are exactly $gr_t H^1(\mf g_{LA}/\mf g_{LA}(p,k),\overline{[Q,-]})$. 
\begin{prop}
Let $(A = \mf g\times M,Q)$ be a Lie algebroid over the vector space $M = V$ such that $ p= 0\in M$ is a fixed point of order $k$. Let $\Pi_Q$ denote the corresponding fiberwise linear Poisson structure on $A^\ast$. Then 
$$
H^{2,t}_{lin}(\Pi^{(k)}_Q) \cong gr_t H^1(\mf g_{LA}/\mf g_{LA}(p,k),\overline{[Q,-]})
$$
for $t = 0,\dots, k-1$. Here the left hand side denotes the cohomology as defined in \cite{dufour2005stability}, where $\Pi_Q^{(k)}$ denotes the $k$-th order Taylor expansion of $\Pi_Q$ around $0_p$. The right hand side denotes the graded cohomology of the filtered complex as defined above \eqref{eq:gradedcohomologygroups}.
\end{prop}
\begin{proof}
As shown in \cite{CrMo08}, there is a $C^\infty(M)$-linear isomorphism of graded Lie algebras
$$
\Phi:\mf X^\bullet_{lin}(A^\ast)[1] \to \mf X(A[1]).
$$
Let $p \in M$ denote the origin. On the left hand side, we can define subspaces corresponding to the order of vanishing at $0_p \in A^\ast$:
For $q,t\geq 0$, let
$$
\mf X^q_{lin}(A^\ast)_{p,t} := I_{0_p}^{t} \mf X^q(A^\ast) \cap \mf X^q_{lin}(A^\ast),
$$
where $I_{0_p}\subset C^\infty(A^\ast)$ is the ideal of functions on $A^\ast$ vanishing at $0_p\in A^\ast$. It is straightforward to verify that 
$$
\Phi(\mf X^q_{lin}(A^\ast)_{p,t}) = \mf X^{q-1}_{p,t}.
$$
Therefore, we have$$ [\mf X^{q_1}_{lin}(A^\ast)_{p,k_1},\mf X^{q_2}_{lin}(A^\ast)_{p,k_2}] \subset \mf X^{q_1+q_2-1}_{lin}(A^\ast)_{p,k_1+k_2-1}.$$ In particular, the assumption that $p$ is a fixed point of order $k$ in terms of $\Pi_Q$ means that $\Pi_Q \in \mf X^2(A^\ast)_{p,k}$. This implies that the differential $[\Pi_Q,-]:\mf X_{lin}^\bullet(A^\ast) \to \mf X_{lin}^{\bullet+1}(A^\ast)$ descends to 
$$
\overline{[\Pi_Q,-]}:\mf X_{lin}^\bullet(A^\ast)/\mf X_{lin}^\bullet(A^\ast)_{p,1 + \bullet(k-1)} \to \mf X^{\bullet+1}_{lin}(A^\ast)/\mf X^{\bullet+1}_{lin}(A^\ast)_{p,k + \bullet(k-1)}.
$$

Again the order of vanishing gives rise to a filtered complex, and $\Phi$ descends to a filtered isomorphism of complexes
\begin{equation}\label{eq:inducedmaplinvf}
\overline{\Phi}:(\mf X_{lin}^\bullet(A^\ast)/\mf X_{lin}^\bullet (A^\ast)_{p,1+\bullet(k-1)}, \overline{[\Pi_Q,-]}) \to (\mf X^{\bullet-1}(A[1])/\mf X^{\bullet-1}(A[1])_{p,1+(\bullet-1)(k-1)}, \overline{[Q,-]}),
\end{equation}
giving rise to an isomorphism of the corresponding graded quotients. It remains to be shown that $H^{2,s}_{lin}(\Pi^{(k)}) $ is isomorphic to the cohomology of the corresponding graded quotient of the left hand side of \eqref{eq:inducedmaplinvf}. However, mapping a homogeneous degree $s$ polynomial (in $V$) fiberwise $k$-vector field to its corresponding class in the graded quotient $\mf X_{lin}^k(A^\ast)_{p,s}/\mf X_{lin}^k(A^\ast)_{p,s+1}$ defines an isomorphism of complexes.

\end{proof}
\subsection{Fixed points of Lie $n$-algebroids}
In this section we apply the main theorem to Lie $n$-algebroids and obtain a result analogous to Theorem \ref{thm:liealgdk}. In classical terms, Lie $n$-algebroids can be described as a graded vector bundle with an anchor, and a collection of higher brackets on the space of its sections. We will first look at the case when the anchor vanishes in a point $p\in M$, and interpret Theorem \ref{thm:mainthm} in this setting.
\subsubsection{Lie $n$-algebroids}
We start with the definition of a Lie $n$-algebroid. Just as for Lie algebroids there are several ways to characterize Lie $n$-algebroids. We choose the graded geometric point of view.
\begin{defn}
Let $M$ be a smooth manifold. A \textit{Lie $n$-algebroid} over $M$ is a pair $(E,Q)$, where
\begin{itemize}
    \item[i)] $E = \bigoplus_{i = 1}^{n} E_i[i-1]$ is a non-positively graded vector bundle,
    \item[ii)] $Q:\Gamma(S(E^\ast[-1])) \to \Gamma(S(E^\ast[-1]))$ is a degree 1 $\mb R$-linear derivation,
\end{itemize}
satisfying
\begin{itemize}
    \item [a)] $Q^2 =\frac{1}{2}[Q,Q]= 0$.
\end{itemize}
We denote by
$$
C^\infty(E[1],Q) := (\Gamma(S(E^\ast[-1])),Q)
$$
the differential graded commutative algebra of \emph{smooth functions} on $E[1]$, and by
$$
(\mf X(E[1]), [Q,-],[-,-]) := (\text{Der}_{\mb R}(C^\infty(E[1])), [Q,-],[-,-])
$$
the differential graded Lie algebra of \emph{vector fields} on $E[1]$.
\end{defn}
\begin{rmk}\label{rmk:lienalgebroidprop}
\begin{itemize}
    \item[]
    \item [i)] $C^\infty(E[1])$ is bigraded as an algebra: the grading coming from the grading on $E$, and the grading coming from the symmetric power of $E^\ast[-1]$, which we will refer to as the \emph{weight}. 
    \item[ii)] The bigrading on $C^\infty(E[1])$ induces a bigrading on $\mf X(E[1])$. The grading coming from the grading on $E$ will be called the $\emph{degree}$ of a vector field, while the grading coming from the weight will be called the $\emph{arity}$.
    \item[iii)] We do not require $Q$ to be homogeneous with respect to the bigrading, merely that $Q$ has degree 1. So $Q$ can be decomposed as $Q = \sum_{i= 0}^{n} Q^{(i)}$, where $Q^{(i)}$ raises the symmetric power by $i$.
    \item[iv)] Due to T. Voronov \cite{qmanvoronov}, the data of a Lie $n$-algebroid can be described equivalently in terms of an anchor $\rho: \Gamma(E_1)\to \mf X(M)$ and a collection of multibrackets $$\ell_k:S^k(E[1])\to E[1]$$ for $k\geq1 $ of degree 1 satisfying quadratic identities. 
    \item[v)] When $E$ is concentrated in degree 0, this reduces to the definition of a Lie algebroid.
    \item[vi)] $\mf X(E[1])$ is a graded Lie algebra whose Maurer-Cartan elements are precisely the Lie $n$-algebroid structures.
    \end{itemize}
\end{rmk}
Although the functions and vector fields on $E[1]$ are more complicated to describe than in the Lie algebroid setting, the analogue of Lemma \ref{lem:sesder} still holds:
\begin{prop}
There is a short exact sequence of graded $C^\infty(M)$-modules:
\[
\begin{tikzcd}
0\arrow{r}& C^\infty(E[1]) \otimes E[1] \arrow{r}{\iota} &\mf X(E[1]) \arrow{r}{\sigma} & C^\infty(E[1]) \otimes \mf X(M) \arrow{r} & 0
\end{tikzcd},
\]
where tensor products are over $C^\infty(M)$, $\iota$ and $\sigma$ are the contraction and restriction to $C^\infty(M)$ respectively. Moreover, any choice of connections on the $E_i$ give rise to a splitting of the sequence.
\end{prop}
Note however that there is a difference from the Lie algebroid case: the module $C^\infty(E[1])$ is not finitely generated as $C^\infty(M)$-module. Consequently, there is no (finite-dimensional) vector bundle $F$ such that $C^\infty(E[1]) =\Gamma(F)$. However, not all hope is lost: the sequence above restricts to every degree, and every degree is finitely generated. Hence, vector fields of a given degree are isomorphic to the sections of some vector bundle. This allows us to make sense of vector fields of given degree being $C^p$-close for some $p$.

\nrt{We define the type of fixed points we are interested in:
\begin{defn}\label{def:lienalg1l}
Let $(E,Q)$ be a Lie $n$-algebroid over $M$ with anchor $\rho:E_1\to TM$, and unary bracket $\ell_1$. Let $l\geq 0$. We say $p\in M$ is a \emph{fixed point of order $(1,l)$ of $Q$} if for all sections $x\in \Gamma(E_1), e\in \Gamma(E_2)$, we have $$\rho(x) \in I_p \mf X(M), \ell_1(e)\in I_p^l\Gamma(E_1).$$ 
\end{defn}
\begin{rmk}\label{rmk:almostliealgd}
\begin{itemize}
\item[]
\item[i)] We include the order of vanishing of the map $\ell_1:\Gamma(E_2)\to \Gamma(E_1)$ in the definition, as different $l$ will lead to different criteria. When applying the result to singular foliations in Section \ref{sec:singfolapplic}, it is essential that we have a criterion for the stability of a fixed point of order $(1,l)$, instead of just a point where the anchor vanishes.
\item[ii)] Before we go on to apply Theorem \ref{thm:mainthm}, we make an observation, pointed out in \cite{univlinfty}: For any Lie $n$-algebroid $(E,Q)$, $(E_1,Q^{(1)})$ is an almost Lie algebroid, which makes sense as the component $Q^{(1)}$ preserves the functions on $E[1]$:
$$
Q^{(1)}(\Gamma(S(E_1^\ast[-1]))) \subset \Gamma(S(E_1^\ast[-1])).
$$
If we now define a fixed point of $Q$ to be a fixed point of $Q^{(1)}$,
Theorem \ref{thm:liealgd1} in combination with Remark \ref{rmk:liealg1} now yields a stability criterion for fixed points of order $(1,0)$ of Lie $n$-algebroids.\\
A Lie $n$-algebroid however contains more data than just the underlying almost Lie algebroid structure on its degree $-1$ part. Therefore, if we want to know if a fixed point is stable for nearby Lie $n$-algebroid structures, we might be able to do better, taking into account properties of other components of $Q$. 
\end{itemize}
\end{rmk}}
\subsubsection{The ingredients}
In this section we show that we are in the setting of Assumptions \ref{ass:mainthmass}.
\begin{itemize}
\item[i)]
We first define the graded Lie algebra where things will take place. There is a slight difference with Lie algebroids here: while for Lie algebroids, the notion of arity and degree of vector fields coincide, that is not the case here. In particular, this implies that degree zero vector fields do not only consist of covariant differential operators on the graded vector bundle $E$, but also contain  $C^\infty(M)$-linear maps $X:\Gamma(E_2^\ast[-2])\to \Gamma(S^2(E_1^\ast[-1])) $ for example. The induced gauge action on a vector field by this element is nothing but 
$$
\mf X^1(E[1])\ni Q \mapsto Q + [X,Q] \in \mf X^1(E[1]).
$$
In particular, if $Q$ is arity-homogeneous, $Q$ gauge transformed by $X$ will not be. Moreover, as $X$ is $C^\infty(M)$-linear, it induces the identity on $M$, so it is not relevant for moving the fixed point around. We therefore define
\begin{equation}\label{eq:glalien}
\mf g_{LnA}^i:= \begin{cases}  {\null}^0\mf X^{0}(E[1]) & i = 0, \\
\mf X^i(E[1])& i > 0,\end{cases}
\end{equation}
where $\null^0 \mf X^0(E[1])$ denotes the set of vector fields with degree and arity $0$. We take the differential $\partial = 0$.
\item[ii)]
\nrt{The general algebraic framework does not allow such a uniform description as for Lie algebroids. Nevertheless,} given a fixed point $p\in M$ of order $(1,l)$ for $l\geq 0$\nrt{,} we can still associate a graded Lie subalgebra $\mf g_{LnA}(p,(1,l))$ of $\mf g_{LnA}$ to it,  such that its Maurer-Cartan elements are those Lie $n$-algebroid structures which have $p\in M$ a fixed point of order $(1,l)$. As the only relevant degrees to Theorem \ref{thm:mainthm} are 0, 1 and 2, we only specify the Lie subalgebra in these degrees. The definition can be extended by setting $\mf g_{LnA}^{\geq 3}(p,(1,l)) = \mf g_{LnA}^{\geq 3}$ to obtain a Lie subalgebra.
\begin{defn}\label{defn:sglalien1} Let $E = \bigoplus_{i = 1}^{n} E_i[i-1]$ be a graded vector bundle over $M$, $p\in M$ and let $l\geq 0$ be an integer. Define
\begin{align*}
\mf g_{LnA}^0(p,(1,l)) := \{\delta \in \null^0\mf X^0(E[1]) \mid&\, \delta(C^\infty(M)) \subset I_p C^\infty(M)\}\\
\mf g_{LnA}^1(p,(1,l)):= \{\delta \in \mf X^1(E[1]) \mid &\,\delta(C^\infty(M)) \subset I_p\Gamma(E_1^\ast[-1]),\\& \,\delta^{(0)}(\Gamma(E_{1}^\ast[-1])) \subset I_p^l\Gamma(E_{2}^\ast[-2])\}\\
\mf g_{LnA}^2(p,(1,l)):= \{\delta\in \mf X^2(E[1]) \mid&\, \delta^{(2)}(C^\infty(M)) \subset I_p \Gamma(S^2(E^\ast_{1}[-1])),\\ &\, \delta^{(1)}(C^\infty(M)) \subset I_p^{l+1} \Gamma(E_{2}^\ast[-2]),\\
&\,\delta^{(1)}(\Gamma(E_1^\ast[-1]))\subset I_p^{l}\Gamma(E_2^\ast[-2])\Gamma(E_1^\ast[-1])\},
\end{align*}
where the superscripts in parentheses correspond to the arity.
\end{defn}
The following lemma shows that these subspaces actually yield a graded Lie subalgebra of $\mf g_{LnA}$.
\begin{lem}\label{lem:diffcompat1}
The subspaces defined above satisfy
\begin{align*}
[\mf g^0_{LnA}(p,(1,l)),\mf g_{LnA}^i(p,(1,l))]&\subset \mf g_{LnA}^i(p,(1,l)),\\ [\mf g_{LnA}^1(p,(1,l)),\mf g_{LnA}^1(p,(1,l))]&\subset \mf g_{LnA}^2(p,(1,l))
\end{align*}
for $i = 0,1,2.$
\end{lem}

\nrt{The subalgebra $\mf g_{LnA}(p,(1,l))$ encodes the Lie $n$-algebroid structures for which $p$ is a fixed point of order $(1,l)$:
\begin{lem}
Let $(E,Q)$ be a Lie $n$-algebroid over  $M$. Let $l\geq 0$ be an integer. A point $p\in M$ is a fixed point of order $(1,l)$ if and only if $Q\in \mf g_{LnA}^1(p,(1,l))$.
\end{lem}}

Similar to the Lie algebroid case, one can show that $\mf g_{LnA}^0/\mf g_{LnA}^0(p,(1,l))\cong T_p M$, and $\mf g_{LnA}^i/\mf g_{LnA}^i(p,(1,l))$ for $i = 1,2$ consists of jets at $p$ of sections of some vector bundles. We will give a description of the complex below in Remark \ref{rmk:expldesccohom}.

\item[iii)] By the description above, the splittings can be chosen to be polynomial sections when restricting to a small neighborhood of $p\in M$.
\item[iv)] Fix a Lie $n$-algebroid structure $Q\in \mf g_{LnA}(p,(1,l))$, i.e. a Lie $n$-algebroid structure on $E$ such that $p\in M$ is a fixed point of order $(1,l)$ of $Q$.
\begin{rmk}\label{rmk:expldesccohom}
We can give an explicit description of $(\mf g_{LnA}/\mf g_{LnA}(p,(1,l)),\overline{[Q,-]})$ in degrees 0, 1 and 2 in terms of the multibrackets. After picking coordinates around the fixed point $p$, we may assume that $M = V$ is a vector space and $p$ is the origin. Let $E = V\times \bigoplus_{i=1}^n \mf e_i[i-1]$ a trivial bundle. Then the cochain spaces are isomorphic to
\begin{align*}
    \mf g_{LnA}^0/\mf g_{LnA}^0(p,(1,l)) \cong &V\\
    \mf g_{LnA}^1/\mf g_{LnA}^1(p,(1,l)) \cong & \mf e_1^\ast \otimes V \oplus J_p^{l-1}(E_2^\ast \otimes E_1)\\
    \mf g_{LnA}^2/\mf g_{LnA}^2(p,(1,l)) \cong & S^2(\mf e_1^\ast[-1])\otimes V \oplus J^l_p(E_2^\ast \otimes TV) \oplus J_p^{l-1}(E_1^\ast \otimes E_2^\ast \otimes E_1).
\end{align*}
For $v \in V$, the map $$\overline{[Q,-]}:\mf g_{LnA}^0/\mf g_{LnA}^0(p,(1,l)) \to  \mf g_{LnA}^1/\mf g_{LnA}^1(p,(1,l))$$ is defined by
\begin{equation}\label{eq:diff1l0}
\overline{[Q,-]}(v) = (d^\tau_{CE}(v),-v(\ell_1)+ I_p^{l}\Gamma(E_2^\ast \otimes E_1)),
\end{equation}
where we use that the bundles are trivialized. Here we recall that $d^\tau_{CE}$ is the Chevalley-Eilenberg differential associated to the linear holonomy representation $\tau:\mf e_1 \to \End(T_pM)$ as in equation \eqref{eq:linholrep} and Definition \ref{def:cecomplex}.\\
For $\alpha: \mf e_1 \to V$ and $\delta: \Gamma(E_2)\to \Gamma(E_1)$, the map $$\overline{[Q,-]}: \mf g_{LnA}^1/\mf g_{LnA}^1(p,(1,l)) \to  \mf g_{LnA}^2/\mf g_{LnA}^2(p,(1,l))$$ is defined by
\begin{equation}\label{eq:diff1ltot}
\overline{[Q,-]}(\alpha,\delta+I_p^l\Gamma(E_2^\ast \otimes E_1)) = (d^\tau_{CE}(\alpha), \rho\circ \delta + \alpha \circ \ell_1 + I_p^{l+1}\Gamma(E_2^\ast \otimes TV), \overline{[Q,-]_{(2,1)}}(\alpha,\delta+I_p^l)),
\end{equation}
where the last term for $e \in \mf e_1, f \in \mf e_2$ is defined by 
\begin{equation}\label{eq:diff1l3}
\overline{[Q,-]_{(2,1)}}(\alpha,\delta+I_p^l)(e,f) = \alpha(e)(\ell_1(f)) + \ell_2(e,\delta(f)) - \delta(\ell_2(e,f)) + I_p^l\Gamma(E_1).
\end{equation}
The cocycles are therefore pairs $(\alpha,\delta+ I_p^l\Gamma(E_2^\ast \otimes E_1))$, where $\alpha$ is a Chevalley-Eilenberg cocycle for $\mf e_1$ with values in the representation $V$, and extending it to any section $\hat{\alpha}$, it satisfies
$$
\rho \circ \delta + \hat{\alpha}\circ \ell_1 \in I_0^{l+1}\Gamma(E_2^\ast \otimes TV)
$$
and for any $e\in \mf e_1, f \in \mf e_2,$
$$
\alpha(e)(\ell_1(f)) + \ell_2(e,\delta(f)) - \delta(\ell_2(e,f)) \in  I_p^l\Gamma(E_1).
$$
\end{rmk}
\end{itemize}
We check that the Assumptions \ref{ass:mainthmass}a)-e) are satisfied.
\begin{itemize}
\item[a)]By considerations similar to the ones for Lie algebroids, we choose the following topologies. 
\begin{itemize}
    \item[-] On $\mf g_{LnA}^0$, we pick the $C^\infty$-topology,
    \item[-] On $\mf g_{LnA}^1$, we pick the $C^{\max\{1,l\}}$-topology,
    \item[-] On $\mf g_{LnA}^2$, we pick the $C^l$-topology.
\end{itemize}
\item[b)] As $\partial = 0$, it is continuous.
\item[c)] Note that the continuity of the bracket is only needed in certain arity components, as the others disappear after taking the quotient. More precisely, the continuity of $[-,-]:\mf g^1_{LnA} \times \mf g^1_{LnA} \to \mf g_{LnA}^2$ is only needed when restricting the domain to  $$(\pi_{E_2^\ast}\oplus \pi_{S^2(E_1^\ast[-1]})\mf g^1_{LnA} \pi_{E_1^\ast} \oplus \pi_{E_1^\ast} \mf g^1_{LnA} \pi_{C^\infty(M)},$$
and the codomain to 
$$
 ( \pi_{E_{-2}^\ast}\oplus \pi_{S^2(E_1^\ast[-1])})\mf g^2_{LnA} \pi_{C^\infty(M)} \oplus \pi_{E_1^\ast \otimes E_2^\ast} \mf g^2_{LnA} \pi_{E_{-1}^\ast},
 $$
 where the $\pi$ denote projections onto the respective component in $C^\infty(E[1])$. In these degrees, continuity is guaranteed by the choices above.
\item[d)] The gauge action of $\mf g_{LnA}^0$ on $\mf g_{LnA}^1$ is similar to item d) in Section \ref{sssec:liealgk}. The degree 0 part is given by $\mf g_{LnA}^0 =  CDO\left(\bigoplus_{i = 1}^n E_i^\ast\right)$, and these act by degree 0 automorphisms of the graded vector bundle $E$, covering the flow of the symbol of the differential operator. Hence the analogue of Lemma \ref{lem:singpointmove} still holds, and a neighborhood of the origin of $\mf g_{LnA} ^0/\mf g_{LnA}^0(p,(1,l))$ corresponds to a neighborhood of $p \in M$.
\item[e)] Lemma \ref{lem:uniqueness} implies that the gauge action preserves Maurer-Cartan elements.
\end{itemize}
\subsubsection{Applying the main theorem}
Let $\mf g_{LnA} $ and $\mf g_{LnA}(p,(1,l))$ \nrt{be} as above in equation \eqref{eq:glalien} and Definition \ref{defn:sglalien1} respectively. Plugging this into the main theorem yields:
\begin{thm}\label{thm:lienalg1}
Let $(E,Q)$ be a Lie $n$-algebroid over $M$. Let $p \in M$ be a fixed point of order $(1,l)$ for $l \geq 0$, \nrt{as in Definition \ref{def:lienalg1l}}. Assume that $$H^1(\mf g_{LnA}/\mf g_{LnA}(p,(1,l)),\overline{[Q,-]}) = 0.$$ Then for every open neighborhood $U$ of $p \in M$, there exists a $C^{\max\{1,l\}}$-neighborhood $\mc U$ of $Q$ such that for any Lie $n$-algebroid structure $Q' \in \mc U$ there is a family $I$ in $U$ of fixed points of order $(1,l)$ of $Q'$ parametrized by an open neighborhood of $$0 \in H^0(\mf g_{LnA}/\mf g_{LnA}(p,(1,l)),\overline{[Q,-]}).$$
\end{thm}

\begin{rmk}
This result gives a more refined criterion for stability of the fixed point $p$: instead of looking at $E_1$ with its almost Lie algebroid structure as in Remark \ref{rmk:almostliealgd}ii), we additionally take into account the map $\ell_1:\Gamma(E_2) \to \Gamma(E_1)$. \\
The relation between Remark \ref{rmk:almostliealgd}ii) and Theorem \ref{thm:lienalg1} is as follows: the projection of a vector field $X\in \mf X^i(E[1])$ to the arity $i$ component restricted to $C^\infty(E_{1}[1])$ defines a map
$$
\text{res}:\mf g_{LnA} \to \mf X(E_1[1]),
$$
with $\text{res}(\mf g_{LnA}^j(p,(1,l))\subset \mf X_{p,1}^j(E_1[1])$ for $j = 0,1,2$. It therefore induces a map
$$
\overline{\text{res}}: \mf g_{LnA}^j/\mf g_{LnA}^j(p,(1,l)) \to \mf X^j(E_1[1])/\mf X^j_{p,1}(E_1[1]).
$$
Now given a Lie $n$-algebroid structure $Q \in \mf g^1_{LnA}(p,(1,l))$ with corresponding almost Lie algebroid structure $Q^{(1)}\in \mf X^1_{p,1}(E_1[1])$, the map $\overline{res}$ is compatible with the differentials $\overline{[Q,-]}$ and $\overline{[Q^{(1)},-]}$ of $\mf g_{LnA}/\mf g_{LnA}(p,(1,l))$ and $\mf X(E_1[1])$ respectively in degrees $0$ and $1$. Consequently, $\overline{res}$ descends to cohomology:
$$
H^1(\overline{\text{res}}): H^1(\mf g_{LnA}/\mf g_{LnA}(p,(1,l)),\overline{[Q,-]}) \to H^1(\mf X(E_1[1])/\mf X_{p,1}(E_1[1]), \overline{[Q^{(1)},-]}).
$$
For $l = 0$, $H^1(\overline{\text{res}})$ is injective, reflecting the fact that any Lie $n$-algebroid $(E,Q)$ has an underlying almost Lie algebroid structure $(E_1,Q^{(1)})$.\\
For $l>0$, the map is no longer injective, but vanishing of $H^1(\mf g_{LnA}/\mf g_{LnA}(p,(1,l)),\overline{[Q,-]})$ also guarantees the existence of fixed points of order $(1,l)$, which the vanishing of $$H^1(\mf X(E_1[1])/\mf X_{p,1}(E_1[1]),\overline{[Q^{(1)},-]})$$ does not.
\end{rmk}
\subsection{Higher order fixed points of Lie $n$-algebroids}
As for Lie algebroids, there are several examples of Lie $n$-algebroids, for which Theorem \ref{thm:lienalg1} fails in a trivial way because the anchor of the Lie $n$-algebroid vanishes up to higher order, causing components of $Q$ to vanish to first order. For this reason, we extend Theorem \ref{thm:lienalg1} to so called fixed points of order $(k,l)$, where $k$ and $l$ are integers. 
\nrt{
\begin{defn}\label{def:lienalgkl}
Let $(E,Q)$ be a Lie $n$-algebroid over $M$ with anchor $\rho:E_1\to TM$, unary bracket $\ell_1$, binary bracket $\ell_2$, and ternary bracket $\ell_3$.  Let $k\geq 1, 0\leq l\leq 2k-2$. We say $p\in M$ is a \emph{fixed point of order $(k,l)$ of $Q$} if for all sections $x,y,z\in \Gamma(E_1), e\in \Gamma(E_2)$, we have $$\rho(x) \in I_p^k \mf X(M), \ell_1(e)\in I_p^l\Gamma(E_1), \ell_2(x,y) \in I_p^{k-1}\Gamma(E_1), \ell_3(x,y,z)\in I_p^{2k-2-l}\Gamma(E_2)
.$$ 
\end{defn}
\begin{rmk}
    The condition on the anchor is what we are interested in, while the restriction on the binary bracket is imposed for the same reason as for higher order fixed points of Lie algebroids. By taking into account the order of vanishing of $\ell_1:E_1\to E_2$ at $p$, the condition on $\ell_3$ becomes more transparent: decomposing the identity $Q^2 = 0$, and considering only the weight 2 part restricted to $\Gamma(E_1^\ast[-1])$, we find the equation 
    $$
        \ell_2(x,\ell_2(y,z)) + \text{cycl.}(x,y,z) = \ell_1(\ell_3(x,y,z)).
    $$
    As the left hand side lies in $I_p^{2k-2}\Gamma(E_1)$, one way to ensure the same holds for the right hand side is to require $\ell_3(x,y,z) \in I_p^{2k-2-l}\Gamma(E_2)$.
\end{rmk}
}
\subsubsection{The ingredients}
We show that we are in the setting of Assumptions \ref{ass:mainthmass}.
\begin{itemize}
\item[i)] As we still work with Lie $n$-algebroid structures on a graded vector bundle $E = \bigoplus_{i = 1}^n E_i[i-1]$, take the graded Lie algebra $\mf g_{LnA}$ as in equation \eqref{eq:glalien}.
\item[ii)] We define the graded Lie subalgebra, corresponding to $p\in M$ being a fixed point of order $(k,l)$.

Let $p \in M$, $k\geq 0, 0\leq l\leq 2k-2$ integers.
\begin{defn}\label{def:lienalgklsubalg} Define
\begin{flalign*}
\mf g_{LnA}^0(p,(k,l)) :=\{ \delta\in {}^0\mf X^0(E[1])\mid& \,\delta(C^\infty(M))\subset I_p\}\\
\mf g_{LnA}^1(p,(k,l)) :=  \{\delta\in \mf X^1(E[1])\mid &\,\delta(C^\infty(M))\subset I_p^k\Gamma(E_1^\ast[-1] ),\\& \,\delta^{(1)}(\Gamma(E_1^\ast[-1] ))\subset I_p^{k-1}\Gamma(S^2(E_1^\ast[-1])),\\
& \,\delta^{(0)}(\Gamma(E_1^\ast[-1])) \subset I_p^l \Gamma(E_2^\ast[-2]),\\&\, \delta^{(2)}(\Gamma(E_{2}^\ast[-2]))\subset I_p^{2k-2-l}\Gamma(S^3(E_1^\ast[-1]))\}\\
\mf g^2_{LnA}(p,(k,l)) := \{\delta \in \mf X^2(E[1]) \mid& \,\delta^{(2)}(C^\infty(M)) \subset I_p^{2k-1}\Gamma(S^2(E_1^\ast[-1])),\\&\, \delta^{(1)}(C^\infty(M))\subset I_p^{k+l}\Gamma(E^\ast_2[-2]),\\&\, \delta^{(2)}(\Gamma(E_1^\ast[-1]))\subset I_p^{2k-2}\Gamma(S^3(E_1^\ast[-1]))\}.
\end{flalign*}
\end{defn}
Then:

\begin{lem}
The subspaces defined above satisfy \begin{align*}[\mf g_{LnA}^0(p,(k,l)),\mf g_{LnA}^i(p,(k,l))] &\subset \mf g_{LnA}^i(p,(k,l)),\\ [\mf g_{LnA}^1(p,(k,l)),\mf g_{LnA}^1(p,(k,l))]&\subset \mf g_{LnA}^2(p,(k,l))
\end{align*}
for $i = 0,1,2$.
\end{lem}
\nrt{
We see that $\mf g_{LnA}(p,(k,l))$ is indeed a graded Lie subalgebra (if as before, we set $\mf g_{LnA}^{\geq 3}(p,(k,l)) = \mf g_{LnA}^{\geq 3}$). Moreover, its Maurer-Cartan elements are precisely those Lie algebroid structures for which $p \in M$ is a fixed point of order $(k,l)$: 
\begin{lem}
Let $(E,Q)$ be a Lie $n$-algebroid over $M$. Let $k\geq 1$ be an integer, and $0\leq l\leq 2k-2$. A point $p\in M$ is a fixed point of order $(k,l)$ of $Q$ if and only if $Q\in \mf g^1_{LnA}(p,(k,l))$.
\end{lem}
}

\begin{rmk}
After picking a suitably small coordinate neighborhood of $p\in M$ such that $E$ is a trivial bundle, we can identify the cochain spaces of the complex $\mf g_{LnA}/\mf g_{LnA}(p,k)$ in the relevant degrees with the following:
\begin{align*}
    \mf g_{LnA}^0/\mf g_{LnA}^0(p,(k,l)) \cong& T_pM\\
    \mf g_{LnA}^1/\mf g_{LnA}^1(p,(k,l))\cong& J^{k-1}_p(E_1^\ast[-1] \otimes TM) \oplus J_p^{k-2} (S^2(E_1^\ast[-1]) \otimes E_1[1]) \\&\oplus J^{l-1}_p(E_2^\ast[-2]\otimes E_1[1]) \oplus J^{2k-l-3}_p(S^3(E_1^\ast[-1])\otimes E_2[2])\\
    \mf g_{LnA}^2/\mf g_{LnA}^2(p,(k,l)) \cong & J^{2k-2}(S^2(E_1^\ast[-1])\otimes TM) \oplus J^{2k-3}_p(S^3(E_1^\ast[-1])\otimes E_1[1])\\& \oplus J^{k+l-1}_p(E_2^\ast [-2]\otimes TM).
\end{align*}
\end{rmk}
\item[iii)] It follows that by restricting to a sufficiently small coordinate neighborhood of $p\in M$ over which $E$ trivializes as a graded vector bundle, we can choose splittings $$\sigma_i:\mf g_{LnA}^i/\mf g_{LnA}^i(p,(k,l)) \to \mf g_{LnA}^i$$ for $i = 0,1$.

\item[iv)] Now fix a Lie $n$\nrt{-}algebroid structure $Q\in \mf g_{LnA}^1(p,(k,l))$. 
\end{itemize}
We check that the data satisfies Assumptions \ref{ass:mainthmass}a)-d).
\begin{itemize}
\item[a)] We pick the following topologies.
\begin{itemize}
    \item[-] On $\mf g_{LnA}^0$, we pick the $C^\infty$-topology,
    \item[-] On $\mf g_{LnA}^1$, we pick the $C^{\max\{2k-1,k+l-1\}}$-topology,
    \item[-] On $\mf g_{LnA}^2$, we pick the $C^{\max\{2k-2,k+l-1\}}$-topology.
\end{itemize}
\item[b)] As $\partial = 0$, it is continuous.
\item[c)] As for fixed points of order $(1,l)$ it is only certain components of $\mf g_{LnA}^1$ and $\mf g^2_{LnA}$ we need continuity of the bracket in. More precisely, for the domain, we may restrict to
$$
(\pi_{E_2^\ast}\oplus \pi_{S^2(E_1^\ast[-1]})\mf g^1_{LnA} \pi_{E_1^\ast} \oplus \pi_{E_1^\ast} \mf g^1_{LnA} \pi_{C^\infty(M)} \oplus \pi_{S^3(E_{1}^\ast[-1])} \mf g_{LnA}^1 \pi_{E_2^\ast[-2]},
$$
while for the codomain, we need
$$
(\pi_{S^2(E_1^\ast[-1])} \oplus \pi_{E_2^\ast)}\mf g_{LnA}^2 \pi_{C^\infty(M)} \oplus \pi_{S^3(E_1^\ast[-1])} \mf g_{LnA}^2 \pi_{E_1^\ast}.
$$
The choices in a) guarantee the continuity.
\item[d)] The gauge action is the same as for fixed points of order $(1,l)$. In particular, a neighborhood of $0$ in $\mf g_{LnA}/\mf g_{LnA}(p,(k,l))$ corresponds to a neighborhood of $p\in M$.
\item[e)] Lemma \ref{lem:uniqueness} implies that the gauge action preserves Maurer-Cartan elements.
\end{itemize}
\subsubsection{Applying the main theorem}
Let $\mf g_{LnA}$ and $\mf g_{LnA}(p,(k,l))$ be as in equation \eqref{eq:glalien} and Definition \ref{def:lienalgklsubalg} respectively. Applying Theorem \ref{thm:mainthm}, we obtain:
\begin{thm}\label{thm:lienalgk}
Let $(E,Q)$ be a Lie $n$-algebroid over $M$. Let $p \in M$ be a fixed point of order $(k,l)$ for $k\geq 0$, $0\leq l \leq 2k-2$, \nrt{as in Definition \ref{def:lienalgkl}}. Assume that
$$H^1(\mf g_{LnA}/\mf g_{LnA}(p,(k,l)),\overline{[Q,-]}) = 0.$$
Then for every open neighborhood $U$ of $p \in M$, there exists a $C^{\max\{2k-1,k+l-1\}}$-neighborhood $\mc U$ of $Q \in \mf g_{LnA}$ such that for any Lie $n$-algebroid structure $Q' \in \mc U$ there is a family $I \subset U$ of fixed points of order $(k,l)$ of $Q'$ parametrized by an open neighborhood of $$0 \in H^0(\mf g_{LnA}/\mf g_{LnA}(p,(k,l)),\overline{[Q,-]}).$$
\end{thm}

\subsection{Examples}
In this section we compute the various cohomologies we encountered in some explicit examples.
\begin{exmp}[\nrt{Action of $\mf {gl}_2$ on the plane}]\label{gl2}
Let $M= \mb R^2$. Consider the Lie algebroid
$$
\mf{gl}_2 \times M,
$$
associated to the standard action on $M$. The origin $p$ is a fixed point of order 1 of the Lie algebroid\nrt{,} and we will show that it is stable for nearby Lie algebroid structures using Theorem \ref{thm:liealgd1}. The relevant cohomology in this case is given by 
$$
H^1_{CE}(\mf {gl}_2, T_pM).
$$
This cohomology vanishes as for any $\alpha:\mf {gl}_2 \to T_pM$, the cocycle condition reads
$$
\alpha([x,y]) = x\alpha(y) - y\alpha(x)
$$
for every $x,y \in \mf {gl}_2$. Plugging in $x = \text{id}$, we see that $\alpha = d_{CE}^\tau(\alpha(\text{id}))$.\\
Therefore, using Theorem \ref{thm:liealgd1}, the origin is a stable fixed point for nearby Lie algebroid structures. \\
But we can say more. There is a natural Lie 2-algebroid $(E,Q)$ associated to this action \cite[Example 3.33, Example 3.98]{univlinfty}, such that the origin is a fixed point of order (1,1). We will apply Theorem \ref{thm:lienalg1} to show that the fixed point is also stable for Lie $2$-algebroid structures close to this one. The only information we need is that $E_1$ is the trivial bundle with fiber $\mf {gl}_2$, such that the binary bracket agrees with the Lie bracket of the action Lie algebroid, $E_2$ is a trivial bundle with fiber $\mb R^2$ and that the complex
\[
\begin{tikzcd}
0 \arrow{r} & \Gamma(E_2) \arrow{r}{\ell_1} & \Gamma(E_1) \arrow{r}{\rho} & \rho(\Gamma(E_1)) \arrow{r}& 0
\end{tikzcd}
\]
is exact.\\
Indeed, following Remark \ref{rmk:expldesccohom}, we can compute the cohomology relevant to Theorem \ref{thm:lienalg1}. If we take a cocycle $(\alpha,\delta + I_p\Gamma(E_2^\ast \otimes E_1))\in \mf X^1(E[1])/\mf X^1_{p,(1,1)}(E[1])$, then the vanishing of $H^1_{CE}(\mf {gl}_2,T_pM)$ implies that $\alpha = d_{CE}^\tau(v)$ for some $v \in T_pM$, or more explicitly, we have $\alpha = -v(\rho)(p) \in (E_1)_p^\ast\otimes T_pM$.\\
The second cocycle condition in equation \eqref{eq:diff1ltot} now reads that
$$
-v(\rho)(\ell_1(e)) + \rho(\delta(e)) \in I_p^{2}\mf X(M)
$$
for every section $e \in \Gamma(E_2)$. If we now take $e$ to be a constant section, then we know in particular that
$$
-v(\rho)(\ell_1(e)) - \rho(v(\ell_1(e)) = -v(\rho(\ell_1(e))) = 0.
$$
So we find
$$
\rho(\delta(e) + v(\ell_1)(e)) \in I_p^{2}\mf X(M).
$$ 
As this means that the first order Taylor expansion of this expression vanishes, we look at the linear part of this:
$$
0 = (\rho(\delta(e) + v(\ell_1)(e)))^{(1)} = \rho^{(1)}(\delta(e)(p) + v(\ell_1)(e)(p)).
$$
However, note that $\rho$ has linear coefficient functions. This means that $\rho^{(1)} = \rho$, and it follows that the constant section $\delta(e)(p) + v(\ell_1(e))(p) \in \ker (\rho:\Gamma(E_1)\to \mf X(M))$. However, as $\ker(\rho) = \text{im}(\ell_1)$ at the level of sections, and all sections in the image of $\ell_1$ vanish at the origin, it follows that $\delta(p) + v(\ell_1)(p) = 0$, which means that $\delta = -v(\ell_1) +I_p\Gamma(E_2^\ast\otimes E_1)$, showing that $H^1(\mf g_{LnA}/\mf g_{LnA}(p,(1,1)),\overline{[Q,-]})$ vanishes (without needing the third cocycle condition!), and that $p$ is stable as a fixed point of order $(1,1)$ of the Lie 2-algebroid by Theorem \ref{thm:lienalg1}.
\end{exmp}
A similar example is given by the special linear subalgebra. 

\subsection{Fixed points of singular foliations}\label{sec:singfolapplic}
In this section we apply the results of the previous section to singular foliations. We work towards a cohomological stability criterion for certain singular foliations (\nrt{Theorem} \ref{thm:singfolstab}). We then apply this to obtain a formal rigidity criterion for foliations induced by a linear representation of a semisimple Lie algebra (Corollary \ref{cor:singfolrig})\nrt{, which is the final result of this section}. We start with a general lemma, which extends Example \ref{gl2}, giving a sufficient condition for when the cohomological assumption of Theorem \ref{thm:lienalg1} is satisfied. \nrt{In particular, the lemma will be used in Theorem \ref{thm:singfolstab}}.
\begin{lem}\label{lem:suffcrit}
Let $V$ be a finite-dimensional vector space, and let $(E = V\times \bigoplus_{i =1}^n \mf g_i[i-1],Q)$ be a Lie $n$-algebroid over $V$, such that the origin $p$ is a fixed point of order $(1,l)$. Assume that
\begin{itemize}
    \item [-] The coefficient functions of the anchor $\rho_Q:\Gamma(E_1)\to \mf X(V)$ are linear,
    \item[-] $H^1_{CE}(\mf g_1,V) = 0$,
    \item[-]$\ker(\rho_Q) = \text{im}(\ell_1^Q) $.
\end{itemize} Then $$H^1(\mf g_{LnA}/\mf g_{LnA}(p,(1,l)),\overline{[Q,-]})=0.$$
\end{lem}
\begin{proof} We use the description of $H^1(\mf g_{LnA}/\mf g_{LnA}(p,(1,l)),\overline{[Q,-]})$ as in Remark \ref{rmk:expldesccohom}.\\
The main ingredients of the proof are already present in Example \ref{gl2}. Given a cocycle $(\alpha,\delta+ I_p^{l}\Gamma(E_2^\ast\otimes E_1)) \in \mf g_{LnA}^1/\mf g_{LnA}^1(p,(1,l))$, we eventually find that 
$$
\rho(\delta(e) + v(\ell_1(e))) \in I_p^{l+1}\mf X(V)
$$
for every constant section $e\in\Gamma(E_2)$. Taking the $s$-homogeneous part of this expression for $s = 1,\dots,l-1$, we see that it is equal to
$$
0 = \rho^{(1)}(\delta^{(s-1)}(e)),
$$
and for the $l$-homogeneous part, we find
$$
0 = \rho^{(1)}(\delta^{(l-1)}(e) + v(\ell_1^{(l)})(e)).
$$
As $\rho^{(1)} = \rho$, and $\ker(\rho) = \text{im}(\ell_1)\subset I_p^l\Gamma(E_1)$, it follows that the $s$-homogeneous part of $\delta$ for $s = 0,\dots, l-2$ vanishes, while the $\nrt{(l-1)}$-homogeneous part is given by $-v(\ell_1^{(l)})$. As $v(\ell_1)\in I_p^{l-1}\Gamma(E_2^\ast\otimes E_1)$, we see that $\delta = -v(\ell_1) + I_p^{l}\Gamma(E_2^\ast\otimes E_1)$, which concludes the proof.
\end{proof}
\begin{rmk}\begin{itemize}
\item[]
\item[-]
    Note that the linearity of $\rho$ only comes into play when we want to use properties of the kernel of $\rho$. In fact a weaker condition which can replace linearity of $\rho$ and $\ker(\rho) = \text{im}(\ell_1)$ is that the linear part $\rho^{(1)}$ of $\rho$, which can be viewed as a map between sections of $E_1$ and vector fields on $V$ has kernel contained in $ I_p^l\Gamma(E_1)$.
\item[-] Note that if the isotropy Lie algebra $\mf g_1$ is semisimple, the condition on $H^1_{CE}(\mf g_1,V)$ is satisfied by Whitehead's first lemma.
\end{itemize}
\end{rmk}
Recall the following definition.
\begin{defn}
Let $M$ be a smooth manifold. A \emph{singular foliation} on $M$ is a subsheaf $\mc F\subset \mf X_M$ which is locally finitely generated, and involutive. Denote a foliated manifold by $(M,\mc F)$.\\
We say that $p\in M$ is a fixed point of $\mc F$ if $\mc F\subset I_p\mf X_M$.
\end{defn}
Just as for Lie algebroids, associated to a fixed point $p\in M$ of a singular foliation $\mc F$, there is a Lie algebra $\mf g_p$ and a representation of $\mf g_p$ on $T_pM$:
\begin{lem} [\cite{holgpd},\cite{AnZa12}] Let $(M,\mc F)$ be a foliated manifold and $p\in M$ a fixed point of $\mc F$. 
\begin{itemize}
    \item [i)] Let $\mc F_p$ denote the stalk of $\mc F$ at $p$. Then the induced Lie bracket on $\mc F_p$ descends to the finite-dimensional vector space
    $$
        \mf g_p := \mc F_p/I_p\mc F_p.
    $$
    \item[ii)] The map
    $$
        \tau: \mf g_p\to \End(T_pM)
    $$
    given by
    $$
    \tau(x)(v) = [\tilde{x},\tilde{v}](p)
    $$
    is a well-defined representation of $\mf g_p$ called the \emph{linear holonomy representation}, where $x\in \mf g_p, v\in T_pM$, and $\tilde{x},\tilde{v}$ are extensions to elements of $\mc F$ and $ \mf X(M)$ respectively.
\end{itemize}
\end{lem}
\begin{rmk}
Note that a Lie $n$-algebroid $(E,Q)$ induces a singular foliation on the base manifold $M$, which is the image of the anchor map at the level of sections.
\end{rmk}
The following definition was proposed by Camille Laurent-Gengoux and Sylvain Lavau \cite{defsingfol}.
\begin{defn}
Let $(M,\mc F)$ be a foliated manifold. \nrt{An \emph{isomodular deformation}} of $\mc F$ is a singular foliation $\mc F_U$ on $M\times U$, where $0\in U \subset \mathbb{R}^n$ is an open subset, such that
\begin{itemize}
    \item[a)] For every $p \in U$, $\mc F^p$ is tangent to $M\times\{p\}$, where $\mc F^p$ is the restriction of $\mc F_U$ to $M\times \{p\}$,
    \item[b)] $\mc F^0 = \mc F$,
\end{itemize}
together with an isomorphism of $C^\infty_{M\times U}$-modules $\phi: C^\infty_{M\times U}\otimes_{C^\infty_M} \mc F \to \mc F_U$.
\end{defn}
This allows us to define a notion of stability of fixed points of a singular foliation.
\begin{defn}
Let $(M,\mc F)$ be a foliated manifold, and let $p \in M$ be a zero-dimensional leaf of $\mc F$. We say that $p$ is \textit{stable} if for every isomodular deformation $(M\times U, \mc F_U)$, for every neighborhood $p \in V\subset M$, there is a neighborhood $0\in W\subset U$, such that $\mc F^q$ has a zero-dimensional leaf in $V$ for $q\in W$.
\end{defn}
The following lemma was proven by Camille Laurent-Gengoux and Sylvain Lavau \cite{defsingfol}. We formulate a weaker version and include the proof in the appendix for completeness.
\begin{lem}\label{lem:resexact} 
Let $(V,\mc F)$ be a vector space equipped with a \emph{linear} foliation, and $(V\times U, \mc F_U)$ an isomodular (not necessarily linear) deformation. Then there exists a geometric resolution (see \cite{univlinfty}) of $\mc F$
\begin{equation}\label{diag:geomresF}
\begin{tikzcd}
0\arrow{r}&  \Gamma(E_n) \arrow{r}&\dots \arrow{r}& \Gamma(E_1)\arrow{r}&\mc F\arrow{r}& 0,
\end{tikzcd}
\end{equation}
 such that the differential has polynomial coefficient functions.
 Moreover, there exists a Lie $n$-algebroid structure $Q$ on $\oplus_{i=1}^n p^\ast E_i$, where $p: V\times U \to V$ is the projection, with the following properties:
\begin{itemize}
    \item[-] The unary bracket extends the complex \eqref{diag:geomresF},
    \item[-] $Q$ induces the foliation $\mc F_U$,
    \item[-] The sequence of $C^\infty_{M\times U}$-modules 
    \[
\begin{tikzcd}
0\arrow{r}&  \Gamma(p^\ast E_n) \arrow{r}&\dots \arrow{r}& \Gamma(p^\ast E_1)\arrow{r}&\mc F_U\arrow{r}& 0,
\end{tikzcd}
\]
is exact,
    \item[-] The restriction of $Q$ to $V\times \{0\}$ is a Lie $n$-algebroid structure on $E$ inducing $\mc F$.
\end{itemize}
\end{lem}
Using the results of this section, we obtain
a stability result for linear foliations.
\nrt{\begin{thm}\label{thm:singfolstab}
Let $(V,\mc F)$ be a vector space, equipped with a linear foliation. Assume that $$H^1_{CE}(\mf g,V) = 0,$$ where $\mf g$ is the isotropy Lie algebra of $\mc F$ at $0$ acting by the linear holonomy representation. Then the origin is a stable fixed point for all isomodular deformations.
\end{thm}}
\begin{proof}
Let $(V\times U,\mc F_U)$ be an isomodular deformation of $\mc F$, and pick a Lie $n$-algebroid $(p^\ast E,Q)$ inducing $\mc F_U$ as in Lemma \ref{lem:resexact}. By \cite[Theorem 2.3.5]{univlinfty}, we may assume that the differential on the complex vanishes in $0$ up to finite order $l>0$. In this case, the fiber $(E_1)_0$ over $0 \in V$ has the same dimension as the isotropy Lie algebra by \cite[Proposition 4.14]{univlinfty}.\\
Denote by $\rho:\Gamma(E_1) \to \mf X(V)$ the map inducing the foliation. 
\begin{claim}We may assume that $\rho$ has linear coefficient functions.
\end{claim}
\begin{proof}[Proof of claim]
To see this, pick linear generators $\{X_i\}_{i=1}^r$ for $\mc F$ which are linearly independent over $\mb R$, and consider their images $\{e_i \}_{i=1}^r$ in $\mf g$. It is clear that the $e_i$ generate $\mf g$. However, they are also linearly independent: if $a_i\in \mb R$ are such that 
$$
\sum_{i =1}^r a_ie_i = 0 \in \mf g,
$$
then
$$
\sum_{i=1}^r a_i X_i \in I_0\mc F\subset  I_0^2\mf X(M).
$$
As the $X_i$ are linear and were assumed to be linearly independent over $\mb R$, it follows that the $a_i$ must be zero. \\
Pick preimages $s_i$ for $X_i$ under $\rho$. The 
diagram
\[
\begin{tikzcd}
\Gamma(E_1) \arrow{r}{\rho}\arrow{d}& \mc F\arrow{d}\\
\Gamma(E_1)/I_0\Gamma(E_1) \arrow{r}{\overline{\rho}} & \mc F/I_0\mc F
\end{tikzcd}
\]
commutes, and the bottom row is a map from $(E_1)_0$ to $\mf g$, sending $s_i(0)$ to $e_i$. This shows that $s_i$ form a local frame around $0$, and $\rho(s_i) = X_i$ is a linear vector field, which proves the claim.
\end{proof}
By the choice of Lie $n$-algebroid $(p^\ast E,Q)$ of $\mc F_U$, the restriction $(E,Q_0)$ to $V\times \{0\}$ has a fixed point of order $(1,l)$, and satisfies the assumptions of Lemma \ref{lem:suffcrit}. So given a neighborhood $W\subset V$ of the origin, there exists a neighborhood $\mc U$ of $Q_0$ in the space of Lie $n$-algebroid structures on $E$ such that all $Q'\in \mc U$ have a fixed point $q\in W$.\\
Finally, as the map
$$
U \to \{\text{Lie }n\text{-algebroid structures on $E$}\}
$$
given by
$$
p \mapsto Q_p
$$
is continuous, the result follows. 
\end{proof}
In particular, we have:
\begin{cor}\label{cor:semisimplestable}
Let $(V,\mc F)$ be a vector space equipped with a foliation that has linear generators, \nrt{which implies} that all vector fields in $\mc F$ vanish at the origin. If the isotropy Lie algebra $\mf g$ of $\mc F$ in $0\in V$ is semisimple, then the origin is a stable fixed point for all isomodular deformations. 
\end{cor}
\begin{proof}
For semisimple Lie algebras $\mf g$, $H^1_{CE}(\mf g,-)$ is identically zero for finite-dimensional representations by Whitehead's first lemma. 
\end{proof}
\begin{rmk}\begin{itemize}
\item[]
\item[-] Under the assumption that the foliation $\mc F$ we start with is linear, we obtained a simplified criterion for stability of the origin, depending only on the foliation, and not on the chosen Lie $n$-algebroid inducing it. 
\item[-] If $\mc F$ is arbitrary, with a fixed point, and $\mc F_U$ is a deformation which admits a geometric resolution as in \cite[Definition 2.1] {univlinfty}, hence a Lie $n$-algebroid inducing $\mc F_U$, then restricting to $\mc F_0$ one gets a Lie $n$-algebroid inducing $\mc F$. Now Theorem \ref{thm:lienalg1} can be applied to the latter Lie $n$-algebroid if $H^1(\mf g_{LnA}/\mf g_{LnA}(p,(1,l)),\overline{[Q,-]}) = 0$.
\end{itemize}
\end{rmk}
If $(M,\mc F)$ is a foliated manifold, and $p \in M$ a fixed point stable under isomodular deformations, let $(M\times U, \mc F_U)$ be such an isomodular deformation. In particular, for $W\subset U$ as in \nrt{Theorem \ref{thm:singfolstab}}, $q\in W$ implies that $\mc F_q$ has a fixed point $p'$ in $M$. This gives rise to two questions.
\begin{itemize}
    \item[-] Can we describe the isotropy Lie algebra of $\mc F_q$ in $p'$?
    \item[-] If so, what can we say about the linear holonomy representation?
\end{itemize}
The following proposition provides an answer for both of these questions when $\mf g$ is semisimple.
\begin{prop}\label{prop:firstorderrigid}
Let $(V,\mc F)$ be a vector space equipped with a foliation that has linear generators, such that the isotropy Lie algebra is semisimple, \nrt{which implies} that the origin $0 = p\in V$ is a stable fixed point \nrt{by Corollary \ref{cor:semisimplestable}}.\\
Let $(V\times U, \mc F_U)$ be an isomodular deformation, such that for any $q\in U$, the foliation $\mc F_q$ has a fixed point $p'(q)\in V$. \\
Then for a possibly smaller neighborhood $W\subset U$, the isotropy Lie algebra at $p'(q)$ is isomorphic to $\mf g$. Moreover, the linear holonomy representations are isomorphic.
\end{prop}
\begin{proof}
As $\mc F_U$ is \nrt{an isomodular deformation}, let $(p^\ast E,Q)$ be a Lie $n$-algebroid inducing $\mc F_U$ as in Lemma \ref{lem:resexact}, for which the unary bracket vanishes at $(p,0)\in V\times U$ up to order $l\in \mb Z_{>0}$.\\
As $l>0$, by \cite[Proposition 4.14]{univlinfty}, $((E_1)_{p'},(\ell_2)_{(p',q)})$ is exactly the isotropy Lie algebra of $\mc F^q$ in $p'$, where $\ell_2$ is the binary bracket of $(p^\ast E,Q)$. As we may assume $E_1$ is a trivial bundle we denote $(E_1)_p$ by $\mf h$, and we consider the map
$$
V\times U\to \wedge^2 \mf h^\ast \otimes \mf h
$$
given by
$$
(v,q) \mapsto (\ell_2)_{(v,q)},
$$
where the right hand side should be seen as the map which extends $x_1,x_2\in \mf h$ to a constant section, applies the binary bracket of the Lie $n$-algebroid, and then evaluates it on $(v,q)\in V\times U$. As semisimple Lie algebras are rigid by \cite[Theorem 7.1]{defliealg} and Whitehead's second lemma, there is a neighborhood $O$ of $(\ell_2)_{(p,0)}$, such that every Lie algebra structure in $O$ is isomorphic to $(\ell_2)_{(p,0)}$. Hence there is a neighborhood $D$ of $(p,0)$ such that for every pair $(p',q) \in D$ such that $p'$ is a fixed point of $\mc F^q$, its isotropy Lie algebra is isomorphic to $\mf g$.\\
For the assertion about the linear holonomy representation the proof is analogous: we note that the linear holonomy representation is the linearization of the anchor, and that representations of semisimple Lie algebras are rigid by \cite[Theorem A]{bams/1183528646} and Whitehead's first lemma.
\end{proof}
Proposition \ref{prop:firstorderrigid} says something about the first order approximation. If $\mc F_U$ is a deformation such that for every $q\in U$, $\mc F^q$ is linearizable, this yields a rigidity result for such a deformation. \\
There is a sufficient condition for a foliation with semisimple isotropy to be formally linearizable around a fixed point.
\begin{cor}\label{cor:singfolrig}
Let $(V,\mc F)$ be a vector space equipped with a foliation with linear generators, such that the isotropy Lie algebra $\mf g$ is semisimple. Let $(V\times U,\mc F_U)$ be an isomodular deformation such that $\mc F^q$ has analytic generators for every $q\in U$. For some neighborhood $W\subset U$ of the origin, for every $q\in W$, there is a formal diffeomorphism of $\phi_q:V\to V$ such that $\phi_q^\ast \mc F^q = \mc F$.
\end{cor}
\begin{proof}
It suffices to show that a foliation with analytic generators with semisimple isotropy in a fixed point is formally linearizable. This follows from \nrt{\cite[Theorem 2.2]{Cerveau1979}}: to apply this result, we need to show that the linear holonomy representation $\mf g\to \End(V)$ is faithful. The kernel is an ideal, hence it is itself a semisimple Lie algebra. However, \cite[Theorem 1.10]{laurentgengoux2020neighbourhood} implies that the kernel is nilpotent. As \nrt{any simultaneously nilpotent and semisimple Lie algebra} is trivial, this shows that the linear holonomy representation is faithful.
\end{proof}

\section{Stability under additional structure}\label{sec:additionalstr}
In the previous sections we have considered Lie $n$-algebroids, and applied Theorem \ref{thm:mainthm} to give a sufficient condition for when a fixed point of some type of a Lie $n$-algebroid is stable. \\
Now suppose we are given a Lie algebroid with some additional structure. A natural question we can ask is if we can refine the criterion when we only look at Lie algebroid structures which also have this additional structure. For instance, given a Lie algebroid structure on $T^\ast M$, we can require the Lie algebroid differential on $\mf X^\bullet(M)$ to be a derivation of the Schouten-Nijenhuis bracket. Is there a theorem similar to Theorem \ref{thm:liealgdk} when we only allow Lie algebroid structures on $T^\ast M$ which are in addition derivations of the Schouten bracket? Of course, it is known that such Lie algebroid structures are precisely the Poisson structures on $M$, so \cite{CrFe} and \cite{dufour2005stability}  contain results on it. \nrt{A pair of Lie algebroid structures on dual vector bundles with this property is called a \emph{Lie bialgebroid}:}
\nrt{\begin{defn}\label{def:liebialgd}
    Let $(A,d_A)$ be a Lie algebroid over a manifold $M$. Let $A^\ast$ be the dual vector bundle to $A$, and let $d_{A^\ast} \in \mf X^1(A^\ast[1])$ be a Lie algebroid structure on $A^\ast$. Then $((A,d_A),(A^\ast,d_{A^\ast}))$ is a \emph{Lie bialgebroid} if $d_{A^\ast}$ is a degree 1 derivation of the graded Lie algebra $(C^\infty(A^\ast[1])[1],0,[-,-]_A)$.
\end{defn}}

In the first subsection, we will address the question above by fixing a Lie algebroid $(A,d_A)$, and apply Theorem \ref{thm:mainthm} to give a stability criterion for fixed points of a Lie algebroid structure $d_{A^\ast}$ on $A^\ast$ such that $((A,d_A),(A^\ast,d_{A^\ast}))$ is a Lie bialgebroid. We then apply this result to obtain the following:
\begin{itemize}
\item[-] First, by taking $A= TM$ with its standard Lie algebroid structure $d_A = d_{dR}$, recover the result from \cite{dufour2005stability} for Poisson manifolds (Corollary \ref{cor:poissonasliebi}).
\item[-]
When $Z \subset M$ is a hypersurface, we let $A = {}^\flat TM$ be the b-tangent bundle. In this case we obtain a result for zeros of self-commuting b-bivector fields (Theorems \ref{cor:bpoissonasliebi} and \ref{thm:bpoissonliebi+}). 
\item[-]
When $N:TM\to TM$ is a Nijenhuis tensor, we obtain a result for Lie algebroid structures near a Poisson structure compatible with $N$ (Theorem \ref{thm:PNLie}), which we then refine to a result dealing only with Poisson-Nijenhuis structures (Theorem \ref{thm:PN}). As application of this, we obtain a stability result for fixed points of holomorphic Poisson structures \nrt{(Corollary \ref{cor:holpoisson})}.
\end{itemize}
\par In the second subsection, we look at Courant algebroids and formulate a stability theorem for fixed points of Courant algebroid structures on a given vector bundle $E$ with fixed non-degenerate metric $\pairing --$ (Theorem \ref{thm:courant}).

Finally, in the last subsection we consider Dirac structures inside a split Courant algebroid $A\oplus A^\ast$. Under the assumption that both $A$ and $A^\ast$ are Dirac structures, we apply the main theorem to obtain a sufficient condition for fixed points of the Dirac structure $A$. There is a difference from all the results obtained so far: the theorem does not only guarantee a fixed point near the given one, but in fact guarantees that the fixed point will lie on the same $A^\ast$-leaf as the original one (Theorem \ref{thm:Dirac}).
% We then look at graded manifolds which come with a symplectic structure of some positive degree. For degree 1, this will again recover the result on Poisson manifolds from \cite{dufour2005stability} and for degree 2, we obtain a stability result for singular points of Courant algebroids.
\subsection{Higher order fixed points of Lie bialgebroids}\label{sec:liebialg}
Throughout this section, let $(A,d_A)$ be a fixed Lie algebroid over the manifold $M$, with anchor $\rho_A$ and bracket $[-,-]_A$. We apply Theorem \ref{thm:mainthm}, and find a sufficient condition for when a fixed point of $(A^\ast,d_{A^\ast})$ is stable for nearby Lie algebroid structures on $A^\ast$ which are compatible with $d_A$, as in Definition \ref{def:liebialgd}.
\subsubsection{Lie bialgebroids}
We embed the Lie algebra of vector fields on $A^\ast[1]$ into a bigger graded Lie algebra, in which compatibility with the Lie algebroid structure on $A$ can be formulated as a commutation condition. The details of this procedure can be found in Section 3 of  \cite{roytenberg1999courant}, but we describe the outline.

Given the graded manifold $A^\ast[1]$, we can consider the graded manifold $T^\ast[2] A^\ast[1]$. As ordinary cotangent bundles, this graded manifold carries a symplectic form, which has degree 2 in this case. To avoid going into details about this, we will work with the corresponding \emph{dual} structure, which is a 2-Poisson algebra structure on $C^\infty(T^\ast[2]A^\ast[1])$ as in \cite{cattaneo2018graded} and can be described explicitly by its properties. The additional property this Poisson bracket has as it comes from a symplectic form, is that there is a Darboux-like theorem for the Poisson bracket \cite{cueca2019geometry}.\\ 
In the following proposition, we summarise the properties of this graded manifold that we use:
\begin{prop}[\cite{roytenberg1999courant,ANTUNES201366,cueca2019geometry}]\label{prop:gradedcotangentproperties}
The 2-Poisson algebra $C^\infty(T^\ast[2] A^\ast[1])$ satisfies the following properties. 
\begin{itemize}
    \item[i)] $C^\infty(T^\ast[2]A^\ast[1])$ is a $(\mb Z_{\geq 0} \times \mb Z_{\geq 0})$-bigraded algebra: as
    $$C^\infty(T^\ast[2]A^\ast[1]) = S_{C^\infty(A^\ast[1])}(\mf X(A^\ast[1])[-2]),$$ the bidegree $(p,q)$-part is given by 
    $$
    S^p_{C^\infty(A^\ast[1])}(\mf X(A^\ast[1])[-2])^{p+q}
    $$
     \item[ii)] The Poisson bracket is non-degenerate. Moreover, the Poisson bracket is homogeneous of bidegree $(-1,-1)$ with respect to the bigrading.
     \item[iii)] The Poisson bracket extends the bracket of vector fields on $A^\ast[-1]$: any $\delta \in \mf X(A^\ast[1])$ gives rise to $f_\delta\in \mf X(A^\ast[1])[-2]$ of bidegree $(1,|\delta|+1)$ \nrt{by the inclusion of $\mf X(A^\ast[1])$}, compatible with the Lie bracket.
    \item[iv)] The Lie algebroid structure $d_A$ on $A$ gives rise to an element $\Pi_{d_A}\in C^\infty(T^\ast[2]A^\ast[1])$ of bidegree $(2,1)$ as follows: as $f_{d_A}$ defines a self-commuting element of $C^\infty(T^\ast[2] A[1])$ and there is a canonical symplectomorphism $T^\ast[2]A^\ast[1]\cong T^\ast[2]A[1]$, we get an element $\Pi_{d_A} \in C^\infty(T^\ast[2] A^\ast[1])$ satisfying \nrt{$\{\Pi_{d_A},\Pi_{d_A}\} = 0$.} 
    \item[v)] The pair $((A,d_A),(A^\ast,d_{A^\ast}))$ is a Lie bialgebroid if and only if $\{f_{d_{A^\ast}},\Pi_{d_A}\} = 0$.
    \item[vi)] The action of $d_{A^\ast}$ on $g\in C^\infty(A^\ast[1])=\Gamma(S(A[-1]))$ is given by
    $$
    d_{A^\ast}(g) = \{f_{d_{A^\ast}},g\}.
    $$
    \item[vii)] $\Pi_{d_A}$ encodes the Schouten-Nijenhuis extension of the bracket on $\Gamma(A)$ to $C^\infty(A^\ast[1])$: given two homogeneous functions $f,g \in C^\infty(A^\ast[1])$, 
    $$
    [f,g]_A = (-1)^{|f|-1} \{\{\Pi_{d_A},f\},g\}.
    $$
\end{itemize}
\end{prop}
\begin{rmk}
In classical terms, the compatibility condition between $d_{A^\ast}$ and $d_A$ can be expressed as follows: for $f,g \in C^\infty(M)$, $X,Y \in \Gamma(A)$,
\begin{align*}
[d_{A^\ast}(f),g]_A &=  [f,d_{A^\ast}(g)]_A ,\\
d_{A^\ast}[X,f]_A &= [d_{A^\ast}(X),f]_A + [X,d_{A^\ast}(f)]_A,\\
d_{A^\ast}[X,Y]_A &= [d_{A^\ast}(X),Y]_A + [X,d_{A^\ast}(Y)]_A,
\end{align*}
as a consequence of the Jacobi identity for $\{-,-\}$.
\end{rmk}
We will apply Theorem \ref{thm:mainthm} to formulate a stability criterion for fixed points of $(A^\ast, d_{A^\ast})$ (as in Definition \ref{def:liealgk}), only taking into account Lie algebroid structures which are compatible with $d_A$.
\subsubsection{The ingredients}
We check that we are in the setting of Assumptions \ref{ass:mainthmass}. Fix the Lie algebroid $(A,d_A)$.
\begin{itemize}
\item[i)]
Proposition \ref{prop:gradedcotangentproperties} now tells us that if we set \begin{align}\label{eq:liebialgk}
\mf g_{LbA} := C^\infty(T^\ast[2] A^\ast[1])^{(\geq 1,\geq 1)}[1,1],
\end{align}
we have a differential bigraded Lie algebra, with differential $\{\Pi_{d_A},-\}$ of bidegree $(1,0)$. Lie algebroid structures $d_{A^\ast}$ on $A^\ast$ such that $((A,d_A), (A^\ast,d_{A^\ast})) $ is a Lie bialgebroid are Maurer-Cartan elements of bidegree $(0,1)$. Note that because of the bigrading, the Maurer-Cartan equation
for $f_{d_{A^\ast}}\in C^\infty(T^\ast[2]A^\ast[1])[1,1]$
$$
\{\Pi_{d_A},f_{d_{A^\ast}}\} + \frac{1}{2}\{f_{d_{A^\ast}},f_{d_{A^\ast}}\}=0
$$
breaks up into two components: the bidegree $(1,1)$-part, which is $$\{\Pi_{d_A},f_{d_{A^\ast}}\} = 0,$$ and the bidegree $(0,2)$-part, which is
$$
\{f_{d_{A^\ast}},f_{d_{A^\ast}}\} = 0.
$$
We would like to apply the main theorem in this setting, but as it is formulated, it is not clear it can be applied, as we are \nrt{only interested in Maurer-Cartan elements} of bidegree $(0,1)$. The naive thing to do here would be to take the graded Lie algebra $\mf g_{LbA}$ by forgetting the bigrading and only caring about the total degree (where we restrict to those where the unshifted bidegree is at least $(1,1)$). We show that this works.

\item[ii)]
Let $p\in M$. We now define the Lie subalgebra $\mf g_{LbA}(p,k)$ corresponding to fixed points of order $k\geq\nrt{1}$. In bidegrees $(0,i)$ for $i = 0,1,2$ we see that $\mf g_{LbA}^{(0,i)} = \mf X(A^\ast[1])^i$. So we set 
\nrt{\begin{equation} \label{eq:liebialgksubalg1}
\mf g_{LbA}^{(0,i)}(p,k):= \mf X(A^\ast[1])^i_{p,i(k-1)+1},
\end{equation}}
as in Section \ref{sec:exampleLiealg}. The compatibility with the Lie algebroid structure on $A$ will be an extra condition on the cocycles in the quotient complex. Now set $\mf g_{LbA}^{(1,0)}(p,k) =\mf g_{LbA}^{(1,0)},\mf g_{LbA}^{(2,0)}(p,k) = \mf g_{LbA}^{(2,0)}$, and
\nrt{\begin{align}\label{eq:liebialgksubalg2}
\mf g_{LbA}^{(1,1)}(p,k) = \left\{\Pi\in \mf g_{LbA}^{(1,1)} \mid \{\{\Pi,f\},g\} \in I_p^{k-j}C^\infty(A^\ast[1])^j \forall f,g \in C^\infty(A^\ast[1])^{\leq 1},j = |f| + |g|\right\}.
\end{align}}

This defines a differential graded Lie subalgebra:
\begin{lem}
For $i = 0,1,2$, let $\mf g_{LbA}^i(p,k) = \bigoplus_{j = 0}^{i}\mf g_{LbA}^{(j,i-j)}(p,k)$.
Then
\begin{align*}
\{\mf g_{LbA}^0(p,k),\mf g_{LbA}^i(p,k)\} &\subset \mf g_{LbA}^i(p,k), \\\{\mf g_{LbA}^1(p,k),\mf g_{LbA}^1(p,k)\} &\subset \mf g_{LbA}^2(p,k).
\end{align*}
\end{lem}
\begin{rmk}
The only space \nrt{we have not encountered yet} is $\mf g_{LbA}^{(1,1)}/\mf g_{LbA}^{(1,1)}(p,k)$. By picking an open neighborhood of $p \in M$ over which $A$ trivializes, we see that
\begin{align*}
\mf g_{LbA}^{(1,1)}/\mf g_{LbA}^{(1,1)}(p,k) \cong J_p^{k-1}(S^2(TM)) &\oplus J_p^{k-2}(\Hom(A\otimes T^\ast M, A)) \\&\oplus J_p^{k-3}(\Hom(S^2(A[1]),S^2(A[1]))).
\end{align*}
\end{rmk}
\item[iii)] The splittings exist for the same reason as for Lie algebroids.
\item[iv)] Now pick a Lie algebroid structure $d_{A^\ast}$ on $A^\ast$ such that $((A,d_A),(A^\ast,d_{A^\ast}))$ is a Lie bialgebroid structure, and $p\in M$ is a fixed point of order $k$ \nrt{of $d_{A^\ast}$}.
\end{itemize}
We now check that the data satisfies the Assumptions \ref{ass:mainthmass}.
\begin{itemize}
\item[a)] We pick the following topologies:
\begin{itemize}
    \item[-] On $\mf g^0_{LbA}$, we pick the $C^\infty$-topology,
    \item[-] On $\mf g^1_{LbA}$, we pick the $C^{2k-1}$-topology,
    \item[-] On $\mf g^2_{LbA}$, we pick the $C^{2k-2}$-topology.
\end{itemize}
\item[b)] As the $(2k-2)$-jet of $\{\Pi_A,X\}$ depends linearly on the $(2k-1)$-jet of $X$, $\{\Pi_A,-\}$ is continuous.
\item[c)] The continuity of $[-,-]:\mf g^1_{LbA}\times \mf g_{LbA}^1 \to \mf g_{LbA}^2$ holds for the same reasons as in Section \ref{sec:exampleLiealg}.
\item[d)]
It remains to understand the gauge action. It is clear that $\mf g_{LbA}^{(0,0)} \cong CDO(A)$. Take $X\in \mf g_{LbA}^{(0,0)}$ and $(Q,\Pi)\in \mf g_{LbA}^1 = \mf g_{LbA}^{(0,1)} \oplus \mf g_{LbA}^{(1,0)}$.\\
The gauge equation then becomes
$$
\frac{d}{dt}(Q_t,\Pi_t) = (\{X,Q_t\}, \{X,\Pi_t\} -\{\Pi_A,X\}), \quad(Q_0,\Pi_0) = (Q,\Pi).
$$
For the purposes of interpreting the main theorem it is sufficient to consider only the first component: as $\mf g_{LbA}^{(1,0)}(p,k) =\mf g_{LbA}^{(1,0)}$\nrt{,} the second component of the conclusion holds no information. The first component however, is simply the same as the gauge action for ordinary Lie algebroids: the take-away message is that $$(Q_t,\Pi_t)\in \mf g_{LbA}^1(p,k) \iff \phi_{ t}^{\sigma(X)}(p) \text{ is a fixed point of $Q_t$}.$$ Here $\sigma(X)\in \mf X(M)$ is the symbol of the differential operator $X$. Consequently, an open neighborhood of $0 \in \mf g_{LbA}^0/\mf g_{LbA}^0(p,k) \cong T_pM$ corresponds to an open neighborhood of $p \in M$.
\item[e)] Lemma \ref{lem:uniqueness} implies that the gauge action preserves Maurer-Cartan elements.

\end{itemize}
\begin{rmk}\label{rmk:concrcohom}

\begin{itemize}
    \item[]
\item[i)] When $k = 1$, the cohomology group $H^1\left(\mf g_{LbA}/\mf g_{LbA}(p,1), \overline{\{\Pi_{d_A},-\}+\{f_{d_{A^\ast}}\}}\right)$ is a subspace of $H^1_{CE}( A_p^\ast,T_pM)$, which is the cohomology group appearing in Theorem \ref{thm:liealgd1}: the coboundaries remain unchanged, while the cocycles are the Chevalley-Eilenberg cocycles $\alpha \in A_p \otimes T_pM$, for which $(\rho_{A,p}\otimes \text{id})(\alpha) \in T_p M\otimes T_p M$ is skew-symmetric. For general $k\geq 1$, there is an injective map
$$
H^1\left(\mf g_{LbA}/\mf g_{LbA}(p,k),\overline{\{\Pi_{d_A},-\}+\{f_{d_{A^\ast}},-\}}\right) \hookrightarrow H^1\left(\mf g_{LA}/\mf g_{LA}(p,k),\overline{\{{d_{A^\ast,-}}\}}\right),
$$
but for $k\geq 1$ the image is not as simple to describe. This is what we should expect: indeed, if a fixed point is stable for \emph{all} Lie algebroid structures, it should in particular be stable for a subclass of Lie algebroid structures.
\item[ii)] The approach taken here to obtain the right graded Lie algebra might seem indirect, and a seemingly more direct approach would be to replace $(C^\infty(T^\ast[2]A^\ast[1]))^{(1,2)}$ by the kernel of the vertical differential $\{\Pi_{d_A},-\}$. In order to make sure this is well-defined, one would also have to restrict the functions in bidegree (1,1) to those which are in the kernel of the vertical differential. As these are not given by the sections of some vector bundle over $M$ in general, we would have less control over the objects we work with.
\end{itemize}
\end{rmk}

\subsubsection{Applying the main theorem}
Applying the main theorem to $\mf g = \mf g_{LbA}$, $\mf h = \mf g_{LbA}(p,k)$ \nrt{as in equations \eqref{eq:liebialgk}, \eqref{eq:liebialgksubalg1}, \eqref{eq:liebialgksubalg2}} yields:
\begin{thm}\label{thm:liebialg}
Let $(A,d_A)$ be a Lie algebroid over $M$, and $(A^\ast,d_{A^\ast})$ a Lie algebroid defined on the dual vector bundle such that $((A,d_A),(A^\ast,d_{A^\ast}))$ is a Lie bialgebroid. Let $p \in M$ be a fixed point of order $k$ for $k \geq 0$ of $d_{A^\ast}$, \nrt{as in Definition \ref{def:liealgk}}. Assume that
$$
H^1\left(\mf g_{LbA}/\mf g_{LbA}(p,k), \overline{\{\Pi_{d_A},-\}+\{f_{d_{A^\ast}},-\}}\right) = 0,
$$ 
where $\Pi_{d_A}$, $f_{d_{A^\ast}}$ are as in Proposition \ref{prop:gradedcotangentproperties}. Then for every open neighborhood $U$ of $p\in M$, there exists a $C^{2k-1}$-neighborhood $\mc U$ of $d_{A^\ast} \in \mf g_{LbA}^1$ such that for any Lie algebroid structure $Q \in \mc U$ compatible with $d_A$, there is a family $I$ in $U$ of fixed points of order $k$ of $Q$ parametrized by an open neighborhood of $$0 \in H^0(\mf g_{LbA}/\mf g_{LbA}(p,k),\overline{\{\Pi_{d_A},-\}+\{f_{d_{A^\ast}},-\}}).$$
\end{thm}

\subsubsection{Poisson manifolds as Lie bialgebroid structures}\label{sssec:Poissonliebialg}
In this section we apply Theorem \ref{thm:liebialg} to the case where $A = TM$ with its standard Lie algebroid structure $d_A = d_{dR}$. In this case, Lie algebroid structures on $T^\ast M$ compatible with $d_{dR}$ are in bijection with Poisson structures, as was pointed out in \cite[Corollary 5.3]{Roytenberg2002}. \nrt{We therefore obtain a result for higher order singularities of Poisson structures, which we show to be equivalent to \cite[Theorem 1.2]{dufour2005stability} (Corollary \ref{cor:poissonasliebi}).} \nrt{First, we} briefly sketch the correspondence \nrt{between Lie algebroid structures on $T^\ast M$ compatible with the $d_{dR}$ and Poisson structures on $M$.}
\begin{lem}\label{lem:liebialgTM}
Let $A = TM$, with $d_A = d_{dR}$ being the standard Lie algebroid structure on $TM$. Then a Lie algebroid structure $d_{T^\ast M}$ on $T^\ast M$ such that $((TM,d_{dR}),(T^\ast M,d_{T^\ast M}))$ is a Lie bialgebroid structure is equivalent to a Poisson structure on $M$.
\end{lem}

\begin{proof}[Sketch of proof]
Given a Poisson structure $\pi\in \mf X^2(M)$, the usual Lie algebroid structure on $T^\ast M$ is compatible with the de Rham differential. Conversely, given any Lie algebroid structure $d_{T^\ast M}$ on $T^\ast M$ compatible with $d_{dR}$, the map
$$
\pi_{d_{T^\ast M}}:C^\infty(M)\times C^\infty(M) \to C^\infty(M)
$$
given by
$$
\pi_{d_{T^\ast M}}(f,g) := [d_{T^\ast M}(f),g]
$$
is a Poisson structure for which the induced Lie algebroid structure on $T^\ast M$ is $d_{T^\ast M}$. Here the bracket on the right hand side is the usual Schouten-Nijenhuis bracket.
\end{proof}
In particular, this implies that Theorem \ref{thm:liebialg} gives a result for Poisson manifolds. The rest of this section is dedicated to showing that this is in fact equivalent to the result obtained in \cite{dufour2005stability}. For this we show that the vanishing of the cohomologies are equivalent conditions.

Given a Poisson structure $\pi$ on $M$ such that its $(k-1)$-jet at $p\in M$ vanishes, consider the graded Lie subalgebra given by
$$
\bigoplus_{j =1}^nI_p^{1 + j(k-1)}\mf X^j(M),
$$
where $n = \dim (M)$. 
Using the arguments from Section \ref{sec:comparison}, one can show that the cohomological assumption in \cite[Theorem 1.2]{dufour2005stability} can be restated as follows.
\begin{lem}
The vanishing of the cohomology in the hypothesis of \cite[Theorem 1.2]{dufour2005stability} is equivalent to 
$$
H^2( \mf X^\bullet(M)/I_p^{1+(\bullet-1)(k-1)} \mf X^\bullet(M),\overline{[\pi,-]}) = 0.
$$
\end{lem}
The cohomology $H^1(\mf g_{LbA}/\mf g_{LbA}(p,k),\overline{\{\Pi_{d_{dR}},-\}+\{f_{[\pi,-]},-\}})$ from Theorem \ref{thm:liebialg} when applied to $A =TM$ with the standard Lie algebroid structure is isomorphic to the cohomology stated in the lemma:
\begin{prop}
Let $\pi\in \mf X^2(M)$ be as above. The injective differential graded Lie algebra map $H:(\mf X^\bullet(M)[1], [\pi,-]) \to (\mf g^{\bullet}_{LbA},\{\Pi_{d_{dR}},-\} + \{f_{[\pi,-]},-\})$
\[
\begin{tikzcd}
\mf X(M)\arrow{d}{[\pi,-]} \arrow{rr}{ f_{[\bullet,-]}}&& \mf g_{LbA}^0\arrow{d}{\kolomtwee{\{f_{[\pi,-]},-\}} {\{\Pi_{d_{dR}},-\})}}\\
\mf X^2(M) \arrow{dd}{[\pi,-]} \arrow{rr}{ \kolomtwee{f_{[\bullet,-]}} 0}&& \mf g_{LbA}^{(0,1)}\oplus \mf g_{LbA}^{(1,0)}\arrow{dd}{\tweedrie {\{f_{[\pi,-] },-\}} 0  {\{\Pi_{d_{dR}},-\}} {\{f_{[\pi,-]},-\}} 0 {\{\Pi_{d_{dR}},-\}}}\\\\
\mf X^3(M) \arrow{rr}{\kolomdrie{f_{[\bullet,-]}} 0 0}&& \mf g_{LbA}^{(0,2)} \oplus \mf g_{LbA}^{(1,1)} \oplus \mf g_{LbA}^{(2,0)}
  \end{tikzcd}
\]
descends to an injective chain map $$\overline{H}:(\mf X^\bullet(M)[1]/I_p^{1 + (\bullet-1)(k-1)} \mf X(M)[1], \overline{[\pi,-]}) \to (\mf g^{\bullet}_{LbA}/\mf g^{\bullet}_{LbA}(p,k),\overline{\{\Pi_{d_{dR}},-\} + \{f_{d_{T^\ast M}},-\}}):$$
\begin{equation}\label{diag:poisliebicomp}
\begin{tikzcd}
\mf X(M)/I_p\mf X(M)\arrow{d}{\overline{[\pi,-]}} \arrow{r}{\overline{f_{[\bullet,-]}}}& \mf g_{LbA}^0/\mf g_{LbA}^0(p,k)\arrow{d}{\overline{\{f_{[\pi,-]},-\}}}\\
\mf X^2(M)/I_p^k \mf X^2(M) \arrow{dd}{\overline{[\pi,-]}} \arrow{r}{\overline{f_{[\bullet,-]}}}& \mf g_{LbA}^{(0,1)}/\mf g_{LbA}^{(0,1)}(p,k) \arrow{dd}{\kolomtwee{\overline{\{f_{[\pi,-]},-\}}} {\overline{\{\Pi_{d_{dR}},-\}}} }\\\\
\mf X^3(M)/I_p^{2k-1}\mf X^3(M) \arrow{r}{\kolomtwee {\overline{f_{[\bullet,-]}}} 0}& \mf g_{LbA}^{(0,2)}/\mf g_{LbA}^{(0,2)}(p,k) \oplus \mf g_{LbA}^{(1,1)}/\mf g_{LbA}^{(1,1)}(p,k)
\end{tikzcd}.
\end{equation}
 Moreover, the top row of \eqref{diag:poisliebicomp} is an isomorphism and the middle row is an isomorphism when restricted to cocycles. Consequently, the induced map on middle cohomology is an isomorphism.
\end{prop}
\begin{proof}
It is straightforward to see that the map descends to an injective chain map on the quotients. \\
Note that both spaces in the top row of \eqref{diag:poisliebicomp} can be identified with $T_p M$, and that the map is compatible with this identification.

As the middle map is injective and preserves the cocycles, it is sufficient to show that it is surjective on cocycles. Let $f_\delta + \mf g_{LbA}^{(0,1)}(p,k)$ be a cocycle on the right hand side for some $\delta \in \mf X^1(T^\ast M[1])$. Motivated by Lemma \ref{lem:liebialgTM}, define the bivector field on $M$ given by
$$
\pi_\delta(f,g) := \frac{1}{2}([\delta(f),g] - [\delta(g),f])
$$
for $f,g \in C^\infty(M)$.\\
Using the second component of the cocycle condition in \eqref{diag:poisliebicomp} stating that $\{\Pi_{d_{dR}},f_\delta\}\in \mf g^{(1,1)}_{LbA}(p,k)$, it follows that
$$
f_\delta - f_{[\pi_\delta ,-]} \in \mf g_{LbA}^{(0,1)} (p,k),
$$
so the class of $f_\delta$ lies in the image of $\overline{H_1}$.
It remains to show that $\pi_\delta + I_p^k \mf X^2(M)$ is a cocycle. As $$\{f_{[\pi,-]},f_{[\pi_\delta,-]}\} = f_{[[\pi,\pi_\delta],-]} \in \mf g_{LbA}^{(0,2)}(p,k).$$
Using injectivity of $H$, it follows that $[\pi,\pi_\delta] \in I_p^{2k-1}\mf X^3(M)$, concluding the proof.
\end{proof}
We therefore obtain:
\begin{cor}\label{cor:poissonasliebi}
Theorem \ref{thm:liebialg} applied to $(A,d_A) = (TM,d_{dR})$ is equivalent to \cite[Theorem 1.2]{dufour2005stability}.
\end{cor}
\subsubsection{Poisson manifolds with a Poisson hypersurface as Lie bialgebroid structures}\label{sec:bgeometry}
Here we apply Theorem \ref{thm:liebialg} to the case where $A = {}^\flat TM$ \nrt{is a b-tangent bundle.} 
Let $M$ be a smooth manifold, and let $Z\subset M$ be a smooth hypersurface. Denote by $A = {}^\flat TM$ the b-tangent bundle, with its standard Lie algebroid structure $d_A$ (see \cite{GUILLEMIN2014864} for details). Its sections are defined by
$$
\Gamma(A) :=\{X \in \mf X(M) \mid \left. X\right|_{Z} \in \mf X(Z)\},
$$
the anchor is the inclusion, which uniquely determines the bracket. For $M$ a manifold with boundary, the b-tangent bundle for $Z = \partial M$ was introduced in \cite{Melrose1993TheAI}. 

For $A = {}^\flat TM$ with the Lie algebroid structure as described above, we can characterize Lie algebroid structures $d_{A^\ast}$ on $A^\ast$ such that $((A,d_A),(A^\ast,d_{A^\ast}))$ is a Lie bialgebroid explicitly, and give a more direct description of the relevant cohomology. The analogue of Lemma \ref{lem:liebialgTM} holds in this setting, of which we omit the proof.
\begin{lem}
Let $A = {}^\flat TM$ with its standard Lie algebroid structure $d_A$. A Lie algebroid structure $d_{A^\ast}$ on $A^\ast$ such that $((A,d_A), (A^\ast,d_{A^\ast}))$ is a Lie bialgebroid is the same as a self-commuting section $\pi_{d_{A^\ast}} \in \Gamma(\wedge ^2 A)$. Moreover, self-commuting $\pi \in \Gamma(\wedge^2 A)$ are in bijection with Poisson structures on $M$, such that $Z$ is a Poisson submanifold.
\end{lem}

Using this, Theorem \ref{thm:liebialg} implies the following.
\begin{thm}\label{cor:bpoissonasliebi}
Let $(M,Z)$ be a manifold with a given hypersurface $Z$. Let $A = {}^\flat TM$ be the $b$-tangent bundle with its standard Lie algebroid structure. Let $k\geq 1$ be an integer, and let $\pi \in \Gamma(\wedge^2 A)$ be a self-commuting element. Let $p\in M$ be a fixed point of order $k$ of $[\pi,-]$, which is the Lie algebroid structure on $A^\ast$, and assume that
$$
H^1\left(\mf g_{LbA}/\mf g_{LbA}(p,k),\overline{\{\Pi_{d_A},-\}+\{f_{d_{A^\ast}},-\}}\right) = 0.
$$
Then for every neighborhood $U$ of $p$, there is a $C^{2k-1}$-neighborhood $\mc U$ of $\pi$ such that for every self-commuting $\pi'\in \mc U$ there is a family $I$ in $U$ of fixed points of order $k$ of $[\pi',-]$ parametrized by an open neighborhood of $$0\in H^0\left(\mf g_{LbA}/\mf g_{LbA}(p,k),\overline{\{\Pi_{d_A},-\}+\{f_{d_{A^\ast}},-\}}\right).$$
\end{thm}
Next, we spell out what the requirement that some point $p\in M$ is a fixed point of order $k$ of the Lie algebroid structure $[\pi,-]$ means directly in terms of $\pi$, describe the relevant cohomology in more detail for fixed points order 1, and improve the result in this case using Remark \ref{rmk:mainthmrmk}ix).\\
When considering fixed points $p\in M$ of the corresponding Lie algebroid structure on $A^\ast$, we distinguish two types:
\begin{itemize}
    \item[-] $p \in Z$,
    \item[-] $p \in M\setminus Z$.
\end{itemize}
The second case does not yield anything new: as the problem is local, and $\left. A\right|_{M\setminus Z} = T(M\setminus Z)$ as Lie algebroids, we recover the result for Poisson structures of the previous section.\\
The following lemma describes what the notion of fixed point of order $k$ of $[\pi,-]$ implies about $\pi$.
\begin{lem}\label{lem:unravelsing}
Let $\pi \in \Gamma(S^2( A[-1]))$, and $p\in Z\subset M$. Then $[\pi,-] \in \mf g_{LbA}^{(0,1)}(p,k)$ if and only if $\pi \in I_p^k\Gamma(S^2( A[-1]))$. \\
Moreover, in this case the Poisson bracket $\{-,-\}_\pi$ on $C^\infty(M)$ induced by applying $\rho_A$ to $\pi$ satisfies 
$$
\{I_Z,C^\infty(M)\}_\pi \subset I_Z\cdot I_p^k, 
$$
$$
\{C^\infty(M),C^\infty(M)\}_\pi \subset I_p^k.
$$
\end{lem}
\begin{proof}[Sketch of proof]
The proof is a local computation, using that around $p\in M$, there are local coordinates $(x^1,\dots, x^n)$ and a local frame $\{e_1,\dots,e_n\}$ for $A$ such that $Z = \{x^1 = 0\}$, and
$$
\rho(e_i)= \begin{cases} x^1\partial_{x^1} & i = 1,\\ \partial_{x^i} & i \neq 1.\end{cases}.
$$
\end{proof}
For the rest of the section, assume that $p \in Z$ is a fixed point of order $1$ of the Lie algebroid structure $(A^\ast, [\pi,-]_A)$ for some self-commuting $\pi \in \Gamma(S^2(A[-1]))$. In this case, we can give a more explicit description of the cohomology appearing in the theorem. As the problem is local around $p\in Z\subset M$, we assume that $M = \mb R^n$, $p = 0$, and $Z = \{x^1 = 0\}$. Consider the induced frame $\{e_i\}_{i =1}^n$ for $A$ as in the proof of Lemma \ref{lem:unravelsing}, with dual frame $\{e^i\}_{i = 1}^n$. Following Remark \ref{rmk:concrcohom}i) the cochain spaces are given as follows. In degree 0, we have $T_pM$, while in degree 1, we can restrict ourselves to the span of
$$
\{e_1 \otimes \partial_{x^j} \in A_p\otimes T_pM \mid i = 1,\dots, n\} \cup \{e_i\otimes \partial_{x^j}-e_j\otimes \partial_{x^i} \mid 2\leq i,j \leq n\}\subset A_p \otimes T_pM,
$$
by the skew-symmetry requirement. Finally, in degree 2, we have $S^2 (A_p[-1]) \otimes T_pM$. The differentials are given by the  Chevalley-Eilenberg formulas as in Definition \ref{def:cecomplex}.
\begin{exmp}
Let $M = \mb R^3$, with coordinates $(x,y,z)$, and $Z = \{x= 0\}$. Let $$\pi = xe_2\wedge e_3 + ye_3 \wedge e_1 + z e_1 \wedge e_2,$$ \nrt{where the notation is as in the sketch of the proof of Lemma \ref{lem:unravelsing}.} It is easy to see that $[\pi,\pi] = 0$. The corresponding Lie algebroid structure on $A^\ast$ has a fixed point of order 1 in the origin $p$. By the discussion above, to compute the cohomology in degree 1, it is sufficient to restrict ourselves to the subspace of $A_p\otimes T_pM$ generated by $\{e_1\otimes \partial_x,e_1\otimes \partial_y, e_1\otimes \partial_z,e_2\otimes \partial_z - e_3\otimes \partial_y\}$.\\
One can show that the cohomology
$$
H^1\left(\mf g_{LbA}/\mf g_{LbA}(p,k),\overline{\{\Pi_A,-\}+\{f_{d_{A^\ast}},-\}}\right),
$$ as described in Remark \ref{rmk:concrcohom} vanishes, and hence the fixed point is stable.
\end{exmp}
The condition that $$H^1\left(\mf g_{LbA}/\mf g_{LbA}(p,k),\overline{\{\Pi_A,-\}+\{f_{d_{A^\ast}},-\}}\right) = 0$$ is only a sufficient condition, but not a necessary one, as the following example shows. 
\begin{exmp}\label{ex:bpoissonr2}
For $M = \mb R^2$, $Z = \{x= 0\}$, consider $\pi = f(x,y)e_1\wedge e_2$ for some $f \in C^\infty(M)$, with $f(p) = 0$, $df_{p} \neq 0$. The equation $[\pi,\pi] =0$ is trivially satisfied. As $0 \in \mb R$ is a regular value of $f$, we know that $p$ is a stable fixed point of order 1. \\
We now compute the relevant cohomology. By the description above, we can restrict ourselves to the span of $e_1\otimes \partial_x$, $e_1 \otimes \partial_y$.\\
Denote the corresponding differential by $d_\pi$. Then 
$$
d_\pi(e_1 \otimes \partial_x) = f(p)\partial_x = 0,
$$
and
$$
d_\pi(e_1\otimes \partial_y) =\left. -x\partial_xf(x,y)\right|_{(x,y)  = p}\partial_y = 0.
$$
So both elements are cocycles. To see what the coboundaries are, we compute $d_\pi$ of $T_pM$.
$$
d_\pi(\partial_x) = -\partial_xf (p)e_1\otimes \partial_y,
$$
$$
d_\pi(\partial_y) = -\partial_yf(p) e_1\otimes \partial_y,
$$
showing that $e_1\otimes \partial_x$ is never a coboundary, while $e_1\otimes \partial_y$ is always a coboundary. Consequently, the relevant cohomology never vanishes when $\dim M = 2$.
\end{exmp}
More generally, it is true in every dimension that $e_1\otimes \partial_{x^1}$ is never a coboundary. In fact, looking at the proof of Lemma \ref{lem:unravelsing}, we see that when we take any $\pi\in \Gamma(S^2( A[-1]))$, the class of
$$
[\pi,-]\in \mf g_{LbA}^{(0,1)}/\mf g_{LbA}^{(0,1)}(p,1) \cong A_p\otimes T_pM
$$
takes values in the subspace spanned by
\begin{equation}\label{eq:subspacecohom}
K=\{e_i\otimes \partial_j \in A_p \otimes T_pM\mid (i,j) \neq (1,1)\}.
\end{equation}
Note that this is an instance of Remark \ref{rmk:mainthmrmk}ix), hence we can improve the result by putting a milder restriction on the relevant cohomology, by replacing $A_p\otimes T_pM$ by $K$ in the proof of Theorem \ref{thm:liebialg} for $k =1$. Fix coordinates near $p$ as in the proof of Lemma \ref{lem:unravelsing}.
\begin{thm}\label{thm:bpoissonliebi+}
Let $\pi\in \Gamma(S^2 (A[-1]))$, with $[\pi,\pi] =0$. Assume $p\in M$ is such that $\pi_p = 0$. If $$H^1_{red} = 0,$$ as defined in Remark \ref{rmk:mainthmrmk}ix) for $K$ as in equation \eqref{eq:subspacecohom}, the conclusion of Theorem \ref{cor:bpoissonasliebi} holds for $k = 1$.
\end{thm}
\begin{rmk}
Note that Theorem \ref{thm:bpoissonliebi+} does not guarantee the existence of a fixed point in $Z$. We will revisit this in Example \ref{exmp:bpoissondirac}.
\end{rmk}
\subsubsection{Poisson-Nijenhuis structures as Lie bialgebroids}
Another class of Lie bialgebroids comes from Poisson-Nijenhuis structures as shown in \cite{KosmannSchwarzbach1996TheLB}. We will give two results regarding stability of fixed points of Poisson structures, compatible with a given Nijenhuis tensor $N$. Theorem \ref{thm:PNLie} deals with stability for nearby Lie algebroid structures on $T^\ast M$ compatible with $N$ in the Lie bialgebroid sense, while Theorem \ref{thm:PN} deals with nearby Poisson structures which form a Poisson-Nijenhuis pair with $N$.

Classically, Poisson-Nijenhuis structures are the following.
\begin{defn}[\cite{pnstructures}]
Let $M$ be a manifold. A \textit{Poisson-Nijenhuis structure} is a pair $(\pi,N)$, where
\begin{itemize}
    \item[i)] $\pi \in \mf X^2(M)$,
    \item[ii)]$N \in \Gamma(\End(TM))$,
\end{itemize}
satisfying
\begin{itemize}
    \item[a)] $[\pi,\pi] = 0$, 
    \item[b)] $[N,N]_{FN} = 0$,
    \item[c)] $\pi^\sharp\circ N^\ast = N \circ \pi^\sharp$,
    \item[d)] $[\alpha,\beta]_{N\pi} = [N^\ast\alpha,\beta]_\pi + [\alpha,N^\ast \beta]_\pi - N^\ast[\alpha,\beta]_\pi$ for all $\alpha,\beta \in \Omega^1(M)$.
\end{itemize}
\end{defn}
In the definition above, $[-,-]_{FN}$ is the Fr\"olicher-Nijenhuis bracket, and for a bivector field $\pi$, the bracket $[-,-]_{\pi}$ is the one induced on $\Omega^1(M)$. Note that $N\pi$ is the bivector field defined by the equation $(N\pi)^\sharp = N\circ \pi^\sharp$, which is skew-symmetric by c).\\
Note that $N^\ast$ can be extended as a derivation to $\Omega^\bullet(M)$, the graded algebra of all differential forms on $M$. As shown in \cite{pnstructures}, condition b) implies that 
$$
[d,N^\ast] \in \mf X^1(TM[1])
$$
is a Lie algebroid structure on $TM$. Denote this Lie algebroid by $(A,d_N)$. By \cite[Proposition 3.2]{KosmannSchwarzbach1996TheLB}, the compatiblity conditions between $\pi$, $N$ can be expressed as follows.
\begin{prop}[\cite{KosmannSchwarzbach1996TheLB}]
Let $M$ be a manifold, $\pi \in \mf X(M)$ and $N\in \Gamma(\End(TM))$. Then $(\pi,N)$ is a Poisson-Nijenhuis structure if and only if $((A,d_N),(A^\ast,[\pi,-]))$ is a Lie bialgebroid.
\end{prop}
Theorem \ref{thm:liebialg} now yields a stability criterion for fixed points of Lie algebroid structures on $T^\ast M$ near a Poisson-Nijenhuis structure:
\begin{thm}\label{thm:PNLie}
Let $M$ be a manifold equipped with a Poisson-Nijenhuis structure $(\pi,N)$ and $p \in M$ such that the $(k-1)$-jet of $\pi$ vanishes at $p$. Assume that $$H^1\left(\mf g_{LbA}/\mf g_{LbA}(p,k),\overline{\{\Pi_{d_N},-\}+\{f_{[\pi,-]},-\}}\right) = 0,$$ where $\Pi_{d_N}$ corresponds to the Lie algebroid structure $d_N$ on $TM$ induced by $N$. Then for every neighborhood $U \subset M$ of $p$, there is a $C^{2k-1}$-neighborhood $\mc U$ of $[\pi,-]\in \mf X^1(A^\ast[1])$ such that for any Lie algebroid structure $Q\in \mc U$ such that $((A,d_N),(A^\ast,Q))$ is a Lie bialgebroid there is a family $I$ in $U$ of fixed points of order $k$ of $Q$ parametrized by a neighborhood of $$0\in H^0\left(\mf g_{LbA}/\mf g_{LbA}(p,k),\overline{\{\Pi_{d_N},-\}+\{f_{[\pi,-]},-\}}\right).$$
\end{thm}
Note that the conclusion of the theorem for Lie bialgebroids is stronger than the conclusion that all nearby $\pi'$ such that $(\pi',N)$ is a Poisson-Nijenhuis structure must have a fixed point of order $k$ near $p$. For an arbitrary Nijenhuis tensor $N$, there might be Lie algebroid structures near $[\pi,-]$ which do not come from Poisson structures. However, by making different choices for $\mf g$ and $\mf h$ in \nrt{Theorem \ref{thm:mainthm}}, we can obtain a result which deals with the problem of stability within the realm of Poisson-Nijenhuis structures.\\
One way to do this is to consider the Lie subalgebra $ \mf g_{PN}$ of $\mf g_{LbA}$ given by the image of the inclusion
$$
\mf X^{i+1}(M) \hookrightarrow \mf g_{LbA}^{(0,i)}
$$
as in Section \ref{sssec:Poissonliebialg}. For the subalgebra $\mf g_{PN}(p,k)$ corresponding to fixed points, simply restrict to the intersection of $\mf g_{LbA}^{(0,i)}(p,k)$ with the image of the inclusion above. The cochain spaces $\mf g_{PN}/\mf g_{PN}(p,k)$ then take the form 
\begin{align*}
    \mf g_{PN}^0/\mf g^0_{PN}(p,k) &\cong T_pM,\\
    \mf g_{PN}^1/\mf g^1_{PN}(p,k) &\cong J_p^{k-1}(TM)\\
    \mf g_{PN}^2/\mf g^2_{PN}(p,k) &\cong J_p^{2k-2}(TM) \oplus \mf g_{LbA}^{(1,1)}/\mf g_{LbA}^{(1,1)}(p,k).
\end{align*}

Applying Theorem \ref{thm:mainthm} to this data, we find:
\begin{thm}\label{thm:PN}
Let $M$ be a manifold equipped with a Poisson-Nijenhuis structure $(\pi,N)$ and $p \in M$ such that the $(k-1)$-jet of $\pi$ vanishes at $p$. If $$H^1\left(\mf g_{PN}/\mf g_{PN}(p,k),\overline{\{\Pi_{d_N},-\}+\{f_{{[\pi,-]}},-\}}\right) = 0,$$ then for every neighborhood $U \subset M$ of $p$, there is a $C^{2k-1}$-neighborhood $\mc U$ of $\pi\in \mf X^2(M)$ such that for any Poisson structure $\pi'\in \mc U$ such that $(\pi',N)$ is a Poisson-Nijenhuis structure there is a family $I$ in $U$ of fixed points of order $k$ of $\pi'$ parametrized by a neighborhood of $$0\in H^0\left(\mf g_{PN}/\mf g_{PN}(p,k),\overline{\{\Pi_{d_N},-\}+\{f_{{[\pi,-]}},-\}}\right).$$
\end{thm}
\begin{rmk}
\begin{itemize}\item[]
\item[-]
For $k =1$, the cohomology can be described explicitly: unpacking the definition of $$H^1\left( \mf g_{PN}/\mf g_{PN}(p,k), \overline{\{\Pi_{d_N},-\}+\{f_{[\pi,-]},-\}}\right)$$ shows that it is the cohomology of the complex
\[
\begin{tikzcd}
T_pM \arrow{r}{\overline{[\pi,-]}} & S^2 (T_p M[-1]) \arrow{r}{\kolomtwee{\overline{[\pi,-]}}{\overline{d_N}}}&S^3(T_pM[-1]) \oplus S^2(T_pM)\\
\end{tikzcd},
\]
where for $\pi'\in S^2(T_pM[-1])$, 
$$
d_N(\pi') = (N\pi'^\sharp)_{sym}
$$
is the symmetric part of $N\pi'^\sharp$. To compute the middle cohomology, we may then restrict ourselves to those cocycles, for which $N\pi'^\sharp$ is still skew-symmetric.
\item[-]
Note that by construction there is an injective map
$$
H^1\left(\mf g_{PN}/\mf g_{PN}(p,k),\delta \right)\hookrightarrow H^1\left(\mf g_{LbA}/\mf g_{LbA}(p,k),\delta \right),
$$
where $\delta = \overline{\{\Pi_{d_N},-\}+\{f_{[\pi,-]},-\}}$.
\end{itemize}
\end{rmk}
An important example of a Nijenhuis tensor is a complex structure on a manifold $M$: it is a map $J:TM\to TM$ with $J^2 = -\text{id}$, such that $[J,J]_{FN} = 0$. It is shown in \cite[Theorem 2.7]{Holomorphic} that a Poisson-Nijenhuis structure $(\pi,J)$ is equivalent to $J\pi^\sharp + i\pi^\sharp \in \Gamma(\wedge^2 TM_{\mb C})$ being a holomorphic Poisson structure. Applying Theorem \ref{thm:PN} yields:
\begin{cor}\label{cor:holpoisson}
If $H^1\left(\mf g_{PN}/\mf g_{PN}(p,k),\overline{\{\Pi_{d_J},-\}+\{f_{{[\pi,-]}},-\}}\right) = 0$, every holomorphic Poisson structure near $J\pi^\# + i\pi^\#$ has a family of fixed points of order $k$ near $p$, parametrized by a neighborhood of $$0\in H^0\left(\mf g_{PN}/\mf g_{PN}(p,k),\overline{\{\Pi_{d_J},-\}+\{f_{{[\pi,-]}},-\}}\right).$$
\end{cor}
\begin{rmk}
Note that if a holomorphic Poisson structure vanishes up to first order at a point $(k=1)$, the $(1,0)$-cotangent space at $p$ of $M$ inherits a complex Lie algebra structure $\mf g_{\mb C}$. As pointed out by Marius Crainic, the cohomology 
$$
H^1\left(\mf g_{PN}/\mf g_{PN}(p,1), \overline{\{\Pi_{d_J},-\} + \{f_{[\pi,-]},-\}}\right)
$$
is isomorphic to the \emph{complex} Lie algebra cohomology
$$
H^2_{CE}(\mf g_{\mb C},\mb C).
$$
\nrt{In particular, it vanishes if $\mf g$ is semisimple.}
\end{rmk}
\begin{exmp}
We compute an explicit example of Corollary \ref{cor:holpoisson}. Consider $M = \mb R^4 \cong \mb C^2$, with coordinates $(x^1,y^1,x^2,y^2) = (z^1,z^2)$, where $$z^j = x^j + iy^j, 
$$
for $j = 1,2$. On $M$ consider the real Poisson structure 
$$
\pi = y^1(\partial_{x^1}\wedge \partial_{x^2} - \partial_{y^1}\wedge \partial_{y^2}) - x^1(\partial_{x^1}\wedge \partial_{y^2} + \partial_{y^1}\wedge \partial_{x^2}),
$$
and let $J$ be the standard complex structure induced by multiplication by $i$. Then $(\pi,J)$ is a Poisson-Nijenhuis structure ($\pi$ is the imaginary part of the holomorphic bivector field $z^1\partial_{z^1}\wedge\partial_{z^2}$), and $\pi$ vanishes in the origin. In this case the relevant cohomology $$H^1\left(\mf g_{PN}/\mf g_{PN}(p,k),\overline{\{\Pi_A,-\}+\{f_{[\pi,-]},-\}}\right)$$ vanishes, and for nearby Poisson structures $\pi'$ such that $(\pi',J)$ is Poisson-Nijenhuis, there is a $q$ near the origin such that $\pi'$ vanishes in $q$. Note that this is really only the case for those Poisson structures for which $(\pi',J)$ is Poisson-Nijenhuis: the bivector field
$$
\pi_\epsilon = \pi + \epsilon \partial_{x^2}\wedge\partial_{y^2}
$$
is Poisson, but non-vanishing.
\end{exmp}
\subsection{Higher order fixed points of Courant algebroids}
In this section we apply the main theorem to Courant algebroids, and obtain a stability result along the same lines as before.
\subsubsection{Courant algebroids}
Classically, a Courant algebroid is defined as follows. See e.g. \cite{liblmein}.
\begin{defn}
A \textit{Courant algebroid} over a manifold $M$ is a quadruple $(E,\langle-,-\rangle,\rho,[\![-,-]\!])$, where
\begin{itemize}
    \item[i)] $E$ is a vector bundle over $M$,
    \item[ii)] $\langle-,-\rangle:E \times E \to \mb R$ is a fiberwise symmetric, non-degenerate bilinear pairing,
    \item[iii)] $\rho: E\to TM$ is a bundle map,
    \item[iv)] $\cb --:\Gamma(E) \times \Gamma(E) \to \Gamma(E)$ is an $\mb R$-bilinear map
\end{itemize}
satisfying for $x,y,z \in \Gamma(A)$
\begin{itemize}
    \item [a)] $\cb x{\cb yz} = \cb {\cb xy} z + \cb y{\cb xz}$,
    \item [b)] $\rho(x)\pairing yz = \pairing {\cb xy}z +\pairing y{\cb xz}$,
    \item [c)] $\cb xy + \cb yx = \rho^\ast d\pairing xy,$ where $\rho^\ast: T^\ast M \to E^\ast$ is the dual map to $\rho$, and we identify $E^\ast \cong E$ via $\pairing --$.
\end{itemize}
\end{defn}
This definition implies among others the following property, as can be found in e.g. \cite{liblmein}.
\begin{lem}
The bracket $\cb --$ satisfies the Leibniz rule in the right entry, i.e. for $x,y \in \Gamma(E)$, $f \in C^\infty(M)$,
$$
\cb x{fy} = f\cb xy + \rho(x)(f)y.
$$
Consequently, $\rho$ is a morphism of brackets.
\end{lem}
This means that a Courant algebroid gives rise to a foliation on $M$, and it makes sense to speak of fixed points of a Courant algebroid.\\
Due to D. Roytenberg \cite{Roytenberg2002}, there is a more concise definition, making use of graded geometry. \\
Fix a vector bundle $E$ with a non-degenerate, symmetric bilinear pairing $\pairing--$. Then:
\begin{prop}[\cite{Roytenberg2002}]
There is a degree 2 graded manifold $\mc M_{E,\pairing--}$ associated to the pseudo-Euclidean vector bundle $(E,\pairing--)$, which is symplectic: its functions $C^\infty(\mc M_{E,\pairing --})$ are equipped with a \nrt{non-degenerate} degree $-2$ Poisson bracket $\{-,-\}$. Conversely, any symplectic degree 2 graded manifold arises in this way.\\
Moreover, Courant algebroid structures on $(E,\pairing--)$ are in 1-1 correspondence with functions $\Theta \in C^\infty(\mc M_{E,\pairing--})^3$ satisfying $\{\Theta,\Theta\} = 0$.
\end{prop}
\nrt{We will apply Theorem \ref{thm:mainthm} to obtain a stability criterion for higher order fixed points of Courant algebroids: 
\begin{defn}\label{def:calgk}
    Let $(E,\Theta)$ be a Courant algebroid with anchor $\rho$ and bracket $\cb --$. Let $k\geq 1$. A point $p\in M$ is a \emph{fixed point of order $k$} if for $x,y \in \Gamma(E)$, we have 
    $$
    \rho(x)\in I_p^k\mf X(M),\, \cb xy \in I_p^{k-1} \Gamma(E).
    $$
\end{defn}
\begin{rmk}
    The requirement on the bracket is made for the same reason as in the case of Lie algebroids, see Remark \ref{rmk:liealgkwhybracket}.
\end{rmk}
}
\subsubsection{The ingredients}
We check that we are in the setting of Assumptions \ref{ass:mainthmass}. 
\begin{itemize}
\item[i)]
An alternative description of a degree $-2$ Lie bracket is a degree $0$ Lie bracket on the algebra $C^\infty(\mc M_{E,\pairing--})[2]$. This is now a graded Lie algebra \begin{equation} \label{eq:calgliealg}
\mf g_{CA} := C^\infty(\mc M_{E,\pairing--})[2]
\end{equation}
such that $\mf g_{CA}^1 = C^\infty(\mc M_{E,\pairing--})^3$, and its Maurer-Cartan elements are precisely Courant algebroid structures on $E$. Intuitively, elements of $\mf g_{CA}^i = C^\infty(\mc M_{E,\pairing--})^{i+2}$ consist of graded symmetric products of elements of $\Gamma(E)$ and $\mf X(M)$, the former counting as degree $1$, while the latter counts as degree $2$. This can be made precise by picking a connection $\nabla$ on $E$ respecting the pairing. Then we can identify
$$
C^\infty(\mc M_{E,\pairing--}) \cong \Gamma(S(E[-1] \oplus TM[-2])).
$$
Here we implicitly use that $\pairing--$ induces an isomorphism $E\cong E^\ast$.\\
In particular, every degree is given by the sections of some vector bundle.
\item[ii)]
Similar to Lie algebroids, we show that there is a graded Lie subalgebra $\mf g_{CA}(p,k)$, whose Maurer-Cartan elements are those Courant algebroid structures with for which $p\in M$ is a fixed point of order $k$.

\begin{defn}
Let $(E,\pairing--)$ be a vector bundle with a symmetric non-degenerate pairing. Define for $p\in M$, $l \geq 0$:
\begin{align*}
C^\infty(\mc M_{E,\pairing --})^2_{p,l} := \left\{X \in C^\infty(\mc M_{E,\pairing --})^2 \mid \right.& \left.\{X,C^\infty(M)\} \subset I_p^l, \{X,\Gamma(E)\}\subset I_p^{l-1}\Gamma(E)\right\},\\
C^\infty(\mc M_{E,\pairing --})^3_{p,l} := \left\{X\in C^\infty(\mc M_{E,\pairing --})^3\mid \right.& \left.\{X, C^\infty(M)\} \subset I_p^l \Gamma(E), \right.\\& \left. \{X,\Gamma(E)\}\subset C^\infty(\mc M_{E,\pairing --})^2_{p,l}\right\},\\
C^\infty(\mc M_{E,\pairing --})^4_{p,l} := \left\{X\in C^\infty(\mc M_{E,\pairing --})^4 \mid \right.&\left.\{X,C^\infty(M)\} \subset C^\infty(\mc M_{E,\pairing --})^2_{p,l}, \right.\\ &\left.\{X,\Gamma(E)\}\subset C^\infty(\mc M_{E,\pairing --})^3_{p,l-1}\right\}.
\end{align*}
\end{defn}
\begin{lem}\label{lem:calgsubalg}
Setting 
\begin{align*}
    \mf g_{CA}^0(p,k)& := C^\infty(\mc M_{E,\pairing --})^2_{p,1},\\
    \mf g_{CA}^1(p,k)& := C^\infty(\mc M_{E,\pairing --})^3_{p,k},\\
    \mf g_{CA}^2(p,k)& := C^\infty(\mc M_{E,\pairing--})^4_{p,2k},
\end{align*}
these subspaces satisfy
\begin{align*}
\{\mf g_{CA}^0(p,k),\mf g_{CA}^i(p,k)\} &\subset \mf g_{CA}^i(p,k),\\ \{\mf g_{CA}^1(p,k),\mf g_{CA}^1(p,k)\} &\subset \mf g_{CA}^2(p,k).
\end{align*}
\end{lem}
\nrt{We see that the subspace $\mf g_{CA}(p,k) \subset \mf g_{CA}$ is a graded Lie subalgebra. Moreover, its Maurer-Cartan elements are precisely the Courant algebroid structures such that $p$ is a fixed point of order $k$:}
\begin{lem}
Let $(E,\Theta)$ be a Courant algebroid. Let $k \geq 1$. A point $p \in M$ is a fixed point of order $k$ of $\Theta$ if and only if $\Theta \in \mf g_{CA}^1(p,k)$.
\end{lem}
\begin{rmk}
\begin{itemize}
    \item []
    \item[i)] For a Courant algebroid $(E,\Theta)$, the algebra $C^\infty(\mc M_{E,\pairing--})$ has an explicit description in \cite{cueca2020courant} in terms of multilinear maps $C^\bullet(E)$, which first appeared in \cite{kwalgebra}, analogous to the deformation complex of \cite{CrMo08} for Lie algebroids. The Courant algebroid structure $\Theta$ induces a differential on the algebra by $\{\Theta,-\}$, and in terms of the description in \cite{cueca2020courant} this differential \nrt{has a de Rham-type description}.
    \item [ii)] In terms of the complex given in \cite{cueca2020courant}, these subspaces can be described as follows. Omitting the argument $\mc {M}_{E,\pairing--}$, we have for $k =2,3,4$
    \begin{align*}
        (C^\infty)^k_{p,l} \cong \{\omega \in C^k(E) \mid\,\,& \omega(e_1,\dots, e_{k-1},-) \in I_p^{l-1}\Gamma(E^\ast),\\ &\sigma_\omega(e_1,\dots, e_{k-2}) \in I_p^l \mf X(M)\,\, \forall e_1,\dots,e_{k-1} \in \Gamma(E)\}.
    \end{align*}
\end{itemize}
\end{rmk}
\begin{rmk}
As before, we have isomorphisms of $\mf g_{CA}^i/\mf g_{CA}^i(p,k)$ with various jet spaces at $p$. More precisely, after picking a coordinate neighborhood of $p$ such that $(E,\pairing--)$ is trivial as a \emph{pseudo-euclidean} vector bundle, we can make the following identifications.
\begin{align*}
    \mf g_{CA}^0/\mf g_{CA}^0(p,k) &\cong T_pM\\
    \mf g_{CA}^1/\mf g_{CA}^1(p,k) & \cong J^{k-1}_p(E^\ast[-1] \otimes TM) \oplus J^{k-2}_p(S^3(E^\ast[-1]))\\
    \mf g_{CA}^2/\mf g_{CA}^2(p,k) & \cong J^{2k-2}_p(S^2(E^\ast[-1])\otimes TM) \oplus J^{2k-3}_p(S^4(E^\ast[-1])) \oplus J^{2k-1}_p(S^2(TM)).
\end{align*}
\end{rmk}
\begin{rmk}\label{exmp:quadliealg}
We can describe the cohomology $H^1(\mf g_{CA}/\mf g_{CA}(p,1),\overline{\{\Theta,-\}})$ when $p \in M$ is a fixed point (of order 1) of a Courant algebroid $\Theta$ on $(E,\pairing --)$. Analogous to Lie algebroids, the Courant algebroid structure gives rise to a Lie algebra structure on $E_p := \mf g$, and this Lie algebra acts on $T_pM$ by means of a representation with representation map $\tau: \mf g\to \End(T_pM)$. Additionally, $\mf g$ has an invariant non-degenerate pairing $\pairing--{}_p$, and the stabilizer Lie algebras $\mf g_q$ for the representation are coisotropic for all $q\in T_pM$. We can describe the cohomology appearing in Theorem \ref{thm:courant} explicitly in terms of $\mf g$, $\pairing--_p$, $T_pM$ and $\tau$. As we may assume the bundle $E$ is trivial with a constant pairing, we can trivialize the algebra of functions $C^\infty(\mc M_{E,\pairing --})$ canonically, and by unpacking the definition of $H^1\left(\mf g_{CA}/\mf g_{CA}(p,k),\overline{\{\Theta,-\}}\right)$ we obtain the complex given by
\[
\begin{tikzcd}
T_pM \arrow{r}{d^0} & \mf g^\ast[-1]\otimes T_pM \arrow{r}{\kolomtwee {d^1_1}{d^1_2}} & S^2( \mf g^\ast[-1])\otimes T_pM \oplus S^2(T_pM) \otimes (T^\ast_pM \oplus \mb R).
\end{tikzcd}
\]
Here we interpret $J^1_p(S^2(TM)) \cong S^2(T_pM)\otimes (T^\ast_pM \oplus \mb R)$, and view this space as linear maps $T_pM\oplus \mb R \to S^2(T_pM)$.\\
Now $d^0$ and $d^1_1$ are given by the standard Chevalley-Eilenberg formulas, while $d^1_2$ is given as follows: Let $\gamma:\mf g \to T_pM$ and let $\{e_i\}_{i = 1}^r$ be a basis of $\mf g$, and $\{e^i\}_{i = 1}^r$ the dual basis with respect to $\pairing --_p$. Then for $v\in T_pM, \lambda \in \mb R$, we set
$$
d^1_2(\gamma)(v,\lambda) = \sum_{i = 1}^r \gamma(e^i)\cdot\tau(e_i)(v)\in S^2(T_pM),
$$
where $\cdot$ is the symmetric product.\\
$d^1_2$ can be interpreted as follows: First note that $\End(T_pM) \cong T_p^\ast M\otimes T_pM \cong\mf X_{lin}(T_pM)$. In particular, using $\tau$, one can construct the transformation Lie algebroid on the trivial $\mf g$-bundle over $T_pM$. Now if $\gamma:\mf g\to T_pM$ is Chevalley-Eilenberg cocycle, then $\tau + \gamma$, which can be interpreted as an affine action, defines a new Lie algebroid structure on $\mf g\times T_pM$ analogous to \cite[Proposition 4.1]{rigflex}. $d^1_2$ infinitesimally measures the failure of the stabilizers of this Lie algebroid to be coisotropic again, serving as an infinitesimal obstruction to giving rise to a new Courant algebroid structure on $(E,\pairing --)$, as shown in \cite{liblmein}.
\end{rmk}
\item[iii)] As the spaces are once again jet spaces of vector bundles at points, we can pick the necessary splittings.
\item[iv)] Pick a Courant algebroid structure $\Theta$ on $E$ such that $p\in M$ is a fixed point of order $k$.
\end{itemize}
We check that the data satisfies Assumptions \ref{ass:mainthmass}a)-e).
\begin{itemize}
\item[a)] We pick the following topologies:
\begin{itemize}
    \item[-] On $\mf g^0$, we pick the $C^\infty$-topology,
    \item[-] On $\mf g^1$, we pick the $C^{2k-1}$-topology,
    \item[-] On $\mf g^2$, we pick the $C^{2k-1}$-topology.
\end{itemize}
\item[b)] As $\partial = 0$, it is continuous.
\item[c)] The bracket is continuous for the same reason as for Lie algebroids.
\item[d)] To understand the gauge action in $\mf g$, we take a closer look at $\mf g^0$. 
\begin{lem}[\cite{roytenberg1999courant}]
$$\mf g^0 \cong CDO(E,\pairing--),$$ the infinitesimal vector bundle automorphisms of $E$ preserving the pairing.
\end{lem}
One can show that the action is by vector bundle automorphisms, and the analogue of Lemma \ref{lem:singpointmove} holds in this context. This identifies a neighborhood of $0\in \mf g^0_{CA}/\mf g^0_{CA}(p,k)$ with a neighborhood of $p \in M$. 
\item[e)] Lemma \ref{lem:uniqueness} implies that the gauge action preserves Maurer-Cartan elements.

\end{itemize}

\subsubsection{Applying the main theorem}
Let $\mf g_{CA} $ and $\mf g_{CA}(p,k)$ as \nrt{in \eqref{eq:calgliealg} and Lemma \ref{lem:calgsubalg}} respectively. Plugging this into the main theorem yields:
\begin{thm}\label{thm:courant}
Let $(E,\Theta)$ be a Courant algebroid over $M$. Let $p \in M$ be a fixed point of order $k$ for $k \geq 0$, \nrt{as in Definition \ref{def:calgk}}. Assume that $$H^1(\mf g_{CA}/\mf g_{CA}(p,k),\overline{\{\Theta,-\}}) = 0.$$ Then for every open neighborhood $U$ of $p \in M$, there exists a $C^{2k-1}$-neighborhood $\mc U$ of $\Theta \in \mf g_{CA}$ such that for any Courant algebroid structure $\Theta' \in \mc U$ there is a family $I\subset U$ of fixed points of order $k$ of $\Theta$ parametrized by an open neighborhood of $$0 \in H^0(\mf g_{CA}/\mf g_{CA}(p,k),\overline{\{\Theta,-\}}).$$ 
\end{thm}

\begin{rmk} We compare Theorem \ref{thm:courant} with Theorems \ref{thm:lienalgk} and \ref{thm:liebialg}.
\begin{itemize}
\item[-]
There are some parallels between Theorems \ref{thm:liebialg} and \ref{thm:courant}, as $T^\ast[2]A^\ast[1]$ is a symplectic degree 2 manifold, and the function $\Pi_A + f_{d_{A^\ast}}$ induces a Courant algebroid structure. However, the notion of fixed points is different: whereas for a Courant algebroid, a point $p\in M$ is fixed of order $k$ if in particular the anchor has vanishing $k-1$-jet at $p$, for a fixed point $p \in M$ of a Lie algebroid structure on $A^\ast$ only the anchor of $A^\ast$ is required to have vanishing $k-1$-jet at $p$. When $A = TM$ with its standard Lie algebroid structure, this is especially clear: the Courant algebroid structure on $TM \oplus T^\ast M$ has anchor $\text{id} \oplus \pi^\sharp$, which is surjective, while Theorem \ref{thm:liebialg} is about fixed points of $\pi^\sharp$.
\item[-] Given a degree 2 graded manifold, it is non-canonically isomorphic to a Lie 2-algebroid. In this case all one needs to make this identification is a connection on $E$ which is compatible with the pairing. The multibrackets of the corresponding Lie 2-algebroid can be found in \cite{Lie2algandcourant}. In particular, one could also try to use this identification to get a stability result for Courant algebroids, using only Theorem \ref{thm:lienalgk} if the induced brackets satisfy the assumptions of the theorem. In this case the fixed point $p$ of order $k$ for $\Theta$ would be of order $(k,k)$ for the induced Lie $2$-algebroid structure $Q_\Theta$. Note that this makes sense, as for $k\geq 2$, we have $k\leq 2k-2$, and for $k = 1$, there is no restriction on $l$. If one is only interested in nearby Courant algebroids however, Theorem \ref{thm:courant} is actually an improvement as the only independent operations in a Courant algebroid are the anchor and the bracket. The unary and ternary bracket are derived from this, which is reflected in the fact that $H^1(\mf g_{CA}/\mf g_{CA}(p,k),\overline{\{\Theta,-\}})$ only carries information about the anchor and the bracket, while $H^1(\mf g_{LnA}/\mf g_{LnA}(p,(k,k),
\overline{
\{\Theta,-\}})$ also carries information about the unary and ternary bracket.\\
In terms of graded geometry, this can be interpreted as forgetting the symplectic structure of the underlying graded manifold and considering all homological vector fields, rather than just the symplectic (or \nrt{equivalently,} Hamiltonian) ones.
\end{itemize}
\end{rmk}
\subsubsection{Examples}

\begin{exmp}
According to \cite{liblmein}, any quadratic Lie algebra together with a representation on a vector space with coisotropic stabilizer algebras gives rise to a Courant algebroid. In particular, the origin will be a fixed point of order 1. \\
One way to obtain a quadratic Lie algebra with a representation such that its stabilizer algebras are coisotropic is as follows. Let $\mf g$ be any Lie algebra, and $V$ a representation of $\mf g$ with representation map $\rho:\mf g\to \End(V)$. Let $\mf g^\ast$ be the linear of dual of $\mf g$, equipped with the coadjoint representation. Then the semi-direct product $\mf d:= \mf g \ltimes \mf g^\ast$ is a quadratic Lie algebra with respect to the standard pairing, and if we extend $\rho$ to $\mf d$ trivially, this yields a representation of $\mf d$ with coisotropic stabilizers.\\
Following Remark \ref{exmp:quadliealg}, the cohomology
\begin{equation}\label{eq:cacohom}
    H^1(\mf g_{CA}/\mf g_{CA}(p,k),\overline{\{\Theta,-\}})
\end{equation}
can now be described more explicitly, where $\Theta$ is the Courant algebroid structure induced by $\rho$ as in \cite{liblmein}. The coboundaries coincide with the Chevalley-Eilenberg coboundaries for the representation $\rho$ of $\mf g$, because $\mf g^\ast$ acts trivially. The same holds for the cocycles coming from $\mf g^\ast \otimes V$. A map $\gamma:\mf g^\ast \to V$ is a cocycle if and only if the following two conditions are satisfied:
\begin{itemize}
    \item [-] $\gamma$ intertwines the coadjoint action and $\rho$,
    \item [-] For some linear basis $\{e_i\}_{i = 1}^n$ of $\mf g$ with corresponding dual basis $\{e^i\}_{i=1}^n$ of $\mf g^\ast$, the expression
    \begin{align}\label{eq:quadrcocyclecond}
    \sum_{i=1}^n \gamma(e^i)\otimes \rho(e_i)(v) \in \wedge^2 V \subset V\otimes V
    \end{align}
    for all $v\in V$.
\end{itemize}
In particular, the cohomology \eqref{eq:cacohom} vanishes precisely when for any module map $\gamma:\mf g^\ast \to V$, the expression \eqref{eq:quadrcocyclecond} has nonzero symmetric part.
\end{exmp}
\begin{exmp}\label{exmp:simpleliealg}
Let $\mf g$ be a simple Lie algebra, and let $\mf d = \mf g\ltimes \mf g^\ast$ be the semidirect product with its coadjoint representation as above. Let $W$ be any non-trivial irreducible representation of $\mf g$ for which the complexification is an irreducible $\mf g\otimes \mb C$-representation, and let $V = \mf g^\ast \oplus W$. Then we claim that the cohomology as described in the previous example vanishes. By Whitehead's first lemma, we have $H^1_{CE}(\mf g,V)= 0$, so in order to prove that the cohomology \eqref{eq:cacohom} vanishes, we need to show that for any nonzero module map $\gamma:\mf g^\ast \to V$, there exists $v\in V$, such that the expression \eqref{eq:quadrcocyclecond} is not skew-symmetric. For this we have to distinguish two cases:
\begin{itemize}
    \item [-] $W\neq \mf g^\ast$: In this case, by Schur's lemma, any module map $\gamma:\mf g^\ast \to V$ is a multiple of the inclusion. Then for any $v\in W$ which is not invariant under the $\mf g$-action, the expression \eqref{eq:quadrcocyclecond} is not skew-symmetric.
    \item[-] $W = \mf g^\ast$: By Schur's lemma, the only module maps are multiples of the inclusions $\iota_i:\mf g^\ast \to V$ for $i = 1,2$. Let $\lambda,\mu \in \mb R$. Then for $\gamma = \lambda \iota_1 + \mu \iota_2$, $v = (\phi,\psi) \in \mf g^\ast \oplus \mf g^\ast$ such that $\lambda \phi + \mu \psi \neq 0$. Then \eqref{eq:quadrcocyclecond} is not skew-symmetric.
\end{itemize}
\end{exmp}
\subsection{Fixed points of Dirac structures in a split Courant algebroid}\label{sec:dirac}
In this section we look at fixed points of Dirac structures in split Courant algebroids. For an introduction to Dirac geometry we refer to \cite{bursztyndirac}.\\
We assume that we are in the following setting: let $((L,d_L),(L^\ast,d_{L^\ast}))$ be a Lie bialgebroid over $M$. In this case, $L\oplus L^\ast$ can be given a Courant algebroid structure, such that \nrt{$L$ and $L^\ast$} are Dirac structures, \nrt{which we define below in Definition \ref{def:dirac}}. Assume that $p \in M$ is a fixed point of $d_L$, that is, if $\rho_L:L\to TM$ denotes the anchor of $L$, we have $(\rho_L)_p = 0$. When is it the case that for any \emph{Dirac} structure near $L$, there exists a fixed point $q\in M$ near $p$? Of course, Theorems \ref{thm:liealgd1} and \ref{thm:liebialg} could be used for this, as any Dirac structure is in particular a Lie algebroid structure on $L$ compatible with the Lie algebroid structure $d_{L^\ast}$ on $L^\ast$. Nevertheless, we try to apply the main theorem to this setting: we observe that there is a differential graded Lie algebra $\mf g_{Dir}$ such that its Maurer-Cartan elements are precisely Dirac structures near $L$, and a differential graded Lie subalgebra $\mf g_{Dir}(p,1)$ such that its Maurer-Cartan elements are the Dirac structures near $L$ with a fixed point at $p$. The conclusion of the main theorem will be of a different form this time: first, the differential in $\mf g_{Dir}$ need not be inner, i.e. of the form $[\pi,-]$ for some $\pi \in \Gamma(S^2 (L^\ast[-1]))$. Therefore, we cannot get away with using a graded Lie algebra, which implies that the gauge equation is inhomogeneous. Further, the notion of gauge equivalence does not simply move the fixed point around on the entire manifold $M$: it only allows to move the fixed point along the leaf of $(L^\ast,d_{L^\ast})$ through $p$.
\subsubsection{Dirac structures}
We first define Dirac structures in a Courant algebroid $(E,\Theta,\pairing --)$ where the pairing has split signature, \nrt{which} implies that $\rk E = 2n$ for some $n\geq 0$.
\begin{defn}\label{def:dirac}
A \emph{Lagrangian subbundle} $L\subset E$ is a rank $n$ subbundle of $E$ such that
$$
\pairing LL \equiv 0.
$$
A Lagrangian subbundle $L\subset E$ is a \emph{Dirac structure} if 
$$
[\![\Gamma(L),\Gamma(L)]\!] \subset \Gamma(L).
$$
In this case the Courant algebroid structure $\Theta$ of $E$ restricts to a Lie algebroid structure on $L$.
\end{defn}

Let $(E,\Theta,\pairing--)$ be a Courant algebroid, and let $L$ be a Lagrangian subbundle. Assume that there is a Lagrangian subbundle $R \subset E$ such that $L \oplus R  = E$. Then the pairing
$$
\pairing--: L\otimes R \to \mb R
$$
is necessarily non-degenerate, identifying $R\cong L^\ast$. Moreover, under this identification, $$\pairing--:E\otimes E\to \mb R$$ becomes the standard symmetric pairing on $L \oplus L^\ast$ given by
$$
\pairing {(x,\alpha)}{(y,\beta)} = \alpha(y) + \beta(x).
$$
It can be shown that such a Lagrangian complement always exists.\\
Now assume that $L$ is Dirac, and that $R \cong L^\ast$ is also Dirac. The induced Lie algebroid structures $d_L$ and $d_{L^\ast}$ on $L$ and $L^\ast$ respectively now form a Lie bialgebroid $((L,d_L),(L^\ast,d_{L^\ast}))$. Moreover, the Courant bracket of $E$ can be recovered from the Lie algebroid structures of $L$ and $L^\ast$, see \cite{liuweinsteinxu}. 
\begin{rmk}
The existence of a \emph{Dirac} complement to $L$ in $E$ is non-trivial: it can be shown that in $\mb TM^H$ (see \cite[Example 1.2b]{liblmein}), the generalized tangent bundle of $M$ twisted by a closed 3-form $H$, $T^*M$ has a Dirac complement precisely when $H$ is exact: as any Lagrangian complement is necessarily the graph of a 2-form $\omega$, it can be shown that the graph is closed under the Courant bracket if and only if $H = d\omega$. \\
Here, we only discuss the case for when such a Dirac complement exists. The reason for that is that in the general case, the deformations are not governed by a differential graded Lie algebra, but by an $L_\infty$-algebra with a ternary bracket measuring the failure of a Lagrangian complement to be Dirac. In view of Remark \ref{rmk:mainthmrmk}viii), \nrt{we have addressed this in \cite{stablinfty}.}
\end{rmk}
\nrt{We will apply Theorem \ref{thm:mainthm} to obtain a stability criterion for fixed points of Dirac structures:
\begin{defn}
    Let $(E,\Theta,\langle -,-\rangle)$ be a Courant algebroid with anchor $\rho$, and let $L\subset E$ be a Dirac structure. A point $p\in M$ is a \emph{fixed point of $L$} if $(\left. \rho\right|_L)_p = 0.$
\end{defn}
}

\subsubsection{The ingredients}
We check that we are in the setting of Assumptions \ref{ass:mainthmass}.
\begin{itemize}
\item[i)]
We first identify the algebraic framework behind Dirac structures. As before, we consider the Lie bialgebroid $((L,d_L),(L^\ast,d_{L^\ast}))$, with the induced Courant algebroid structure on $E = L\oplus L^\ast$. Note that $L\subset L\oplus L^\ast$ is a Dirac structure, which is transverse to $L^\ast$. Dirac structures close to $L$ will still be transverse to $L^\ast$, hence they can be written as the graph of a bundle map $A:L \to L^\ast$.\\
Now because the pairing restricted to the subbundle $\text{gr}(A):= \{(l,A(l)) \in L\oplus L^\ast \mid l\in L\}$ needs to be identically zero, it follows that $A$ is \emph{skew-symmetric}: it can be interpreted as an element $A\in \Gamma(S^2(L^\ast[-1])) \subset \Gamma(L^\ast[-1] \otimes L^\ast[-1])$. The condition that $\text{gr}(A)$ is Dirac can be written as a Maurer-Cartan equation:
\begin{lem}\cite{liuweinsteinxu}
Let $A\in \Gamma(S^2(L^\ast[-1]))$. Then the graph of $A^\#:L \to L^\ast$ is Dirac if and only if 
$$
d_{L}(A) + \frac{1}{2}[A,A]_{L^\ast} = 0.
$$
\end{lem}
This is the Maurer-Cartan equation in the differential graded Lie algebra
$$
(\Gamma(S( L^\ast[-1]))[1],d_L,[-,-]_{L^\ast}).
$$
Here we use that $((L,d_L),(L^\ast,d_{L^\ast}))$ is a Lie bialgebroid. 
\begin{defn}\label{def:diracliealg}
Let $((L,d_L),(L^\ast,d_{L^\ast}))$ be a Lie bialgebroid. We define
$$
\mf g_{Dir} := \Gamma(S(L^\ast[-1]))[1],
$$
with differential $d_L$, and bracket $[-,-]_{L^\ast}$.
\end{defn}
\item[ii)]
We now identify the graded Lie subalgebra $\mf g_{Dir}(p,1) \subset \mf g_{Dir}$ such that the graph of a Maurer-Cartan element $\pi\in \mf g_{Dir}(p,1)$ is a Dirac structure for which $p\in M$ is a fixed point. \\
Assume that $p$ is a fixed point of $d_L$ and let $A\in \Gamma(S^2(L^\ast[-1]))$. As the anchor of the graph of $A^\#$ is given by $$\rho_L + \rho_{L^\ast} \circ A^\#:L \to TM, $$
where we identify $\text{gr}(A^\#) \cong L$, we see that 
\begin{align*}
\text{$p\in M$ is a fixed point for $\text{gr}(A^\#)$} &\iff \text{im}(A^\#_p)\subset \ker((\rho_{L^\ast})_p) \\&\iff A_p \in S^2( \ker((\rho_{L^\ast})_p)[-1])\subset S^2( L^\ast_p[-1]).
\end{align*}
This hints at how to pick out subspaces of $\mf g_{Dir}$:
\begin{lem}\label{lem:diracsubalg}
Let $i \geq 0$ be an integer. Set
$$\mf g_{Dir}^i(p,1) := \{\Lambda \in \Gamma(S^{i+1}( L^\ast[-1])) \mid \Lambda_p \in S^{i+1}(\ker((\rho_{L^\ast})_p)[-1])\subset S^{i+1} (L^\ast_p[-1])\}.$$ Then $\mf g_{Dir}(p,1)$ is closed under the Lie bracket $[-,-]_{L^\ast}$ and differential $d_L$.
\end{lem}
\begin{proof}
Note that because both $[-,-]_{L^\ast}$ and $d_L$ are derivations of the wedge product, it is sufficient to show that for $X,Y \in \mf g_{Dir}^0(p,1)$, 
$$
[X,Y]_{L^\ast} \in \mf g_{Dir}^0(p,1)
$$
and 
$$
d_L(X) \in \mf g_{Dir}^1(p,1).
$$
The first is easy to show:
$$
\rho_{L^\ast}([X,Y]_{L^\ast})(p) = [\rho_{L^\ast}(X),\rho_{L^\ast}(Y)](p) = 0, 
$$
as the Lie bracket of two vector fields vanishing at $p$ vanishes at $p$. Next, in order to prove the second requirement on $X$, we note that for $\Lambda\in \Gamma(S^2 (L^\ast[-1]))$, we have $$\Lambda \in\mf g_{Dir}^1(p,1) \iff
[\Lambda,C^\infty(M)]_{L^\ast} \subset I_p \Gamma(L^\ast).
$$
Now for $f\in C^\infty(M)$
$$
[d_L(X),f] = d_L(\rho_{L^\ast}(X)(f)) - [X,d_L(f)]_{L^\ast},
$$
as $((L,d_L),(L^\ast,d_{L^\ast}))$ form a Lie bialgebroid.
The first term lies in $I_p\Gamma(L^\ast)$ because $p$ is a fixed point of $d_L$, while the second term lies in $I_p\Gamma(L^\ast)$ because additionally $X \in \mf g_{Dir}^0(p,1)$.
\end{proof}
\nrt{We see that the subalgebra $\mf g_{Dir}(p,1)$ encodes the Dirac structures for which $p\in M$ is a fixed point.}
Now the quotients $\mf g_{Dir}/\mf g_{Dir}(p,1)$ are finite-dimensional vector spaces:
\begin{lem}\label{lem:dirquotspac}
$$ \mf g_{Dir}^i/\mf g_{Dir}^i(p,1) \cong S^{i+1}( L^\ast_p[-1])/S^{i+1}( \ker((\rho_{L^\ast})_p)[-1]).$$
\end{lem}
\begin{proof}
The map \[\Gamma(S^{i+1} (L^\ast[-1])) \to S^{i+1} (L^\ast_p[-1])/S^{i+1}( \ker((\rho_{L^\ast})_p[-1]))\] given by evaluating a section at $p$, and taking the equivalence class mod $S^{i+1}(\ker((\rho_{L^\ast})_p)[-1])$ is surjective with kernel precisely $\mf g_{Dir}^{i+1}(p,1)$.
\end{proof}
\item[iii)]
By restricting to a small neighborhood of $p\in M$, we may again assume that we have splittings $\sigma_i:\mf g_{Dir}^i/\mf g^i_{Dir}(p,1) \to \mf g_{Dir}^i$. 
\item[iv)]As $0\in \Gamma(S^2( L^\ast[-1]))$ represents the Dirac structure $L$, we want to look for Maurer-Cartan elements near $0$. 
\end{itemize}
We check that the data satisfies Assumptions \ref{ass:mainthmass}a)-e).
\begin{itemize}
\item[a)] We pick the following topologies on $\mf g_{Dir}$:
\begin{itemize}
    \item [-]We pick the $C^\infty$-topology on $\mf g_{Dir}^0$,
    \item[-] We pick the $C^1$-topology on $\mf g_{Dir}^1$,
    \item[-] We pick the $C^0$-topology on $\mf g_{Dir}^2$.
\end{itemize}
\item[b)] As the value of $d_L(a)$ depends linearly on the $1$-jet of $a\in \Gamma(S^2(L^\ast[-1]))$, $d_L$ is continuous.
\item[c)] As the value of $[a,b]_{L^\ast}$ depends bilinearly on the $1$-jets of $a,b\in \Gamma(S^2(L^\ast[-1]))$, it is continuous.
\item[d)] An important difference from all the cases considered so far is the gauge action, which we would like to understand in order to interpret the main theorem. In this example, as $d_L$ is not necessarily an inner derivation of the Lie bracket, we have to solve an \emph{inhomogenous} initial value problem. We first give a general formula for the solution to the gauge equation and then interpret it in this case.\\
Let $X\in \Gamma(L^\ast), \pi \in \Gamma(S^2 (L^\ast[-1]))$. Recall the initial value problem we are interested in:
\begin{equation}\label{eq:gaugedirac}
\frac{d}{dt}\pi_t = [X,\pi_t]-d_LX, \pi_0 = \pi
\end{equation}
We give the solution in the following lemma, of which the proof is a computation  (see the text below Lemma \ref{lem:cdodeg0} for the notation):
\begin{prop}\label{prop:dirgaugesol}
Let $X\in \Gamma(L^\ast)$, and let $D= [X,-]_{L^\ast}:\Gamma(L^\ast) \to \Gamma(L^\ast)$ be the covariant differential operator with symbol $\rho_{L^\ast}(X)$, and denote by $\tilde{\Phi}^D_t$ its flow until time $t$ (if it exists), as well as its extension to the shifted symmetric powers of $L^\ast$. Then 
\begin{equation}\label{eq:solutiongaugedirac}
\pi_t := \tilde{\Phi}^D_{-t}(\pi) - \int^t_0 \tilde{\Phi}^D_{-s}(d_L X) \dd s
\end{equation}
satisfies equation \eqref{eq:gaugedirac}.
\end{prop}
Now we are interested in how the anchor of gr$(\pi_t)$ changes with $t$. In particular, if $\pi_1$ exists, how the anchor of gr$(\pi_1)$ compares to the anchor of gr$(\pi_0)$. We give the answer here, but postpone the proof until the appendix.
\begin{lem}\label{lem:anchorsolutiondiracgauge} \nrt{Let $X\in \Gamma(L^\ast), \pi \in \Gamma(S^2(L^\ast[-1]))$. Let $D=[X,-]_{L^\ast} $ be the covariant differential operator with symbol $\rho_{L^\ast}(X)$, and denote by $\tilde{\Phi}_t^D$ its flow until time $t$ (if it exists). Let $(\tilde{\Phi}^D_{t})^\ast: \Gamma(L) \to \Gamma(L)$ be the dual automorphism. Then:}
$$\rho_{\text{gr}(\pi_t)} = (\phi^{\rho_{L^\ast}(X)}_{-t})_\ast\circ \rho_{\text{gr}(\pi_0)} \circ (\tilde{\Phi}^D_{-t})^\ast,$$
where $(\phi_{-t}^{\rho_{L^\ast}(X)})_\ast$ is the pushforward by the time $-t$-flow of $\rho_{L^\ast}(X)$.
\end{lem}
As $\mf g_{Dir}^0/\mf g_{Dir}^0(p,1)  = L^\ast_p/\ker((\rho_{L^\ast})_p) \cong T_pS$,
where $S$ is the leaf of $L^\ast$ going through $p \in M$, and by the Lie algebroid splitting theorem \cite[Theorem 1.1]{LOJAFERNANDES2002119}, we see that we can pick splittings such that the gauge action exists for all elements in the image of $\sigma_0$. In particular, Lemma \ref{lem:anchorsolutiondiracgauge} implies that
$$
\pi^{\sigma_0(v)}\in \mf g_{Dir}^1(p,1) \iff \phi_{1}^{\rho_{L^\ast}(\sigma_0(v))}(p) \text{ is a fixed point of gr$(\pi_0)$},
$$
and a neighborhood of $0 \in \mf g_{Dir}^0/\mf g_{Dir}^0(p,1)$ corresponds to a neighborhood of $p\in S$, using the gauge action.
\begin{rmk}
\nrt{There is a more geometric argument to interpret the gauge action in \cite{stablinfty}: the element $X\in \Gamma(L^\ast)$, when interpreted as an element of $L\oplus L^\ast$ gives rise to a covariant differential operator $\cb X-$  on $L\oplus L^\ast$ with symbol $\rho_{L^\ast}(X)$. The time-1 flow of this operator transforms the graph of $\pi_0$ into the graph of $\pi_1$.}
\end{rmk}
\item[e)] Lemma \ref{lem:uniqueness} implies that the gauge action preserves Maurer-Cartan elements.
\end{itemize}
\subsubsection{Applying the main theorem}
Applying the main theorem to $\mf g_{Dir}$ and $\mf g_{Dir}(p,1)$ \nrt{as in Definition \ref{def:diracliealg} and Lemma \ref{lem:diracsubalg} respectively} now yields:
\begin{thm}\label{thm:Dirac}
Let $p \in M$ be a fixed point of the Dirac structure $L$ inside the Courant algebroid $L\oplus L^\ast$, that is, $(\rho_L)_p = 0$. Assume that $$H^1(\mf g_{Dir}/\mf g_{Dir}(p,1),\overline{d_L}) = 0.$$ Then for every open neighborhood $U$ of $p\in S$ in $S$, where $S$ is the $(L^\ast,d_{L^\ast})$-leaf through $p$, there exists a $C^{1}$-neighborhood $\mc U$ of $0 \in \mf g_{Dir}^1$ such that for any Dirac structure $\pi \in \mc U$ there is a family $I$ in $U$ of fixed points of the Dirac structure $\text{gr}(\pi)$ parametrized by an open neighborhood of $$0 \in H^0(\mf g_{Dir}/\mf g_{Dir}(p,1),\overline{d_L}).$$
\end{thm}
\subsubsection{Examples}
\begin{exmp}[Poisson structures]
The first example we apply this to is the one of Poisson structures: for $M$ a smooth manifold, $\mb TM = TM \oplus T^\ast M$ has a Courant algebroid structure, which \nrt{arises from} the construction of Section \ref{sec:liebialg} applied to the standard Lie algebroid structure on $TM$, and the zero Lie algebroid structure on $T^\ast M$. Now let $\pi \in \Gamma(S^2(TM[-1]))$ be a Poisson structure, that is, $$
[\pi,\pi]=0.
$$
Then the graph of $\pi^\#$ is a Dirac structure, whose fixed points are exactly the zeroes of $\pi$. Now let $L = T^\ast M$, and consider the Dirac structure given by the graph of $\pi$. Assume that $p\in M$ is a fixed point of \text{gr}$(\pi)$. Then $d_L:\Gamma(S^\bullet( TM[-1])) \to \Gamma(S^{\bullet+1}( TM[-1]))$ is given by $d_L = [\pi,-]$. Now $\mf g_{Dir}= \mf X^\bullet(M)$, and $\mf g_{Dir}(p,1) = I_p\mf X^\bullet(M)$. Hence the relevant complex is given by
\[
\begin{tikzcd}
T_pM \arrow{r}{\overline{[\pi,-]}} & S^2 (T_p M[-1]) \arrow{r}{\overline{[\pi,-]}} & S^3( T_p M[-1])
\end{tikzcd}
\]
as in Lemma \ref{lem:dirquotspac}.\\
This is precisely the complex appearing in \cite{CrFe} for zero-dimensional leaves and \cite{dufour2005stability} for $k =1$.\\
Finally, the conclusion of the theorems is the same as well. We see that Theorem \ref{thm:Dirac} recovers the above-mentioned results.
\end{exmp}

\begin{exmp}[${}^\flat$-Poisson structures]\label{exmp:bpoissondirac}
We now look at a slight variation of this. Let $M$ be a smooth manifold and $Z\subset M$ a smooth connected hypersurface. We then consider the Lie algebroid $\null^\flat TM$, with its standard Lie algebroid structure given by the inclusion $\Gamma(\null^\flat TM) \subset \mf X(M)$. Now let $\pi \in \Gamma(\wedge^2 \null^\flat TM)$ be a self-commuting element. This is a Poisson structure for which $Z$ is a Poisson hypersurface, and the graph of $\pi^\#:\null^\flat T^\ast M \to \null^\flat TM$ is a Dirac structure inside $\null^\flat TM \oplus \null^\flat T^\ast M$. Now assume that $p \in M$ is a fixed point of this Dirac structure. One can show that this is is equivalent to $\pi_p =0 \in \wedge^2 \null^\flat T_pM$. We apply Theorem \ref{thm:Dirac} to this example. Note that $\null^\flat TM$ has two kinds of leaves: there is the hypersurface $Z$, and the connected components of $M\setminus Z$. As the latter leaves are open and the stability problem is local, $p\in M\setminus Z$ puts us in the ordinary Poisson case. We therefore consider $p\in Z$. In this case $\mf g_{Dir} = \Gamma(S(\null ^\flat TM[-1]))$, and as $\ker((\rho_{\null^\flat TM})_p)$ is one-dimensional, we find that the relevant complex is given by
\begin{equation}\label{diag:bdirac}
\begin{tikzcd}
T_pZ \arrow{r}{\overline{[\pi,-]}} & S^2(\null^\flat T_pM[-1]) \arrow{r}{\overline{[\pi,-]}} & S^3 (\null^\flat T_pM[-1])
\end{tikzcd}
\end{equation}
as in Lemma \ref{lem:dirquotspac}.\\
The conclusion of the theorem tells us now that if the middle cohomology vanishes, we find that for every Poisson structure near $\pi$ such that $Z$ is a hypersurface, there is a family of zeroes near $p$ inside $Z$. A natural question is how this compares to the question of fixed points of Poisson structures on $Z$, starting with the Poisson structure $\left.\pi\right|_Z$. In this case the relevant complex is given by
\begin{equation}\label{diag:hypersurfacepoisson}
\begin{tikzcd}
T_pZ \arrow{r}{\overline{[\left.\pi \right|_Z,-]}}& S^2( T_p Z[-1]) \arrow{r}{\overline{[\left.\pi\right|_Z,-]}}& S^3 (T_p Z[-1]).
\end{tikzcd}
\end{equation}
There is a surjective chain map from \eqref{diag:bdirac} to \eqref{diag:hypersurfacepoisson} which induces a surjection on the middle cohomology. This reflects the fact that any Poisson structure on $Z$ is (locally around $p\in M$) the restriction of a Poisson structure on $M$ for which $Z$ is a Poisson hypersurface.
\end{exmp}

\section*{Appendix}
\appendix
\section{Auxiliary lemmas}
The following lemmas are a small variation of Proposition 4.4 of \cite{crainic2013survey}, and are used in the proof of the main theorem of \cite{dufour2005stability}, as well as Theorem \ref{thm:mainthm}.
\begin{lem}\label{lem:transverse}
Let $V,W$ be finite-dimensional real vector spaces and $B \subset W$ a linear subspace of codimension $r$. Assume that $f:V\to W$ is a smooth map such that $f(0) = 0$, and that $(Df)_0(V) \oplus B = W$. Then for every neighborhood $U$ of $0\in V$, there exists a $C^0$-neighborhood $\mc U$ of $f$ in $C^\infty(V,W)$ such that for every $g \in \mc U$ there exists $q\in U$ with the property $g(q) \in B$. \\
Moreover, there exists a $C^1$-neighborhood $\mc U'$ of $f$ such that $g \in \mc U'$ in addition also satisfies that $(Dg)_q(V) \oplus B = W$.\\
Finally, in the latter case, the set $g^{-1}(B)$ is a smooth submanifold near $q\in V$ of dimension $\dim_{\mb R}(\ker (Df)_0)$.
\end{lem}
\begin{proof}
Let $A\subset W$ be a complement to $B$, and decompose $f: V \to A \oplus B$ as $(f_A,f_B)$. For the first statement, it suffices to show that for every neighborhood $U$ of $0\in V$, there is a $C^0$-neighborhood $\mc U$ of $f_A$ in $C^\infty(V,A)$ such that for every $g\in \mc U$, there exists $q \in U$ such that $g(q) = 0 \in A$.\\
Observe that since $f_A:V\to A$ is a submersion at $0\in V$, up to diffeomorphism we may assume that $ V = A\times P$, $U= U_1 \times U_2$ and that $f_A$ is the projection onto $A$. Picking a basis for $A$, let $\epsilon >0$ be small enough such that $[-\epsilon,\epsilon]^r \subset U_1$. Let $$\mc U = \left\{ g\in C^\infty(V,A) \mid \lVert f_A-g \rVert_{[-\epsilon, \epsilon]^r\times \{0_P\},0}<\frac{\epsilon}{2}\right\} $$ be the $\frac{\epsilon}{2}$-ball around $f_A$ with respect to the $C^0$-seminorm associated to the compact set $K = [-\epsilon,\epsilon]^r \times \{0_P\}$. Now if $g\in \mc U$, then 
$$
\left. g\right|_{A\times \{0_P\}}: \mathbb{R}^r \times \{0_P\} \to \mathbb{R}^r, (x_1,\dots, x_r) \mapsto (g_1(x_1,\dots,x_r),\dots, g_r(x_1,\dots,x_r))
$$
satisfies
$$
g_i(x_1,\dots, x_{i-1},-\epsilon, x_{i+1},\dots x_r) < \frac{\epsilon}{2}+ f_i(x_1,\dots, x_{i-1},-\epsilon, x_{i+1},\dots x_r) = -\frac{\epsilon}{2} <0,
$$
$$
g_i(x_1,\dots, x_{i-1},\epsilon, x_{i+1},\dots x_r) > -\frac{\epsilon}{2}+ f_i(x_1,\dots, x_{i-1},\epsilon, x_{i+1},\dots x_r) = \frac{\epsilon}{2} >0.
$$
By the Poincar\'e-Miranda theorem \cite{pmthm}, there exists a point $q \in \mathbb{R}^r \cong A$ such that $g(q,0_P) = 0$. \\
For the second statement, once we know the existence of $q$, the derivative of $f_A$ at $q$ is surjective (it is still the projection). As this is an open condition and $K$ is compact, the result follows. \\
The final statement follows from the preimage theorem, as $g$ intersects $B$ transversely in $q\in V$.
\end{proof}
\begin{rmk}
\begin{itemize}\item[]
\item[-] Note that we can replace $B$ by an open subset of $B$, by restricting the obtained neighborhood to only those maps which take values in the open subset.
\item[-] By using the implicit function theorem instead of the preimage theorem, $\ker(Df)_0$ can be used as a local chart for $g\inv(B)$.
\end{itemize}
\end{rmk}
\begin{lem}[Lemma A on p. 61 of \cite{Golubitsky1973}]\label{lem:immersion}
Let $f:X \to Y$ be a smooth immersion at a point $p \in X$. Then there is an open neighborhood $U$ of $p \in X$, and a $C^1$-neighborhood $\mc U$ of $f$ such that every $g \in \mc U$ is an injective immersion on $U$.
\end{lem}

\section{Omitted proofs}
In this section we prove two statements whose proofs were omitted in the main text.
We start with a proof of Lemma \ref{lem:resexact}, which is about pulling back certain geometric resolutions.
\begin{proof}[Proof of Lemma \ref{lem:resexact}]
For the first statement, we construct the resolution as follows: pick linear generators $X_1,\dots, X_r$ of $\mc F(V)$, and let $\mc F^{pol}:= \langle X_1,\dots, X_r\rangle_{S(V^\ast)}$, the $S(V^\ast)$-submodule of $\mc F(V)$ consisting of \emph{polynomial} linear combinations of the generators. This is a finitely generated module over the ring of polynomials, so by the Hilbert Syzygy theorem, there exist free modules $F_i = S(V^\ast)\otimes_{\mb R} E_i$, where $E_i$ is a finite-dimensional vector space, and module maps $d_i:F_i\to F_{i-1}$ such that
\begin{equation}\label{diag:polres}
\begin{tikzcd}
0\arrow{r} &F_n \arrow{r}{d_n}& F_{n-1} \arrow{r}{d_{n-1}} &\dots \arrow{r}{d_2} & F_1 \arrow{r}{\rho} & \mc F^{pol} \arrow{r} & 0
\end{tikzcd}
\end{equation}
is exact. Extending this sequence above as sheaf morphisms, we obtain a complex of $C^\infty_V$-modules. As a sequence of sheaves is exact precisely when the sequence is exact on every stalk, we need to show that this implies that for every $q\in V$, the sequence 
\[
\begin{tikzcd}
0\arrow{r}& \Gamma_q(E_n) \arrow{r}{d_n}& \Gamma_q(E_{n-1})\arrow{r}{d_{n-1}} & \dots \arrow{r}{d_2} &\Gamma_q(E_1) \arrow{r}{\rho}& \mc F_q \arrow{r}& 0
\end{tikzcd}
\]
is exact, where $\Gamma_q(E)$ is the stalk of the sheaf of sections of $E$ at $q$. Note that $$\Gamma_q(E_i) = C^\infty_{V,q} \otimes_{ S(V^\ast)} F_i  = C^\infty_{V,q}\otimes_{C^\omega_{V,q}} C^\omega_{V,q} \otimes_{S(V^\ast)} F_i,$$ where $C^\omega_V$ is the sheaf of analytic functions on $V$, and $C^\omega_{V,q}$ is an $S(V^\ast)$-module by means of the inclusion of polynomials as analytic functions. As 
$$
\mc F_q = C^\infty_{V,q} \otimes_{S(V^\ast)} \mc F^{pol}, 
$$
it suffices to show that $C^\infty_{V,q}$ is flat over $S(V^\ast)$. We do this in two steps: we show that germs of analytic functions in a point $q$ are flat over polynomials, and that the germs of smooth functions in a point $q$ are flat over germs of analytic functions in $q$.\\
The second step is just \cite[Corollary VI.1.12]{Malgrange1966IdealsOD}. For the first step, note that we have a commutative triangle,
\[
\begin{tikzcd}
S(V^\ast)\arrow{d} \arrow{r}{T_q}& \hat{S}(V^\ast)\\
C^\omega_{V,q} \arrow{ur}{T_q}
\end{tikzcd}
\]
where $\hat{S}(V^\ast)$ is the algebra of formal power series on $V$, $T_q$ takes the Taylor expansion around $q$, and the vertical map is the inclusion. Now we note that $T_q$ is actually the $I_q$-adic completion map with respect to the vanishing ideal of $q\in V$, which for $C^\omega_{V,q}$ is faithfully flat by \cite{anacompflat}. As the ring of polynomials is not local, we cannot apply this argument to the horizontal map. However, we can use \cite{polcompflat} to conclude that the horizontal map is flat. It then follows from the definition of (faithful) flatness that the vertical map is also flat.\\
For the second statement, we first find a resolution for $C^\infty_{V\times U}\otimes_{C^\infty_V} \mc F$. By the assumption that $\mc F_U$ is an isomodular deformation, this gives a resolution of $\mc F_U$.\\ The straightforward thing to do here would be to take the resolution we found earlier, and pull it back to $V\times U$. As far as we know, it is not clear whether this preserves exactness. We therefore take a different way. Consider again the resolution \eqref{diag:polres}. If we apply the functor $S(V^\ast \oplus \mb R^n) \otimes_{S(V^\ast)} -$, the sequence stays exact. If we now apply the same argument as before, we obtain a geometric resolution of $C^\infty_{V\times U}\otimes_{C^\infty_V} \mc F$ as in the formulation of the lemma. As the differentials are unchanged, the first two properties are automatic. Finally, the fact that the Lie $n$-algebroid restricts is a consequence of the fact that $\mc F_U$ is tangent to $V\times \{0\}$. Now the existence of the Lie $n$-algebroid structure is guaranteed by \cite[Theorem 2.7]{univlinfty}.
\end{proof}
We now prove Lemma \ref{lem:anchorsolutiondiracgauge} about the anchor of a gauge transformed Dirac structure.
\begin{proof}[Proof of Lemma \ref{lem:anchorsolutiondiracgauge}]
The non-trivial part of the proof is evaluating the integral of equation \eqref{eq:solutiongaugedirac}. For that we give the antiderivative of $\rho_{L^\ast} \circ \tilde{\Phi}_{-s}^D(d_L X)^\#$:
\begin{claim}
$$
\frac{d}{ds}(\phi_{-s}^{\rho_{L^\ast}(X)})_{\ast} \circ \rho_L \circ (\tilde{\Phi}^D_{-s})^\ast = -\rho_{L^\ast} \circ \tilde{\Phi}_{-s}^D(d_L X)^\#.
$$
\end{claim}
\begin{proof}[Proof of claim]
It suffices to prove that the claim holds when evaluating both sides on a section $a\in \Gamma(L)$.
\begin{align*}
    \frac{d}{ds}(\phi_{-s}^{\rho_{L^\ast}(X)})_{\ast} \circ \rho_L \circ (\tilde{\Phi}^D_{-s})^\ast(a)=& [\rho_{L^\ast}(X),(\phi_{-s}^{\rho_{L^\ast}(X)})_{\ast} \circ \rho_L \circ (\tilde{\Phi}^D_{-s})^\ast(a)]\\&- (\phi_{-s}^{\rho_{L^\ast}(X)})_{\ast} \circ \rho_L \circ \mc L_X((\tilde{\Phi}^D_{-s})^\ast(a))\\
   =&(\phi_{-s}^{\rho_{L^\ast}(X)})_{\ast} [\rho_{L^\ast}(X),\rho_L ((\tilde{\Phi}^D_{-s})^\ast(a))]\\&- (\phi_{-s}^{\rho_{L^\ast}(X)})_{\ast} (\rho_L ( \mc L_X((\tilde{\Phi}^D_{-s})^\ast(a))))\\
    \stackrel{\triangle}{=}& (\phi_{-s}^{\rho_{L^\ast}(X)})_\ast(\rho_L(\mc L_X(\tilde{\Phi}^D_{-s})^\ast (a))) \\&- (\phi_{-s}^{\rho_{L^\ast}(X)})_\ast(\rho_{L^\ast}(\mc L_{(\tilde{\Phi}_{-s}^D)^\ast(a)}(X)))\\
    & + (\phi_{-s}^{\rho_{L^\ast}(X)})_\ast(\rho_{L^\ast}(d_L(\iota_{(\tilde{\Phi}_{-s}^D)^\ast(a)}(X)))\\& - (\phi_{-s}^{\rho_{L^\ast}(X)})_{\ast} (\rho_L ( \mc L_X((\tilde{\Phi}^D_{-s})^\ast(a))))\\
    =& - (\phi^{\rho_{L^\ast}(X)}_{-s})_\ast(\rho_{L^\ast}(\iota_{(\tilde{\Phi}_{-s}^D)^\ast(a)}(d_LX)))\\
    \stackrel{\star}{=}& -\rho_{L^\ast}(\tilde{\Phi}_{-s}^D(\iota_{(\tilde{\Phi}_{-s}^D)^\ast(a)}(d_LX)))\\
    =& - \rho_{L^\ast}(\tilde{\Phi}_{-s}^D(d_LX)^\#(a)).
    \end{align*}
    Here $\mc L_X  = d_{L^\ast}\iota_X + \iota_X d_{L^\ast}$, and an analogous formula holds for $\mc L_{(\tilde{\Phi}^D_{-s})^\ast(a)}$. Further, at $\triangle$, we apply \cite[Lemma 4.3]{liuweinsteinxu}, and at $\star$ we apply the equality
    $$
    \rho_{L^\ast} \circ \tilde{\Phi}_{-s}^D = (\phi_{-s}^{\rho_{L^\ast}(X)})_\ast \circ \rho_{L^\ast},
    $$
    which holds as both sides of the equation satisfy the ODE
    $$
    \frac{d}{ds}\gamma_s(Y) = [\rho_{L^\ast}(X),\gamma_s(Y)]
    $$
    for all $Y \in \Gamma(L^\ast)$, with $\gamma_0 = \rho_{L^\ast}$.
\end{proof}
To finish the proof, recall from Proposition \ref{prop:dirgaugesol} that
\begin{equation}\label{eq:dirgaugesol1}
\pi_1 = \tilde{\Phi}_{-1}^D(\pi_0) -\int_{0}^1\tilde{\Phi}_{-s}^D(d_L X)\dd s.
\end{equation}
Then
\begin{align*}
    \rho_{\text{gr}(\pi_1)} &= \rho_L + \rho_{L^\ast}\circ \pi_1^\#\\
    &\stackrel{\star}{=} \rho_{L} + \rho_{L^\ast} \circ \tilde{\Phi}^D_{-1}(\pi_0)^\# + (\phi_{-1}^{\rho_{L^\ast}(X)})_\ast \circ \rho_{L} \circ (\tilde{\Phi}^D_{-1})^\ast - \rho_L\\
    &= \rho_{L^\ast} \circ \tilde{\Phi}_{-1}^{D} \circ \pi_0^\# \circ (\tilde{\Phi}_{-1}^D)^\ast + (\phi_{-1}^{\rho_{L^\ast}(X)})_\ast \circ \rho_{L} \circ (\tilde{\Phi}^D_{-1})^\ast\\
    &= (\phi_{-1}^{\rho_{L^\ast}(X)})_\ast \circ (\rho_{L^\ast} \circ \pi_0^\# + \rho_L) \circ (\tilde{\Phi}^D_{-1})^\ast\\
    &= (\phi_{-1}^{\rho_{L^\ast}(X)})_\ast \circ \rho_{\text{gr}(\pi_0)} \circ (\tilde{\Phi}^D_{-1})^\ast,
\end{align*}
where $\star$ uses the claim and \eqref{eq:dirgaugesol1}.
\end{proof}
\bibliography{bib.bib}{}
\bibliographystyle{abbrv}
\Addresses
\end{document}